\documentclass[11pt,a4paper]{article}

\usepackage{fullpage}
\usepackage[T1]{fontenc} 
\usepackage{amsmath,amssymb}
\usepackage{amsopn,amsthm}
\usepackage{subcaption}

\usepackage{graphicx}
\usepackage{color}
\usepackage{psfrag}

\usepackage{tikz}
\usepackage{paralist}

\usepackage{mathrsfs} 
\usepackage{bm}
\usepackage{cases}

\usepackage{algpseudocode} 
\usepackage{algorithm}
\floatstyle{plain} \newfloat{myalgo}{tbhp}{mya}

\usepackage[a]{esvect}
\usepackage{float}

\usepackage[font=small,labelfont=bf]{caption}

\usepackage{comment}
\usepackage[normalem]{ulem}

\usepackage{enumitem}

{\begin{myalgo}[#1]
    \centering
    \begin{minipage}{0.9\textwidth}
      \begin{algorithm}[H]}%
      {\end{algorithm}
    \end{minipage}
  \end{myalgo}}

\usepackage[colorlinks=true]{hyperref}
\hypersetup{urlcolor=blue, citecolor=blue, linkcolor=blue}

\newtheorem{theorem}{Theorem}[section]
\newtheorem{proposition}[theorem]{Proposition}

\theoremstyle{definition}
\newtheorem{definition}[theorem]{Definition}
\newtheorem{example}[theorem]{Example}

\theoremstyle{remark} \newtheorem{remark}[theorem]{Remark}

\numberwithin{equation}{section}
\numberwithin{figure}{section}
\numberwithin{algorithm}{section}
\numberwithin{table}{section}

\newcommand{\field}[1]{{\mathbb{#1}}}
\newcommand{\C}{\field{C}}
 
\newcommand{\N}{\field{N}}
\newcommand{\R}{\field{R}}
\newcommand{\Z}{\field{Z}}

\newcommand{\bs}{\boldsymbol} 
\newcommand{\bfa}{{\bs a}}

\newcommand{\bfc}{{\bs c}} 
\newcommand{\bfd}{{\bs d}}
\newcommand{\bfe}{{\bs e}}

\newcommand{\bfx}{{\bs x}}
\newcommand{\bfy}{{\bs y}}

\newcommand{\bfzero}{{\bs 0}}

\newcommand{\bfF}{{\bs F}}

\newcommand{\bfI}{{\bs I}}

\newcommand{\bfQ}{{\bs Q}}

\newcommand{\bfW}{{\bs W}}

\newcommand{\xhat}{\widehat{\bfx}}
\newcommand{\yhat}{\widehat{\bfy}}

\newcommand{\HS}{\mathrm{HS}}
\newcommand{\loc}{{\mathrm{loc}}}

\newcommand{\ra}{\rightarrow}

\newcommand{\di}{\partial}

\newcommand{\dr}{\, \dif r}
\newcommand{\ds}{\, \dif s}

\newcommand{\dy}{\, \dif \bfy}

\newcommand{\idir}{\bfd} 

\newcommand{\rmi}{\mathrm{i}} 
\newcommand{\rme}{\mathrm{e}}

\newcommand{\ph}{\,\cdot\,}

\newcommand{\uinfty}{u^\infty}

\newcommand{\Sone}{{S^1}}

\newcommand{\Sd}{{S^1}}
\newcommand{\Rd}{{\R^2}}

\definecolor{ForestGreen}{RGB}{34,139,34}

\newcommand{\ui}{u^i}
\newcommand{\us}{u^s}

\newcommand{\Ltwo}{{L^2}}
\newcommand{\Linfty}{{L^{\infty}}}

\newcommand{\LtwoSd}{{L^2(\Sd)}}

\DeclareMathAlphabet{\mathbi}{\encodingdefault}{\rmdefault}{\bfdefault}{\itdefault}

\DeclareMathOperator{\dif}{d\!}

\DeclareMathOperator{\diag}{diag}

\begin{document}

\title{Direct reconstruction for acoustic inverse Born scattering}  
\author{Nuutti Hyvönen and Lisa Sch\"atzle\footnote{Department of Mathematics and System Analysis, Aalto University, 00076 Helsinki, Finland ({\tt nuutti.hyvonen@aalto.fi, lisa.schatzle@aalto.fi}).}
}
\date{\today}

\maketitle

\begin{abstract}
We consider the inverse medium scattering problem for the Helmholtz equation in two dimensions, i.e., the task to recover a compactly supported penetrable two-dimensional scatterer from full knowledge of the associated far field data or, equivalently, the far field operator.
Although this problem is uniquely solvable, it is severely ill-posed and nonlinear.
In the regime of weak scattering, the Born approximation yields a linearized relation between the contrast and the far field data, thus overcoming the second difficulty. This linear setting allows to build on recent work on linearized electrical impedance tomography, which relies on triangular Zernike decompositions, to derive an explicit reconstruction formula that expresses the expansion coefficients of the contrast in terms of those of the far field data.
By choosing the expansion functions appropriately, the resulting system matrix decouples into separate (infinite) triangular systems for the spatial angular frequencies in the contrast.
Consequently, each of these systems can be solved independently by performing forward substitutions.
Our numerical experiments indicate that this approach, combined with an adequate regularization method, remains effective even when applied to full nonlinear far field data beyond the Born regime.
\end{abstract}

{\small\noindent
  Mathematics subject classifications (MSC2010):
  35R30,  
  78A46, 
  47A52 
  \\\noindent
  Short title: Inverse Born scattering
}

\section*{Introduction}

In inverse scattering, one seeks for an unknown material parameter, described by a contrast function, from associated scattered near or far field data, with applications in,~e.g.,~seismology, medical imaging, radar technology, geophysical exploration and nondestructive testing.
Comprehensive surveys of this active research area are provided in~\cite{CakColHad23,ColKre19}.
In this article, we focus on the inverse medium scattering problem for time-harmonic acoustic waves, a setting governed by the Helmholtz equation. We assume to observe far field data for plane wave incident fields along all possible illumination directions at a fixed frequency and present in two dimensions a direct, fast reconstruction algorithm in the framework of the Born approximation, following the ideas in~\cite{AutGarHirHvy24,Garde25} for electrical impedance tomography (EIT).

It is well-known that full knowledge of the far field data uniquely determines the contrast.
However, inverse scattering is ill-posed as small perturbations in the observed data may result in huge errors in the reconstructed contrast.
Moreover, the data depend nonlinearly on the contrast due to multiple scattering effects.

There are two classes of reconstruction algorithms for this problem, either aiming to reconstruct (1) the entire contrast function or (2) the shape or boundary of its compact support.
Methods focusing on solving problem (2) are referred to as qualitative methods.
Prominent examples include the Linear Sampling Method, the Factorization Method, and monotonicity-based techniques (see, e.g., \cite{AudHad14,ColKir96,
GriHar18,Kir98,KirGri08}).
Quantitative methods addressing problem (1) typically either solve a computationally expensive regularized optimization problem or seek to invert a linearized model in the weak scattering regime, with the hope that the resulting inversion procedure remains effective even when the measured data contain nonlinear effects (see, e.g., \cite{ColMon88,KilMosSch12,
ZhoAudMenZha26}).
We particularly emphasize the recent low-rank method \cite{ZhoAudMenZha26} in the latter category as it is conceptually close to the approach proposed in this paper.

We assume that the contrast is supported in an {\em a priori} known $\bfc$-centered ball $B_R(\bfc)$ of radius~$R$.
By interpreting the available far field data in the standard way as the kernel of the far field operator on $\Ltwo(\Sone)$, the linearized forward map can be presented as a linear operator from
$\Ltwo(B_R(\bfc))$ to the space of Hilbert--Schmidt operators $\HS(\Ltwo(\Sone))$, which allows its parametrization as an infinite matrix after introducing bases for these separable Hilbert spaces.
In particular, expanding the angular dependence of the unknown contrast and the observed Born far field data in Fourier bases reveals a decoupling of the system in the angular direction: a certain angular frequency in the contrast only affects a corresponding diagonal in the Born far field operator expanded  with respect to a (modulated) Fourier basis.
This observation motivates the application of concepts in numerical linear algebra to solve each of the decoupled infinite linear systems that connect the radial behavior of the contrast for a certain angular frequency to a diagonal in the data matrix. Our proposed algorithm essentially applies a QR factorization in an offline stage to the angular subsystems, allowing one to only solve (small) triangular systems when the data become available.

Another key component is the intrinsic low-rank structure of the measured far field data caused by the super-exponential decay of Bessel functions with respect to their order, providing a natural justification for truncating the infinite-dimensional systems while essentially maintaining the retrievable information. This is particularly important for our algorithm since the QR factorization step applies a Gram--Schmidt orthogonalization in an appropriately weighted $L^2$ topology on a finite interval to certain products of Bessel functions, a process that is guaranteed to become unstable if run too long due to the aforementioned super-exponential decay.
The sparsity of far field data has been previously investigated for related inverse problems in~\cite{GriSch24,GriSch25}, and foundational results on this research direction can be found in~\cite{GriHanSyl14,GriSyl16,GriSyl17two}.

The low-rank structure of the inverse medium scattering problem has also been analyzed from a different perspective in \cite{Men23,ZhoAudMenZha26}.
In these works, the far field data are expanded in terms of eigenfunctions of a restricted Fourier integral operator, known as the generalized prolate spheroidal wave functions.

The remainder of this article structures as follows.
In Section~\ref{sec:InhomMediumScattering},
we provide theoretical background on acoustic inhomogeneous medium scattering as well as on its linearization given by the Born approximation.
Section~\ref{sec:MatrixReprForwardProblem} then derives a representation for the linearized system, decoupling it into triangular systems along the radial direction over the angular frequencies in the contrast.
Finally, in Section~\ref{sec:Regularization}, the essential support of the expansion coefficients of the far field data is utilized to justify a truncation-based regularization for inverting the infinite-dimensional triangular systems.
At this stage, we also discuss numerical instability in the construction of the radial basis functions.
Numerical results are presented in Section \ref{sec:NumericalExamples}, where we test our method using (noisy) linearized Born and nonlinear full far field data and benchmark it against the low-rank method \cite{ZhoAudMenZha26} and the MATLAB's built-in nonuniform fast Fourier transform (NUFFT). We close the article with some conclusions.

\section{Inhomogeneous medium scattering}
\label{sec:InhomMediumScattering}

We consider scattering of time-harmonic acoustic waves by a compactly supported penetrable object lying in a homogeneous background medium.
Let $\kappa>0$ denote the wave number and assume the incident field to be a plane wave
\begin{subequations}
\label{eq:scatteringProblem}
\begin{equation}
\label{eq:incidentWave}
	\ui(\bfx,\idir) \,:=\, \rme^{\rmi\kappa\bfx\cdot\idir}\,, \qquad \bfx\in\R^2\,,
\end{equation}
that propagates along an illumination direction $\idir\in\Sone:=\{\bfx\in\R^2\,:\,|\bfx|=1\}$.
In the following, all dependencies on $\idir$ are marked by a second argument.
The incident field $\ui(\ph,\idir)$ hits and interacts with a compactly supported penetrable scatterer $\Omega\subset \R^2$, modeled by the exterior support of a real-valued contrast function $q\in\Linfty(\R^2)$ satisfying $q>-1$ a.e.~on $\Omega$ and $q=0$ a.e.~on $\R^2\setminus\overline\Omega$.
This interaction generates a total field $u(\ph,\idir)\in H^1_{\loc}(\R^2)$ that solves
\begin{equation}
\label{eq:HelmholtzEquation}
	\Delta u(\ph,\idir) + \kappa^2(1+q)u(\ph,\idir) \,=\, 0 \qquad \text{in } \R^2\,,
\end{equation}
such that the scattered field $\us(\ph,\idir):=u(\ph,\idir)-\ui(\ph,\idir)$ satisfies the Sommerfeld radiation condition
\begin{equation}
    \label{eq:Sommerfeld}
    \lim_{r\to\infty} \sqrt{r} \Bigl( \frac{\di \us}{\di r}(\bfx,\idir)
    - \rmi\kappa \us(\bfx,\idir) \Bigr) \,=\, 0 \,, \qquad r=|\bfx|\rightarrow\infty \,,
\end{equation}
\end{subequations}
uniformly with respect to the direction $\xhat = \bfx/|\bfx|\in\Sone$.

It is known that the unique weak solution $u(\ph,\idir)\in H^1_{\loc}(\Rd)$ of~\eqref{eq:scatteringProblem} (see,~e.g.,~\cite[Thm.~7.13]{Kir21}) satisfies the Lippmann--Schwinger integral equation
\begin{equation}
  \label{eq:LippmannSchwinger}
  u(\ph,\idir) 
  \,=\, \ui(\ph,\idir)
  + \kappa^2 \int_\Omega q(\bfy)u(\bfy,\idir) \Phi(\ph-\bfy)  \dy
  \,=:\, \ui(\ph,\idir)
  + L_q u(\ph,\idir)
  \qquad \text{in } \Omega \,,
\end{equation}
with $L_q:\Ltwo(\Omega)\rightarrow\Ltwo(\Omega)$ and $\Phi(\bfx):=\rmi/4\, H^{(1)}_0(\kappa|\bfx|)$, $\bfx\not=\bfzero$, denoting
the fundamental solution to the Helmholtz equation in free space at wave number $\kappa$ (see,~e.g.,~\cite[Thm.~7.12]{Kir21}). 
The function $H^{(1)}_0$ is the Hankel function of the first kind and order zero. Due to the asymptotic behavior of $H^{(1)}_0$ for a large argument, the scattered field $u^s(\ph,\idir)$ fulfills the asymptotic far field expansion
\begin{equation}
\label{eq:far_field_expansion}
  \us(\bfx,\idir) \,=\, \frac{\rme^{\rmi\pi/4}}{\sqrt{8\pi}}
  \frac{\rme^{\rmi \kappa r}}{\sqrt{\kappa r}} 
  \uinfty(\xhat,\idir) + O\bigl(r^{-\frac{3}{2}}\bigr) \,, \qquad 
  r=|\bfx| \to \infty \,,
\end{equation}
uniformly with respect to the observation direction $\xhat=\bfx/|\bfx|\in\Sd$. Here, the far field pattern ${\uinfty \in \Ltwo(\Sd\times\Sd)}$ is given by
\begin{equation}
\label{eq:full_far_field}
  \uinfty(\xhat,\idir) 
  \,=\, \kappa^2 \int_\Omega q(\bfy) u(\bfy,\idir) \rme^{-\rmi \kappa \xhat\cdot \bfy} \dy \,, 
  \qquad \xhat\in\Sd \,,
\end{equation}
(see, e.g., \cite[Thm.~7.15]{Kir21}), and it determines 
the associated far field operator
\begin{equation*}
  F: \Ltwo(\Sd) \ra \Ltwo(\Sd) \,, \quad
  \bigl(Fg\bigr)(\xhat)
  \,:=\, \int_{\Sd} \uinfty(\xhat,\idir) g(\idir)\ds(\idir)\,,
\end{equation*}
which maps superpositions of plane wave incident fields to the far field patterns of the associated scattered fields.
This operator is well-known to be compact, normal and of trace class (see, e.g., \cite{ColKre95b}). In particular, it is Hilbert--Schmidt on $\Ltwo(\Sone)$, i.e., $F \in \HS(\Ltwo(\Sone))$, and it can thus be viewed as an infinite-dimensional matrix after fixing an orthonormal system for $\Ltwo(\Sone)$.

In this work, we are interested in the inverse medium scattering problem of recovering the contrast $q$ from the knowledge of the associated far field operator $F$, or equivalently, from the knowledge of the far field data $\uinfty(\xhat,\idir)$ for all $\xhat,\idir\in\Sone$.
This broadly studied problem is known to be uniquely solvable (see, e.g.,~\cite[Thm.~7.28]{Kir21}) but severely ill-posed and nonlinear as multiple scattering effects have to be taken into account.

\subsection*{Linearization by considering the Born approximation}

If $L_q$ from \eqref{eq:LippmannSchwinger} satisfies 
$\|L_q\|_{\mathscr{L}(\Ltwo(\Omega))}\ll1$
the total field $u(\ph,\idir)$ can be accurately approximated by its Born approximation $u_B(\ph,\idir):= (I+L_q)\ui(\ph,\idir)$ for all $\idir\in\Sone$.
This follows by only accounting for the first two terms in the Neumann series of $(I - L_q)^{-1}$ when viewing the Lippmann--Schwinger equation \eqref{eq:LippmannSchwinger} as a fixed-point equation for $u(\ph,\idir)$.
The field $u_B(\ph,\idir)$ solves a source problem for the Helmholtz equation,
\begin{equation*}
	\Delta u_B(\ph,\idir) + \kappa^2u_B(\ph,\idir) \,=\, -\kappa^2 q\ui(\ph,\idir)\qquad \text{in } \R^2\,,
\end{equation*}
with the associated scattered field $u^s_B(\ph,\idir):=u_B(\ph,\idir)-\ui(\ph,\idir)$ fulfilling the Sommerfeld radiation condition \eqref{eq:Sommerfeld}.
The Born far field pattern $u_B^\infty(\ph,\idir)$ is defined by replacing the scattered field $u^s$ by its Born approximation $u^s_B$ in the asymptotic expansion \eqref{eq:far_field_expansion}, which leads to the representation
\begin{equation}
\label{eq:BornFF}
	u_B^\infty(\xhat,\idir) \,:=\, \kappa^2 \int_{\Omega} q(\bfy)\rme^{-\rmi \kappa (\xhat-\idir)\cdot\bfy}\dy\,, \qquad \xhat,\idir\in\Sone\,.
\end{equation}
The associated Born far field operator
\begin{equation}
\label{eq:BornffO}
	F_B:L^2(\Sone) \rightarrow L^2(\Sone)\,, \qquad
	(F_Bg)(\xhat) \,:=\, \int_{\Sone} u_B^\infty(\xhat,\idir) g(\idir) \ds (\idir)\,
\end{equation}
is also a Hilbert--Schmidt operator. We refer to inverse Born scattering as the task to recover the contrast $q$ from the knowledge of the Born far field operator $F_B$.

Formula \eqref{eq:BornFF} indicates that knowing $u_B^\infty(\xhat,\idir)$ for all $\xhat,\idir\in\Sone$ is equivalent to knowing the two-dimensional Fourier transform of $q$ on the origin-centered disk of radius $2 \kappa$
\begin{equation*}
    B_{2\kappa}(\bfzero)\,=\,\{ \kappa(\xhat-\idir)\;:\; \xhat,\idir\in\Sone \} \subset \R^2.
\end{equation*}
 However, sampling $\xhat$ and $\idir$ uniformly on $\Sone$ demonstrates that the natural sampling pattern for the Fourier transform of $q$ over $B_{2\kappa}(\bfzero)$ is non-uniform, as visualized in Figure~\ref{fig:sketchSamplingFourierData}. Since $q$ has compact support, its Fourier transform is an analytic function by virtue of the Paley--Wiener theorem, and the unique continuation principle thus reveals that the contrast $q$ is uniquely determined by the knowledge of the Born far field operator $F_B$. However, it follows from \cite{Kir17} that the eigenvalues of $F_B$ decay at the same rate as those of $F$, implying that inverse Born scattering is ill-posed to the same extent as the original inverse medium scattering problem.

\begin{figure}[t]
  \centering
  \includegraphics[width=.35\textwidth]{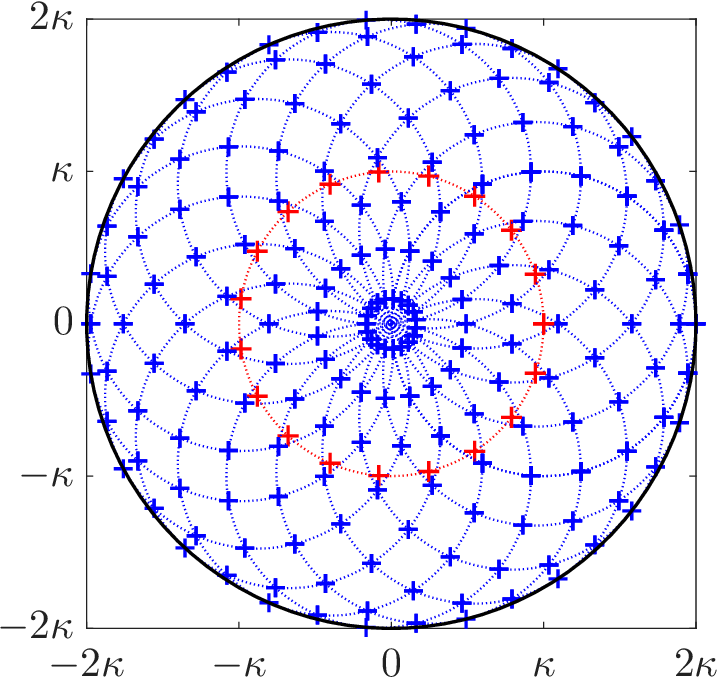}
  \caption{The natural non-uniform sampling pattern (blue crosses) for the Fourier data of the contrast $q$ in the definition \eqref{eq:BornFF} of the Born far field pattern $u_B^\infty$ for $20$ equiangular illumination directions $\idir$ (red crosses) and observation directions $\xhat$.} 
  \label{fig:sketchSamplingFourierData}
\end{figure}

Inspired by the treatment of the linearized continuum model of EIT in~\cite{AutGarHirHvy24,GarHyv24}, we adopt formulas \eqref{eq:BornFF}--\eqref{eq:BornffO} as the starting point for deriving a direct reconstruction formula for inverse Born scattering. Since we plan to treat the contrast $q$ as an element of $L^2(\Omega)$ and $F_B$ as a Hilbert--Schmidt operator on $L^2(\Sone)$, i.e., as an element of $\HS(\Ltwo(\Sone))$, we complete this section by noting that the forward map of Born scattering is indeed bounded between these spaces.

\begin{proposition}
\label{prop:boundedness}
    The linear forward map
    \begin{equation*}
        T_B: q  \, \mapsto \,  F_B
    \end{equation*}
    defined by \eqref{eq:BornFF}--\eqref{eq:BornffO} is bounded from $L^2(\Omega)$ to $\HS(\Ltwo(\Sone))$ with
    \begin{equation*}
    \| T_B \|_{\mathscr{L}(L^2(\Omega),\HS(\Ltwo(\Sone))} \, \leq \, 2 \pi \kappa^2 \sqrt{|\Omega|}\,.
    \end{equation*}
\end{proposition}

\begin{proof}
Since
\begin{equation*}
(F_B g)(\xhat) \, = \,  \kappa^2 \int_{\Sone} \int_{\Omega} q(\bfy)\rme^{-\rmi \kappa (\xhat-\idir)\cdot\bfy}\dy\, g(\idir) \ds (\idir)\,, \qquad \xhat \in\Sone\,,
\end{equation*}
the basic theory on Hilbert--Schmidt integral operators yields (see,~e.g.,~\cite{Weidmann1980})
\begin{equation}
\label{eq:forward_kernel}
\| T_B q \|_{\HS(\Ltwo(\Sone))}^2 
\leq 
\kappa^4 \int_{\Sone} \int_{\Sone} \bigg | \int_{\Omega} q(\bfy)\rme^{-\rmi \kappa (\xhat-\idir)\cdot\bfy}\dy \bigg|^2  \ds (\idir)  \ds (\xhat) \, .
\end{equation}
As 
\begin{equation*}
 \bigg | \int_{\Omega} q(\bfy)\rme^{-\rmi \kappa (\xhat-\idir)\cdot\bfy}\dy \bigg|^2 \leq |\Omega| \| q \|_{L^2(\Omega)}^2 \, 
\end{equation*}
by the Cauchy--Schwarz inequality, the assertion follows by integrating twice over $\Sone$ in \eqref{eq:forward_kernel} and taking the square root.
\end{proof}

\section{Angularly decoupled triangular systems for Born scattering}
\label{sec:MatrixReprForwardProblem}

 We start by formulating Born scattering in a matrix form between orthonormal bases of $L^2(B_1(\bfzero))$ and $\HS(\Ltwo(\Sone))$, then consider a specific choice for the radial parts of the basis for $L^2(B_1(\bfzero))$, and finally show that our choices lead to angularly decoupled triangular systems for determining the expansion coefficients of the contrast. In the following, we assume prior knowledge of a discoidal region of interest (ROI) $B_R(\bfc)\subset\R^2$ containing $\Omega$.

\subsection{Matrix representation}
We begin by deriving the expansion coefficients for the observed far field data with respect to a modulated Fourier basis following~\cite{GriSch24}. By linear substitution, we rewrite the Born far field pattern \eqref{eq:BornFF} as
\begin{equation*}
	u_B^\infty(\xhat,\idir) \,=\, \kappa^2 \int_{B_R(\bfzero)} q(\bfy+\bfc)\rme^{-\rmi\kappa\bfc\cdot(\xhat-\idir)}\rme^{-\rmi \kappa (\xhat-\idir)\cdot\bfy}\dy\,, \qquad \xhat,\idir\in\Sone\,.
\end{equation*}
Due to the Jacobi--Anger expansion (see, e.g.,~\cite[Eq.~(3.112)]{ColKre19}), the plane wave term in this formula can be expanded as
\begin{equation*}
\rme^{-\rmi \kappa (\xhat-\idir)\cdot\bfy} \,=\,  2\pi \sum_{m,n\in\Z}  \rmi^{n-m} \rme^{\rmi(n-m)\arg\bfy}J_m(\kappa|\bfy|)J_n(\kappa|\bfy|)\bfe_m(\xhat)\overline{\bfe_{n}(\idir)}\,,
\end{equation*}
where $(\bfe_m)_{m\in\Z}:=(\exp(\rmi m\arg(\ph))/\sqrt{2\pi})_{m\in\Z}$ is the standard Fourier basis of $\LtwoSd$ and $J_m$ denotes the Bessel function of the first kind and order $m$.
By introducing the modulated Fourier system $(\bfe^\bfc_m)_{m\in\Z}:=(\exp(-\rmi\kappa\bfc\cdot(\ph))\bfe_m)_{m\in\Z}$, which also forms an orthonormal basis for $\Ltwo(\Sone)$, we conclude that
\begin{equation*}
\rme^{-\rmi\kappa\bfc\cdot(\xhat-\idir)}\rme^{-\rmi \kappa (\xhat-\idir)\cdot\bfy} \,=\,  2\pi \sum_{m,n\in\Z}  \rmi^{n-m} \rme^{\rmi(n-m)\arg\bfy}J_m(\kappa|\bfy|)J_n(\kappa|\bfy|)\bfe^\bfc_m(\xhat)\overline{\bfe^\bfc_{n}(\idir)}\,.
\end{equation*}
This immediately yields a representation for the Born far field operator corresponding to $q$ in the orthonormal basis $(\langle\ph,\bfe^\bfc_n\rangle_{\Ltwo(\Sone)}\bfe^\bfc_m)_{m,n\in\Z}$ of $\HS(\Ltwo(\Sone))$ with expansion coefficients 
\begin{align}
	a_{m,n} &\,:=\, \langle F_B\bfe^\bfc_n,\bfe^\bfc_m\rangle_{L^2(\Sone)} \notag\\
    &\,=\, 2\pi \kappa^2\rmi^{n-m} \int_{B_R(\bfzero)} q(\bfy+\bfc)\rme^{\rmi(n-m)\arg\bfy}J_m(\kappa|\bfy|)J_n(\kappa|\bfy|)\dy \notag\\
	&\,=\, (2\pi)^{3/2} \kappa^2\rmi^{n-m} \int_{B_R(\bfzero)} q(\bfy+\bfc)\overline{\bfe_{m-n}(\yhat)}J_m(\kappa|\bfy|)J_n(\kappa|\bfy|)\dy \notag\\
    &\,=\, (2\pi)^{3/2} (\kappa R)^2\rmi^{n-m} \int_{B_1(\bfzero)} q(R\bfy+\bfc)\overline{\bfe_{m-n}(\yhat)}J_m(\kappa R|\bfy|)J_n(\kappa R|\bfy|)\dy \,,
    \label{eq:exp_coeff_FF}
\end{align}
where $\yhat = \bfy/|\bfy|\in\Sone$.

We continue by introducing an orthonormal basis $(\Psi_{j,k})_{j\in\Z,k\in\N_0}$ of $L^2(B_1(\bfzero))$ for expanding the shifted and rescaled contrast $q(R(\ph)+\bfc)$ to enable representing the forward map (cf.~Proposition~\ref{prop:boundedness})
\begin{equation}
\label{eq:forward}
L^2(B_1(\bfzero)) \ni q(R(\ph)+\bfc) \mapsto F_B \in \HS(\Ltwo(\Sone))
\end{equation}
between orthonormal bases. Observing that the kernel of the integral transform \eqref{eq:exp_coeff_FF}, mapping  $q(R(\ph)+\bfc)$ to $(a_{m,n})_{m,n \in \Z}$, separates into a radial $|\bfy|$-dependent part and an angular $\yhat$-dependent factor that is given by the standard Fourier basis, it is natural to search for the basis of $L^2(B_1(\bfzero))$ in the form
\begin{equation}
\label{eq:def_Psi}
    \Psi_{j,k}(\bfy):=\bfe_j(\yhat)R_k^{|j|}(|\bfy|),
    \qquad  \bfy\in B_1(\bfzero)\,.
\end{equation}
Here, the radial functions $(R_k^{|j|})_{j\in\Z,k\in\N_0}$ are chosen such that $(\Psi_{j,k})_{j\in\Z,k\in\N_0}$ forms an orthonormal basis for $\Ltwo(B_1(\bfzero))$; due to the orthonormality of the standard Fourier basis on $\Ltwo(\Sone)$, it straightforwardly follows that the necessary and sufficient condition is that $(R_k^{|j|})_{k\in\N_0}$ is an orthonormal basis for the weighted $L^2$-space
\begin{equation*}
L^2_r(0,1) = \bigg\{ f: (0,1) \to \C \text{ measurable} \; : \; \int_0^1 |f(r)|^2 r \dr < \infty \bigg\} 
\end{equation*}
for each $j \in \Z$. We may thus expand
\begin{equation}
\label{eq:exp_qc}
	q(R(\ph)+\bfc) \,=\, \sum_{j\in\Z} q_j \,=\, \sum_{j\in\Z}\sum_{k=0}^\infty c_{j,k} \Psi_{j,k}
	\qquad \text{for } c_{j,k}\,:=\, \big \langle q(R(\ph)+\bfc), \Psi_{j,k} \big \rangle_{\Ltwo(B_1(\bfzero))}\,,
\end{equation}
where we call $q_j$ the $j$-th angular frequency in $q(R(\ph)+\bfc)$. 
Note that $q_j$, $(a_{m,n})_{m,n\in\Z}$ and $(c_{j,k})_{j\in\Z,k\in\N_0}$ all depend on $\bfc$ and $R$, but we suppress this dependence to improve readability. The specific choice of the radial basis functions $(R_k^{|j|})_{j\in\Z,k\in\N_0}$ will be considered in Subsection~\ref{subsec:RadialBasis} below.

We address the inverse medium scattering problem by recovering the coefficients $(c_{j,k})_{j\in\Z,k\in\N_0}$ of the unknown contrast from the knowledge of the coefficients $(a_{m,n})_{m,n\in\Z}$ of the observed Born far field operator.
Inserting \eqref{eq:exp_qc} into \eqref{eq:exp_coeff_FF} and exploiting the orthonormality of the Fourier basis reveals the underlying infinite-dimensional system matrix $(b_{m,n}^{j,k})_{m,n \in \Z, j\in\Z,k\in\N_0}$, that characterizes the Born forward map \eqref{eq:forward} with respect to the chosen orthonormal basis:
\begin{equation}
\label{eq:LS}
    a_{m,n}
    \,=\, \sum_{j\in\Z}\sum_{k=0}^\infty b_{m,n}^{j,k} c_{j,k}\,,
    \qquad m,n\in\Z\,,
\end{equation}
for
\begin{equation}
\label{eq:matrix_coeff}
 b^{j,k}_{m,n} 
 \,:=\, \begin{cases}
 (2\pi)^{3/2}(\kappa R)^2(-\rmi)^{j} \big\langle R^{|j|}_{k}, J_m(\kappa R\ph)J_{m-j}(\kappa R\ph) \big\rangle_{L^2_r(0,1)} & \text{if } n=m-j\,, \\[1mm]
 0 & \text{else}.
 \end{cases}
\end{equation}
Having a closer look at \eqref{eq:LS}--\eqref{eq:matrix_coeff}, we make the following central structural observations.
\begin{remark}
\label{rem:propertiesLS}
    \begin{enumerate}
        \item[(i)] As in \cite{AutGarHirHvy24,GarHyv24} for the case of EIT, \eqref{eq:matrix_coeff} uncovers a decoupling of angular frequencies, which enables considering the diagonals $(a_{m,m-j})_{m\in\Z}$ of the data matrix separately for $j\in\Z$.
    Indeed, all available information on the $j$-th angular frequency $q_j$ in $q(R(\ph)+\bfc)$ is included in the $(-j)$-th diagonal of the data matrix, which allows to solve for the associated expansion coefficients $( c_{j,k} )_{k \in \N_0}$ from
        \begin{equation}
        \label{eq:LSj}
            a_{m,m-j}
            \,=\, \sum_{k=0}^\infty b_{m,m-j}^{j,k} c_{j,k}\,,
            \qquad m\in\Z\, ,
        \end{equation}
        for each $j \in \Z$.
        \item[(ii)]  For each $j\in\Z$, a half of the equations in \eqref{eq:LSj} are redundant due to a symmetry in  $m \in \Z$, which means that only a half of the equations needs to be considered and the remaining data can be used for noise filtering.
        Indeed, the reciprocity relation in the Born far field data
        \begin{equation*}
            u^\infty_B(\xhat,\idir) \,=\, u^\infty_B(-\idir,-\xhat)
            \qquad \text{for all } \xhat,\idir\in\Sone\,,
        \end{equation*}
        yields~\cite[Eq.~(2.28)]{GriSch24}
        \begin{equation}
        \label{eq:a_symmetry}
           a_{-(m-j),-m} \,=\,  (-1)^j\, a_{m,m-j}  \qquad \text{for all } m, j \in\Z\,.
        \end{equation}
        The same can alternatively be deduced from (see,~e.g.,~\cite[Eq.~(10.4.1)]{NIST:DLMF})
        \begin{equation*}
            J_{-m}J_{-(m-j)} \,=\, (-1)^jJ_mJ_{m-j}
            \qquad \text{for all } m, j \in\Z,
        \end{equation*}
       which also gives  
        \begin{equation*}
            b^{j,k}_{-(m-j),-m} \,=\, (-1)^j\,b^{j,k}_{m,m-j} \qquad \text{for all }m, j \in\Z \text{ and }  k \in \N_0,
        \end{equation*}
        by virtue of \eqref{eq:matrix_coeff}.
        Consequently, the $(-j)$-th diagonal of the data matrix $(a_{m,n})_{m,n\in\Z}$ is symmetric up to the factor $(-1)^j$ with respect to the (possibly virtual) element $a_{j/2,-j/2}$, and the analogous conclusion also holds for the system matrix $(b_{m,n}^{j,k})_{m,n \in \Z}$ with fixed $j$ and $k$.
        In consequence, we do not discard any unique equations in \eqref{eq:LSj} if we only consider 
        \begin{equation}
        \label{eq:LSj_reciprocity}
            a_{m,m-j}
            \,=\, \sum_{k=0}^\infty b_{m,m-j}^{j,k} c_{j,k}\,,
            \qquad  m\geq \tfrac j 2\,,
        \end{equation}
         where $a_{m,m-j}$ could be replaced by the averaged data 
         \begin{equation}
         \label{eq:noise_filt}
         \widetilde a_{m,m-j}:=(a_{m,m-j}+(-1)^ja_{-m+j,-m})/2\,,
         \end{equation}
         without altering the equations, to improve the signal-to-noise ratio.
    \hfill$\lozenge$
    \end{enumerate}
\end{remark}

It remains to construct orthonormal bases $(R_k^{|j|})_{k\in\N_0}$, $j \in \Z$, for $L^2_r(0,1)$ so that  inverting~\eqref{eq:LS}, or equivalently~\eqref{eq:LSj_reciprocity}, becomes  straightforward.
Following \cite{AutGarHirHvy24, GarHyv24}, we aim for a choice that makes the system matrix in~\eqref{eq:LSj_reciprocity} lower triangular for every $j \in \Z$. The radial Zernike bases employed in \cite{AutGarHirHvy24, GarHyv24} are unsuitable for this purpose but it turns out they can be replaced by bases generated through a Gram--Schmidt orthonormalization process of the products of Bessel functions appearing in \eqref{eq:matrix_coeff}.
The main difference between our setting and the linearized EIT framework of~\cite{AutGarHirHvy24, GarHyv24} is that here the functions on the right side of the inner product~\eqref{eq:matrix_coeff} are analytic functions with power series representations, whereas in \cite{AutGarHirHvy24, GarHyv24} this role is played by simple monomials in the radial variable. As the radial Zernike bases are orthogonal {\em polynomials} with respect to the $r$-weighted inner product, they are a natural choice for the bases in~\cite{AutGarHirHvy24, GarHyv24}, but there is no immediate reason why they would also interplay well with the products of Bessel functions appearing in \eqref{eq:matrix_coeff}. Even though it is possible to derive a closed formula for the inner products in~\eqref{eq:matrix_coeff} in terms of the radial Zernike bases, this does not generate an easily invertible representation.

\subsection{Choice of the radial bases}
\label{subsec:RadialBasis}

For $j,m\in\Z$, we define 
\begin{equation*}
    P_{m}^{j}(r) \,:=\, J_{m}(\kappa R r) J_{m-j}(\kappa R r)
    \,,\qquad r\in(0,1)\,,
\end{equation*}
which appear in \eqref{eq:matrix_coeff} and will act as our initial, i.e., non-orthonormalized, radial basis functions. To ease the notation, we set
\begin{equation*}
    \Z_{\geq c} := \{ k \in \Z \; : \;  k \geq c \}
\end{equation*}
for $c \in \R$.

\begin{proposition}
\label{prop:bases}
For each $j \in \Z$, the functions $(P_{m}^{j})_{m \in \Z_{\geq j/2}}$ are linearly independent and their linear span is dense in $L^2_r(0,1)$. 
\end{proposition}

\begin{proof}
    We start by proving that $(P_{m}^{j})_{m \in \Z_{\geq j/2}}$ are linearly independent for any fixed $j \in \N_0$. 
     Let $\emptyset \not= \mathcal{M} \subset \Z_{\geq j/2}$ be finite and $\{ \alpha_m \}_{m \in \mathcal{M}} \subset \C$ such that 
    \begin{equation*}
        \sum_{m \in \mathcal{M}} \alpha_m P_m^j(r) = 0 \qquad \text{for all } r \in (0,1)\,.
    \end{equation*}
Without loss of generality, we may assume that this index set for (potentially) nonzero coefficients is 
    \begin{equation*}
    \mathcal{M} = \{ l \in \Z \; : \;  \lceil j/2 \rceil \leq l \leq j + M \}
    \end{equation*}
    for some $M \in \N$. Define
    \begin{equation}
    \label{eq:tri_poly}
        S(\varphi) \,:=\, \frac{1}{2 \pi} \sum_{m \in \mathcal{M}} \alpha_m  {\rm e}^{{\rm i} m \varphi} = \frac{1}{2 \pi} \sum_{l = 0}^{M+\lfloor j/2 \rfloor} \alpha_{ \lceil j/2 \rceil + l} \, {\rm e}^{{\rm i} (\lceil j/2 \rceil + l) \varphi}, \qquad \varphi \in [0,2 \pi)\,.
    \end{equation}
    By Graf's addition theorem (e.g.,~\cite[Eq.~(10.23.7)]{NIST:DLMF} with $u=v= \kappa R r$, $\alpha=\varphi$, so $\chi=(\pi-\varphi)/2$),
    \begin{equation*}
      v_r(\varphi) := {\rm i}^j J_j\big(2 \kappa R r \sin(\varphi/2) \big)  {\rm e}^{{\rm i} j \varphi/2} = \sum_{m \in \Z} J_{m}(\kappa R r) J_{m-j}( \kappa R r) {\rm e}^{{\rm i} m \varphi} = \sum_{m \in \Z} P_m^j(r) {\rm e}^{{\rm i} m \varphi}
    \end{equation*}
    for $\varphi \in [0, 2 \pi)$. In particular, due to the orthogonality of the Fourier basis functions,
    \begin{equation}
    \label{eq:S_inner}
        \langle S, v_r \rangle_{L^2(0,2\pi)} = \sum_{m \in \mathcal{M}} \alpha_m P_m^j(r) = 0 \qquad \text{for all } r \in (0,1)\,.
    \end{equation}
    Our aim is to prove that this condition forces the coefficients $\{ \alpha_m \}_{m \in \mathcal{M}}$ to vanish, thus proving the claim.

    Substituting the Maclaurin series of $J_j$, $j \in \N_0$,
    into the left-hand integral in \eqref{eq:S_inner} gives
    \begin{equation*}
     \langle S, v_r \rangle_{L^2(0,2\pi)} = (-{\rm i})^j \sum_{k=0}^\infty \frac{(-1)^k}{k! (k+j)!} (\kappa R r)^{2k + j} \int_{0}^{2 \pi} S(\varphi) \sin^{2 k +j}(\varphi/2) \, {\rm e}^{-{\rm i} j \varphi/2} \, {\rm d} \varphi\,.
    \end{equation*}
    Due to \eqref{eq:S_inner}, the coefficients of all distinct powers of $r$ must vanish, resulting in
    \begin{equation}
    \label{eq:Fourier_cond}
    \int_{0}^{2 \pi} S(\varphi) \sin^{2 k +j}(\varphi/2) \, {\rm e}^{-{\rm i} j \varphi/2} \, {\rm d} \varphi = 0, \qquad k \in \N_0\,.
    \end{equation}
    By the binomial theorem,
    \begin{align*}
    \sin^{2 k + j}(\varphi/2) \, {\rm e}^{-{\rm i} j \varphi/2}
    & = \Big( \frac{1}{2 {\rm i}} ({\rm e}^{{\rm i} \varphi/2} -  {\rm e}^{- {\rm i} \varphi/2}) \Big)^{2 k + j} \, {\rm e}^{-{\rm i} j \varphi/2} \\[1mm]
    &= \frac{1}{(2 {\rm i})^{2k+j}}  \sum_{l=0}^{2 k + j} \binom{2k+j}{l} (-1)^{2k+j-l}{\rm e}^{-{\rm i} (k + j - l) \varphi} \\[1mm]
    &=\frac{(-1)^{k+\lceil j/2 \rceil}}{(2 {\rm i})^{2k+j}}  \sum_{m=-k - \lceil j/2 \rceil}^{k+\lfloor j/2 \rfloor } \binom{2k+j}{k+ \lfloor j/2 \rfloor - m} (-1)^{m} {\rm e}^{-{\rm i} (\lceil j/2 \rceil + m) \varphi},
    \end{align*}
    where $m = k + \lfloor j/2 \rfloor - l$, and we reversed the order of summation. Plugging this in \eqref{eq:Fourier_cond} and utilizing  \eqref{eq:tri_poly} leads to the infinite system of equations
    \begin{equation}
    \label{eq:lin_sys}
   \sum_{m=0}^{\min\{M,k\} + \lfloor j/2 \rfloor}  \binom{2k+j}{k+ \lfloor j/2 \rfloor -m} (-1)^{m} \alpha_{\lceil j/2 \rceil+m} 
   =    0, \qquad k \in \N_0\,.
    \end{equation}
    
    Let us consider \eqref{eq:lin_sys} only for $k \in \Z_{\geq M}$, so that the summation runs up to $M + \lfloor j/2 \rfloor$, and prove that it can only hold if $\alpha := [\alpha_{\lceil j/2 \rceil + m}]_{m=0}^{M + \lfloor j/2 \rfloor} \in \C^{M + \lfloor j/2 \rfloor + 1}$ vanishes. Expanding
    \begin{equation*}
        \binom{2k+j}{k+ \lfloor j/2 \rfloor -m}
        \,=\,
        \frac {(2k+j)!}{(k+\lceil j/2\rceil+m)!(k+\lfloor j/2\rfloor-m)!}
    \end{equation*}
    and multiplying \eqref{eq:lin_sys} with the nonzero factor $(k+j+M)!(k+\lfloor j/2 \rfloor)!/(2k+j)!$ yields
    \begin{equation}
    \label{eq:q_equiv_zero}
        q(k) \,:=\, \sum_{m=0}^{M + \lfloor j/2 \rfloor}  (-1)^m p_m(k) \alpha_{\lceil j/2 \rceil+m} 
    \,=\,
    0\,,
    \qquad k \in \Z_{\geq M}\,,
    \end{equation}
    where
    \begin{equation*}
        p_m(k)
        \,:=\,
        \prod_{l=1}^{M+\lfloor j/2\rfloor-m} (k+\lceil j/2\rceil + m + l)
		\prod_{s=0}^{m-1} (k+\lfloor j/2\rfloor - s)\,
    \end{equation*}
    are polynomials of degree $M + \lfloor j/2 \rfloor$.
    This means that $q(k)$ is also a polynomial of degree $M + \lfloor j/2 \rfloor$ with an infinite number of roots ${k\in\Z_{\geq M}}$, and thus the identity theorem for polynomials indicates that $q$ must be the zero polynomial,~i.e.,~that \eqref{eq:q_equiv_zero} holds, in fact, for all $k \in \R$.
    In particular, for
    \begin{equation*}
	  k = k_n \,:=\, n - \lfloor j/2\rfloor\,,
	   \qquad n\in\{0,\ldots,M+\lfloor j/2\rfloor\} \, ,
    \end{equation*}
    the identity \eqref{eq:q_equiv_zero} is equivalent to $D \alpha = 0$ for a matrix $D=(d_{m,n})\in\R^{(M+\lfloor j/2\rfloor+1)\times(M+\lfloor j/2\rfloor +1)}$ defined as
    \begin{equation*}
	   d_{m,n}
	   \,=\,
	   (-1)^m\prod_{l=1}^{M+\lfloor j/2\rfloor-m} (n-\lfloor j/2\rfloor+\lceil j/2\rceil + m + l)
		\prod_{s=0}^{m-1} (n-s)
		\begin{cases}
		 \,=\, 0\,, & \text{if } m>n\,, \\
		 \,\neq\, 0\,, & \text{else}\,.
		\end{cases}
    \end{equation*}
    Since $D$ is triangular with non-vanishing diagonal, it is invertible, which proves that $\alpha= 0$. This completes the proof that
    $(P_{m}^{j})_{m \in \Z_{\geq j/2}}$ are linearly independent for each $j \in \N_0$.

    For any $m, j\in\Z$,
\begin{equation}
\label{eq:P_jm}
    P_{m-j}^{-j}(r) 
    \,=\,
    J_{m-j}(\kappa R r) J_{(m-j)-(-j)}(\kappa R r)
    \,=\,
    J_mJ_{m-j}(\kappa R r)
    \,=\,
    P_m^{j}(r)\,,
\end{equation}
from which it follows that for any $j \in \Z$,
\begin{equation}
\label{eq:lin_ependence_Pjm}
(P_{m}^{j} )_{m \in \Z_{\geq j/2}} = (P_{m}^{-j} )_{m \in \Z_{\geq -j/2}} = (P_{m}^{|j|})_{m \in \Z_{\geq |j|/2}}\,,
\end{equation}
with the functions in these sets given in the same order with respect to increasing $m$. Hence, the linear independence of $(P_{m}^{j})_{m \in \Z_{\geq j/2}}$ for $j \in \Z \setminus \N_0$ follows from the corresponding result for $-j \in \N$, and $(P_{m}^{j})_{m \in \Z_{\geq j/2}}$ are thus linearly independent for each $j \in \Z$.
    
    Let $\rho \in L^2_r(0,1)$. We prove the assertion on the density by showing that $\rho$ can be orthogonal in $L^2_r(0,1)$ to all functions in the set $(P_{m}^{j})_{m \in \Z_{\geq j/2}}$ only if it vanishes. Define a shifted and scaled contrast as $q(R \bfy + \bfc) = \bfe_j(\yhat) \rho(|\bfy|)$, $\bfy \in B_1(\bfzero)$, for a fixed but arbitrary $j \in \Z$, which according to Remark~\ref{rem:propertiesLS} means that the data matrix $(a_{m,n})_{m,n\in\Z}$ only has nonzero elements on its $(-j)$-th diagonal. By inserting our choice of $q(R (\ph) + \bfc)$ into \eqref{eq:exp_coeff_FF} and integrating over $\Sone$, we get the representation (cf.~\eqref{eq:matrix_coeff})
    \begin{equation*}
        a_{m,m-j} =  (2\pi)^{3/2}(\kappa R)^2(-\rmi)^{j} \langle \rho,  P_{m}^{j} \rangle_{L^2_r(0,1)}\,, \qquad m \in \Z\,,
    \end{equation*}
    for the elements on the $(-j)$-th diagonal. If $\rho$ is orthogonal to all functions in $(P_{m}^{j})_{m \in \Z_{\geq j/2}}$, the $(-j)$-the diagonal of $(a_{m,n})_{m,n\in\Z}$ is thus empty due to the symmetry~\eqref{eq:a_symmetry}, meaning that the data matrix altogether vanishes. This means that our $q(R (\ph) + \bfc)$ is in the nullspace of the forward operator \eqref{eq:forward}, and thus $q(R (\ph) + \bfc)$ is zero almost everywhere by the unique solvability of the considered inverse Born scattering problem. This completes the proof.
\end{proof}

Now we are ready to properly introduce our basis $(\Psi_{j,k})_{j\in\Z,k\in\N_0}$ for $L^2(B_1(\bfzero))$ by defining the orthonormal basis $(R_k^{|j|})_{k \in \N_0}$  of $L^2_r(0,1)$ in \eqref{eq:def_Psi} for each $j \in \Z$.
Due to \eqref{eq:lin_ependence_Pjm}, we only need to consider the radial bases $(P_{m}^{|j|})_{m \in \Z_{\geq |j|/2}}$, $j \in \Z$, in what follows.

\begin{definition}
The radial basis functions $(R_k^{|j|})_{k \in \N_0}$, $j \in \Z$, in \eqref{eq:def_Psi} are defined by applying the Gram--Schmidt orthogonalization process with respect to the inner product of $L^2_r(0,1)$ to the functions $(P_{m}^{|j|})_{m \in \Z_{\geq |j|/2}}$. That is,
\begin{equation}
\label{eq:def_radial_basis}
  R_k^{|j|} \,:=\, \frac{\widetilde R_k^{|j|}}{\big \|\widetilde R_k^{|j|} \big \|_{L^2_r(0,1)}} 
    \qquad \text{with} \  \widetilde R_k^{|j|}
    \,:=\, P_{k+\lceil|j|/2\rceil}^{|j|} - \sum_{m=0}^{k-1} \big \langle P_{k+\lceil|j|/2\rceil}^{|j|},R_{m}^{|j|} \big \rangle_{L^2_r(0,1)}R_{m}^{|j|}
\end{equation}
for $k \in \N_0$.
\hfill$\lozenge$
\end{definition}

Since by Proposition~\ref{prop:bases} the functions $(P_{m}^{|j|})_{m \in \Z_{\geq |j|/2}}$ are linearly independent and their linear span is dense in $L^2_r(0,1)$, the set $(R_k^{|j|})_{k \in \N_0}$ is a well-defined orthonormal basis of $L^2_r(0,1)$ for each $j \in \Z$, which is precisely the requirement for the construction leading to \eqref{eq:LSj_reciprocity} to be valid. On the negative side, the procedure in \eqref{eq:def_radial_basis} is numerically unstable, which we will address in Section~\ref{sec:Regularization} below.

\subsection*{Angularly decoupled triangular systems}

Let us then examine how our specific choice for the radial bases simplifies the system \eqref{eq:LSj_reciprocity}. For $j \in \N_0$, $m \in \Z_{\geq j/2}$ and $k \in \N_0$,  formula \eqref{eq:matrix_coeff} gives
\begin{equation}
\label{eq:reduced_b}
b_{m,m-j}^{j,k} = (2\pi)^{3/2}(\kappa R)^2(-\rmi)^{j} \big\langle R^{|j|}_{k}, P_m^{j} \big\rangle_{L^2_r(0,1)} = (2\pi)^{3/2}(\kappa R)^2(-\rmi)^{j} \big\langle R^{|j|}_{k}, P_m^{|j|} \big\rangle_{L^2_r(0,1)}\,.
\end{equation}
The right side of \eqref{eq:reduced_b} vanishes if $k > m-\lceil |j|/2\rceil = m-\lceil j/2\rceil$ since $R^{|j|}_{k}$ is orthogonal to the subspace of $L^2_r(0,1)$ spanned by the first $k$ functions in $(P_{m}^{|j|})_{m \in \Z_{\geq |j|/2}}$ due to the Gram--Schmidt process \eqref{eq:def_radial_basis}. On the other hand, for $-j \in \N_0$ and $m \in \Z_{\geq j/2}$, it follows from \eqref{eq:P_jm} that
\begin{align}
b_{m,m-j}^{j,k} &= (2\pi)^{3/2}(\kappa R)^2(-\rmi)^{j} \big\langle R^{|j|}_{k}, P_m^{j} \big\rangle_{L^2_r(0,1)} \nonumber \\[1mm ]
&= (-\rmi)^{2j} (2\pi)^{3/2}(\kappa R)^2 (-\rmi)^{-j} \big\langle R^{|j|}_{k}, P_{m-j}^{-j} \big\rangle_{L^2_r(0,1)} = (-1)^{j} b_{m-j,m}^{-j,k}\,.
\label{eq:reduced_b_minus}
\end{align}
Since in this case $P_{m-j}^{-j} = P_{m+|j|}^{|j|}$, the orthogonalization process \eqref{eq:def_radial_basis} dictates that $b_{m,m-j}^{j,k} = (-1)^{j} b_{m-j,m}^{-j,k}$ vanishes when 
\begin{equation*}
k >  m + |j| - \lceil |j|/2\rceil =  m + \lfloor |j| /2\rfloor = m - \lceil j/2\rceil\,. 
\end{equation*}
Together with \eqref{eq:LSj_reciprocity}, these conclusions yield angularly decoupled triangular systems for determining the expansion coefficients $(c_{j,k})_{j \in \Z, k \in \N_0}$ from the knowledge of the data matrix $(a_{m,n})_{m,n\in \Z}$:
\begin{equation}
\label{eq:triangular0}
	a_{m,m-j}
	\,=\, \sum_{k=0}^{m-\lceil j/2\rceil} b_{m,m-j}^{j,k}c_{j,k}, \qquad m \in \Z_{\geq j/2} \, ,
\end{equation}
for $j \in \Z$. 

In order to recast \eqref{eq:triangular0} for a fixed $j\in\Z$ in a matrix-vector form, we introduce the infinite-dimensional vectors 
\begin{equation}
\label{eq:full_vectors}
	\bfc^j \,:=\, \big[c_{j,k-1} \big]_{k=1}^{\infty}
	\qquad\text{and}\qquad
	\bfa^j \,:=\, \big[ a_{(m-1)+\lceil j/2\rceil,(m-1)-\lfloor j/2\rfloor} \big]_{m=1}^{\infty}\,,
\end{equation}
corresponding to the $j$-th angular frequency in $q(R(\ph)+\bfc)$ and a half of the $(-j)$-th diagonal in the data matrix, respectively, as well as the lower triangular infinite-dimensional system matrix $\bfF^j:=[F^j_{m,k}]_{m,k=1}^\infty$ given componentwise as
\begin{equation}
\label{eq:F_comp}
	F^{j}_{m,k} \,=\, \begin{cases}
	(2\pi)^{3/2}(-\rmi)^j(\kappa R)^2 \big \langle R_{k-1}^{|j|}, P_{(m-1)+\lceil |j|/2\rceil}^{|j|} \big\rangle_{L^2_r(0,1)} & \text{if }  k\leq m
    \,,\\[1mm]
	0 & \text{else}\,.
	\end{cases}
\end{equation}
For each $j\in\Z$, the resulting system
\begin{equation}
\label{eq:infinite_system}
	\bfF^{j} \bfc^j\,=\, \bfa^j
\end{equation}
corresponds to \eqref{eq:triangular0} when $m$ runs from $\lceil j/2\rceil$ to infinity and can be solved through a forward substitution (not accounting for instability).
Note that according to \eqref{eq:reduced_b} and \eqref{eq:reduced_b_minus}, the latter term in the inner product in \eqref{eq:F_comp} should, in fact, be $P_{(m-1)+\lceil j/2\rceil}^{j}$, but employing \eqref{eq:P_jm} for $j < 0$, one deduces
\begin{equation*}
P_{(m-1)+\lceil j/2\rceil}^{j} = P_{(m-1)+\lceil j/2\rceil - j }^{-j} = P_{(m-1) + \lceil |j|/2\rceil}^{|j|}\,,
\end{equation*}
which allows the presented form.

These considerations lead to the following reconstruction formula that is our main result.
\begin{theorem}
\label{Theorem:main}
Denote by $(\Psi_{j,k})_{j\in\Z,k\in\N_0}$ the orthonormal system for $\Ltwo(B_1(\bfzero))$ defined by \eqref{eq:def_Psi} and \eqref{eq:def_radial_basis}, and let $(a_{m,n})_{m,n\in\Z}$ be the given expansion coefficients of the Born far field operator as in \eqref{eq:exp_coeff_FF}.
Then, the expansion coefficients of the shifted and scaled contrast $q(R(\ph)+\bfc)$ with respect to $(\Psi_{j,k})_{j\in\Z,k\in\N_0}$, as given in \eqref{eq:exp_qc}, can be computed separately for each $j \in \Z$ via a recursion with respect to $k$:
\begin{equation}
\label{eq:forward_substitution}
	c_{j,k} \,=\, \frac 1 {(2\pi)^{3/2}(\kappa R)^2(-\rmi)^j \big\| \widetilde{R}_{k}^{|j|} \big \|_{L^2_r(0,1)}} a_{k+\lceil j/2\rceil, k-\lfloor j/2\rfloor} 
	- \sum_{i=0}^{k-1} \frac{\big \langle R_{i}^{|j|}, P_{k+\lceil |j|/2\rceil}^{|j|} \big \rangle_{L^2_r(0,1)}}{\big\| \widetilde{R}_{k}^{|j|} \big\|_{L^2_r(0,1)} } \, c_{j, i}
\end{equation}
for $k \in \N_0$. Here, $(\widetilde{R}_{k}^{|j|})_{k \in \N_0}$ are the unnormalized orthogonal basis functions from \eqref{eq:def_radial_basis} and the sum in \eqref{eq:forward_substitution} is viewed to be empty for $k=0$.
\end{theorem}

\begin{proof}
Fix $j \in \Z$. Solving the lower triangular system \eqref{eq:infinite_system} via forward substitution directly gives
\begin{equation*}
	c_{j,k} \,=\, \frac 1 {(2\pi)^{3/2}(\kappa R)^2(-\rmi)^j \big \langle R_{k}^{|j|}, P_{k+\lceil |j|/2\rceil}^{|j|} \big \rangle_{L^2_r(0,1)}} a_{k+\lceil j/2\rceil, k-\lfloor j/2\rfloor} 
	- \sum_{i=0}^{k-1} \frac{\big \langle R_{i}^{|j|}, P_{k+\lceil |j|/2\rceil}^{|j|} \big \rangle_{L^2_r(0,1)}}{\big \langle R_{k}^{|j|}, P_{k+\lceil |j|/2\rceil}^{|j|}\big \rangle_{L^2_r(0,1)}} c_{j, i} \, 
\end{equation*}
for $k \in \N_0$. Representing $P_{k+\lceil |j|/2\rceil}^{|j|}$ as in \eqref{eq:def_radial_basis} and utilizing the orthogonality of $(R_{k}^{|j|})_{k \in \N_0}$ allows to replace the inner product in the denominator by the norm of $\widetilde{R}_{k}^{|j|}$, which is nonzero by construction. This completes the proof. 
\end{proof}

Observe that all inner products and norms appearing in \eqref{eq:forward_substitution} have already been computed (and stored) when running the Gram--Schmidt process in \eqref{eq:def_radial_basis}. Hence, applying the reconstruction formula \eqref{eq:forward_substitution}  is essentially for free, assuming the Gram--Schmidt process has been run offline prior to having the data in hand. It is also worth noting that our approach to solving the system \eqref{eq:LSj_reciprocity} is essentially an infinite-dimensional QR factorization, with \eqref{eq:forward_substitution} corresponding to solving the triangular system defined by the ``R'' matrix.

\begin{example}
\label{exa:analytic_example}
	As an example we consider $q=\chi_{B_r(\bfzero)}$ for some $0<r<1$, with $\chi_{(\, \cdot \, )}$ denoting the characteristic function of a given set.
	From \eqref{eq:exp_coeff_FF} we conclude that the expansion coefficients of the associated Born far field operator are given by
	\begin{equation*}
		a_{m,n} \,=\, \begin{cases}
		2\pi\kappa ^2\|J_m(\kappa|\ph|)\|^2_{L^2(B_r(\bfzero))} \,=\, 2\pi\|J_m(|\ph|)\|^2_{L^2(B_{\kappa r}(\bfzero))} & \text{if } m=n\,, \\
		0 & \text{else}\,,
		\end{cases}
	\end{equation*}
	i.e.,~by an infinite-dimensional diagonal matrix.
	By \cite[Eq.~(10.22.5)]{NIST:DLMF}, we can rewrite
	\begin{equation*}
		a_{m,n} \,=\, \begin{cases}
		2\pi^2(\kappa r)^2 \left( J^2_m(\kappa r) - J_{m-1}(\kappa r)J_{m+1}(\kappa r)\right) & \text{if } m=n\,, \\
		0 & \text{else}\,.
		\end{cases}
	\end{equation*}
	Thus, for this specific choice of $q$ we have an explicit formula for the data vector in \eqref{eq:full_vectors}, namely
	\begin{equation*}
		\bfa^j \,=\, \begin{cases}
		\big [2\pi^2(\kappa r)^2 \left( J^2_{m-1}(\kappa r) - J_{m-2}(\kappa r)J_{m}(\kappa r)\right) \big]_{m=1}^{\infty} & \text{if } j=0\,, \\[1mm]
		\bfzero & \text{else}\,.
		\end{cases}
	\end{equation*}
	Consequently, only one infinite-dimensional triangular system for $j=0$ has to be solved.
     \hfill$\lozenge$
\end{example}

\begin{example}
\label{ex:Zernike}
We visualize some elements of the orthonormal basis \eqref{eq:def_Psi} of $\Ltwo(B_1(\bfzero))$ for $\kappa R=2.5$ and $\kappa R=10$, with the radial components computed numerically via the Gram--Schmidt process \eqref{eq:def_radial_basis}. We restrict the angular frequency to $j \in \{0, \dots, 2 N \}$ and the radial index to $k \in \{ 0, \dots, N - \lceil j/2 \rceil \}$ with $N = 3$. According to $\eqref{eq:forward_substitution}$, the associated expansion coefficients of a scaled and shifted contrast, together with those for the corresponding negative frequencies $j \in \{-2 N, \dots, -1 \}$, can be determined from the knowledge of the truncated data matrix $( a_{m,n} )_{m,n=-N}^{N}$ (or, more precisely, from the knowledge of slightly more than a half of it). We use a Gauss--Legendre quadrature with $N_r=100$ nodes 
for evaluating the integrals involved in the Gram--Schmidt process \eqref{eq:def_radial_basis}. The resulting $(N+1)^2 = 16$ basis functions  are shown in Figure \ref{fig:Psi} for $\kappa R=2.5$ and in Figure \ref{fig:Psi_kR5} for $\kappa R=10$. The functions in Figure~\ref{fig:Psi} are qualitatively similar to the corresponding Zernike polynomials employed as the spatial basis functions for EIT in \cite{AutGarHirHvy24,GarHyv24}, but the ones in Figure \ref{fig:Psi_kR5} systematically have more oscillations in the radial direction and more finer details close to the center of the unit disk compared to the Zernike polynomials. To the best of our knowledge, no families of functions obtained via the orthonormalization of such products of Bessel functions have previously been documented in the literature.
\begin{figure}[t]
    \centering
    \begin{subfigure}[b]{.2\textwidth}
      \centering
      \includegraphics[width=\textwidth]{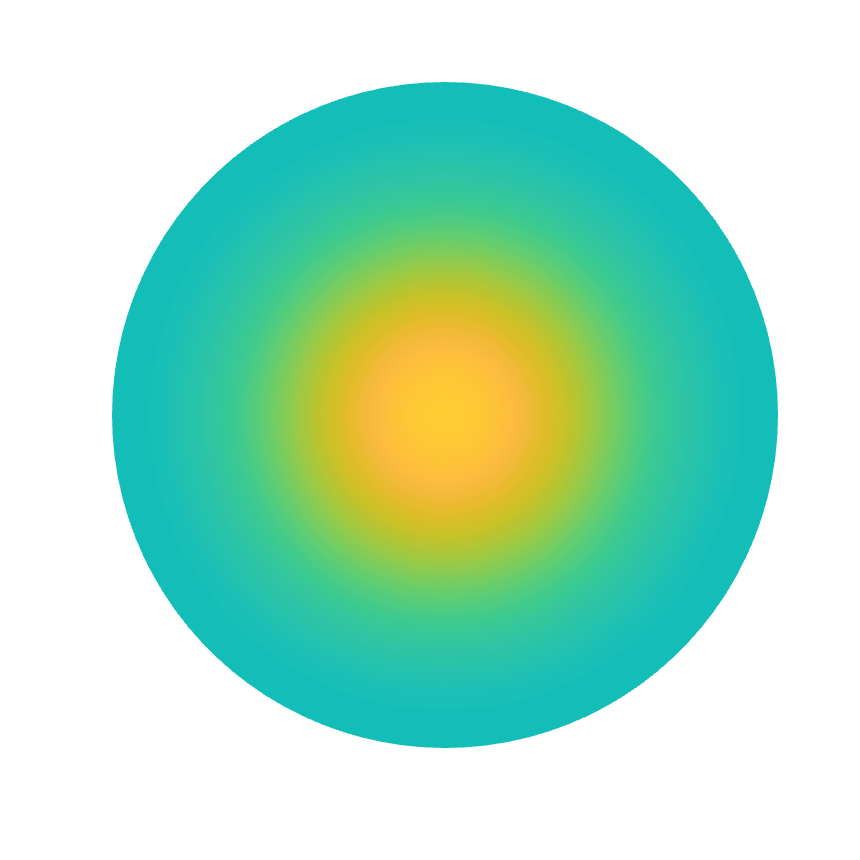}
    \end{subfigure}
    \begin{subfigure}[b]{.2\textwidth}
      \centering
      \includegraphics[width=\textwidth]{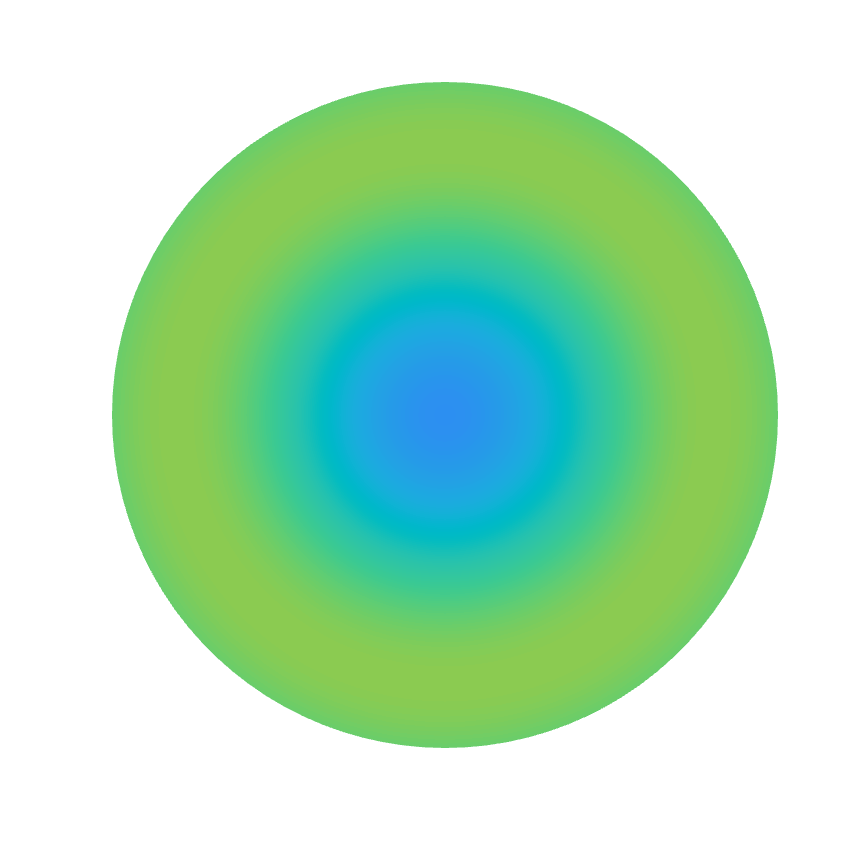}
    \end{subfigure}
    \begin{subfigure}[b]{.2\textwidth}
      \centering
      \includegraphics[width=\textwidth]{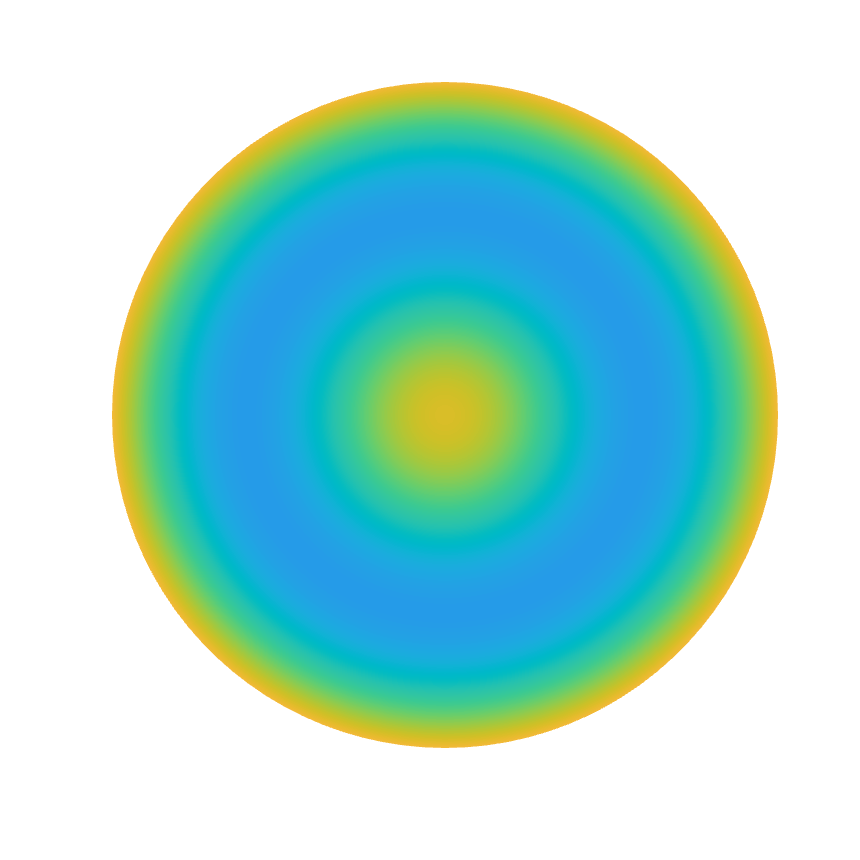}
    \end{subfigure}
    \begin{subfigure}[b]{.2\textwidth}
      \centering
      \includegraphics[width=\textwidth]{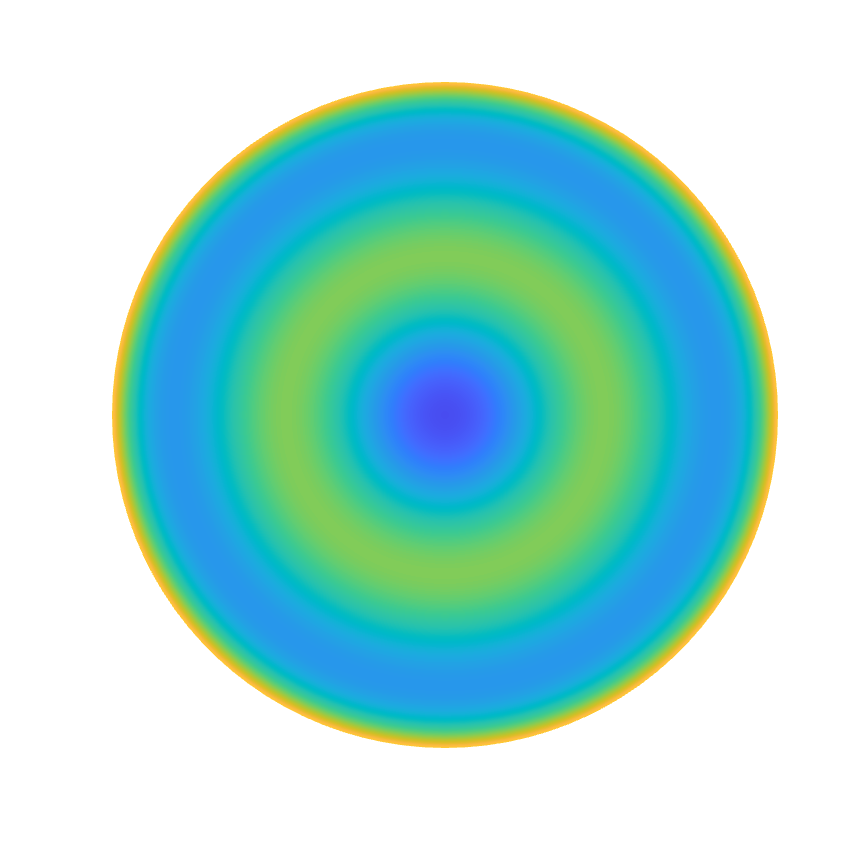}
    \end{subfigure}
    \begin{subfigure}[b]{.2\textwidth}
      \centering
      \includegraphics[width=\textwidth]{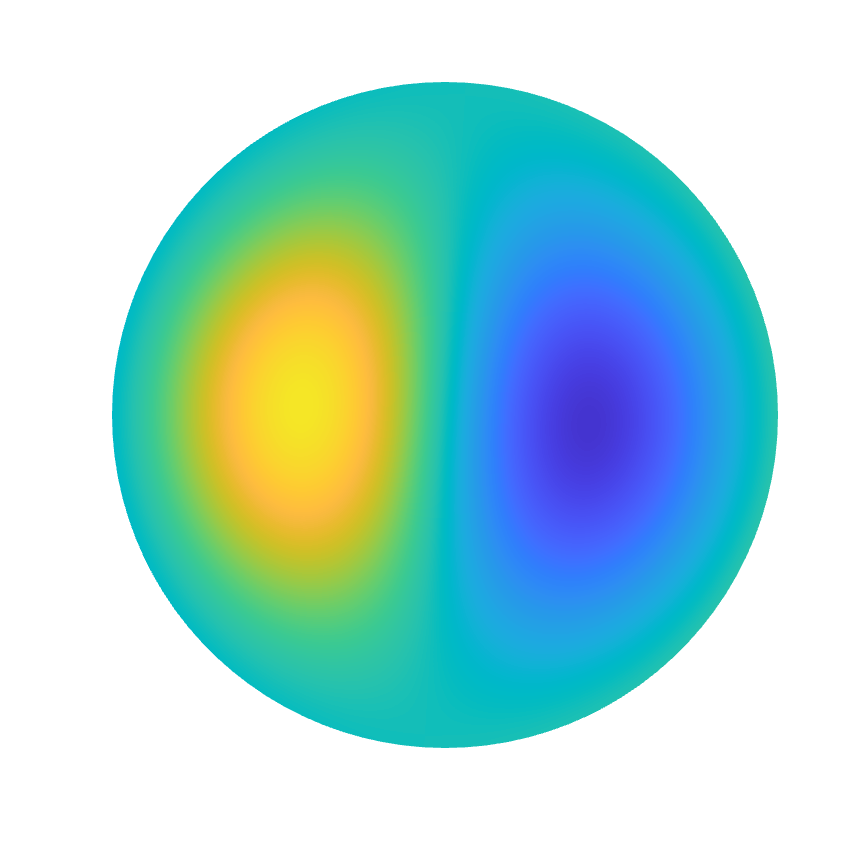}
    \end{subfigure}
    \begin{subfigure}[b]{.2\textwidth}
      \centering
      \includegraphics[width=\textwidth]{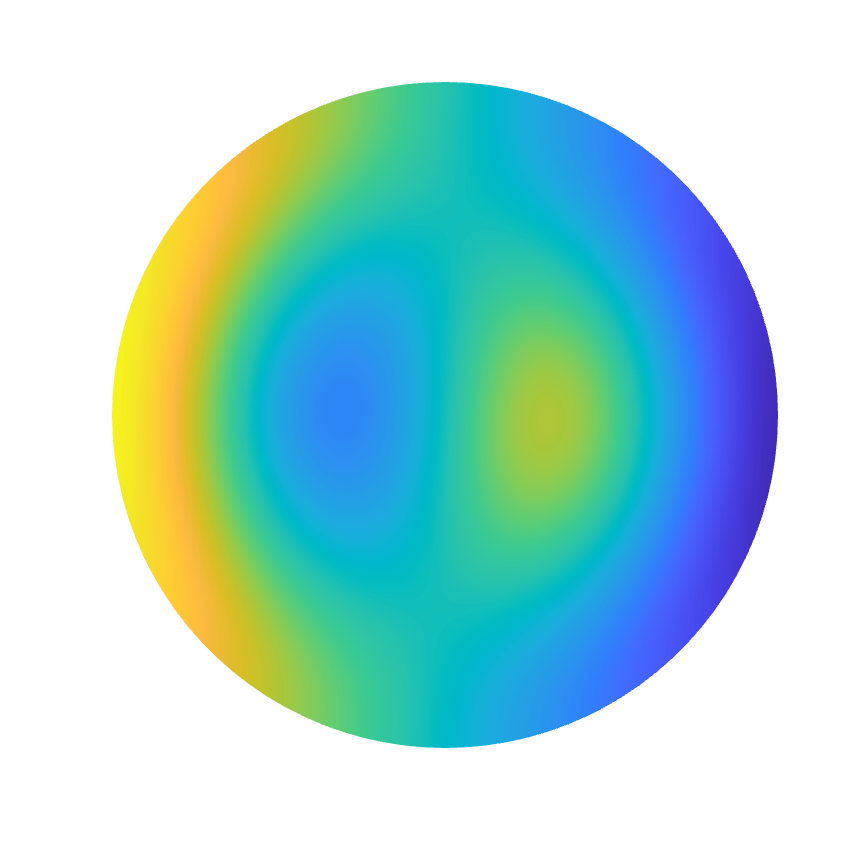}
    \end{subfigure}
      \begin{subfigure}[b]{.2\textwidth}
      \centering
      \includegraphics[width=\textwidth]{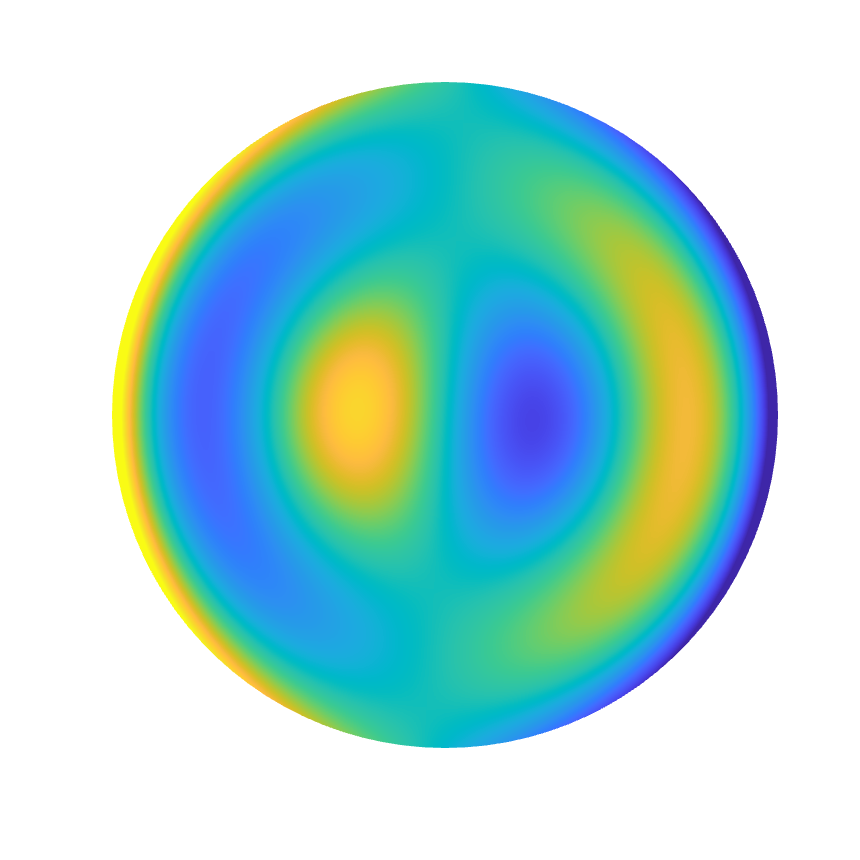}
    \end{subfigure}
    \begin{subfigure}[b]{.2\textwidth}
      \centering
      \includegraphics[width=\textwidth]{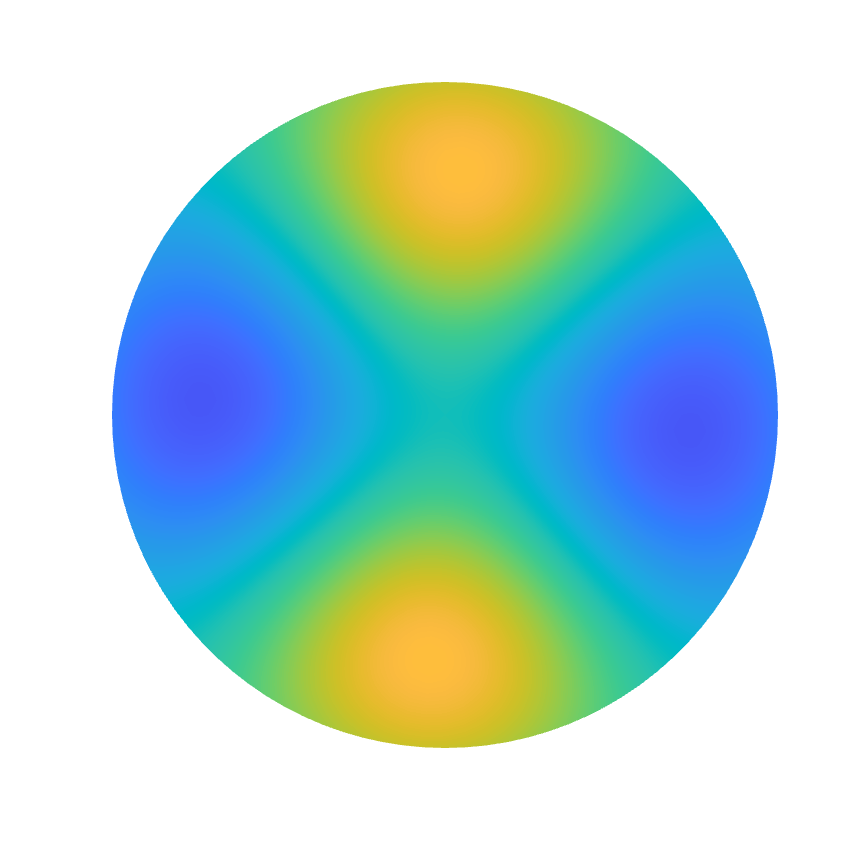}
    \end{subfigure}
    \begin{subfigure}[b]{.2\textwidth}
      \centering
      \includegraphics[width=\textwidth]{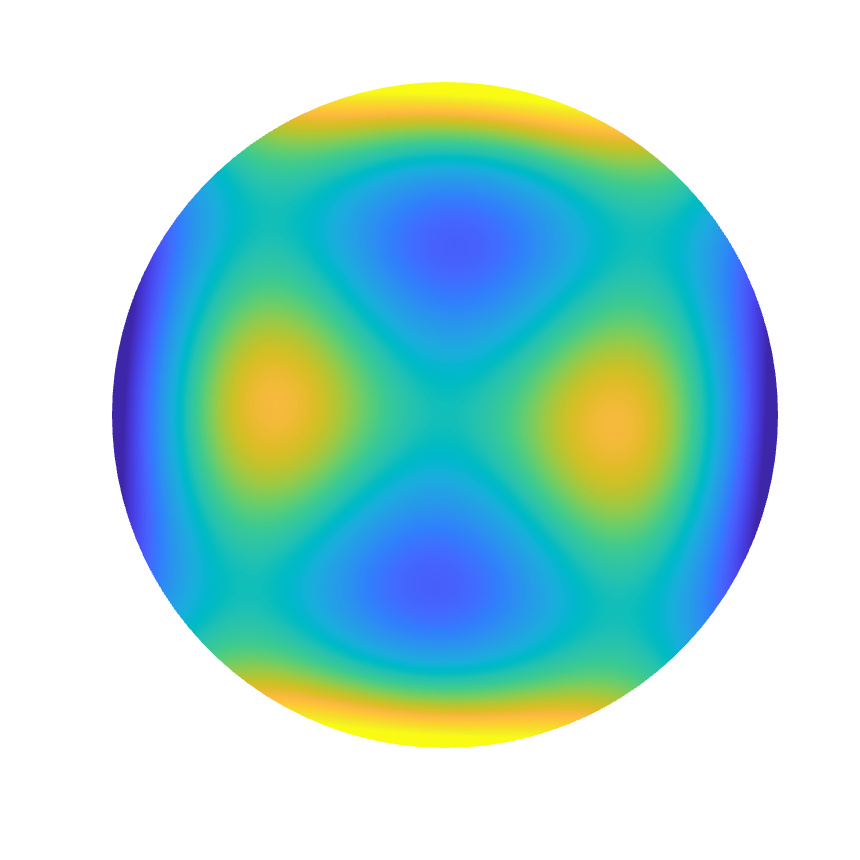}
    \end{subfigure}
    \begin{subfigure}[b]{.2\textwidth}
      \centering
      \includegraphics[width=\textwidth]{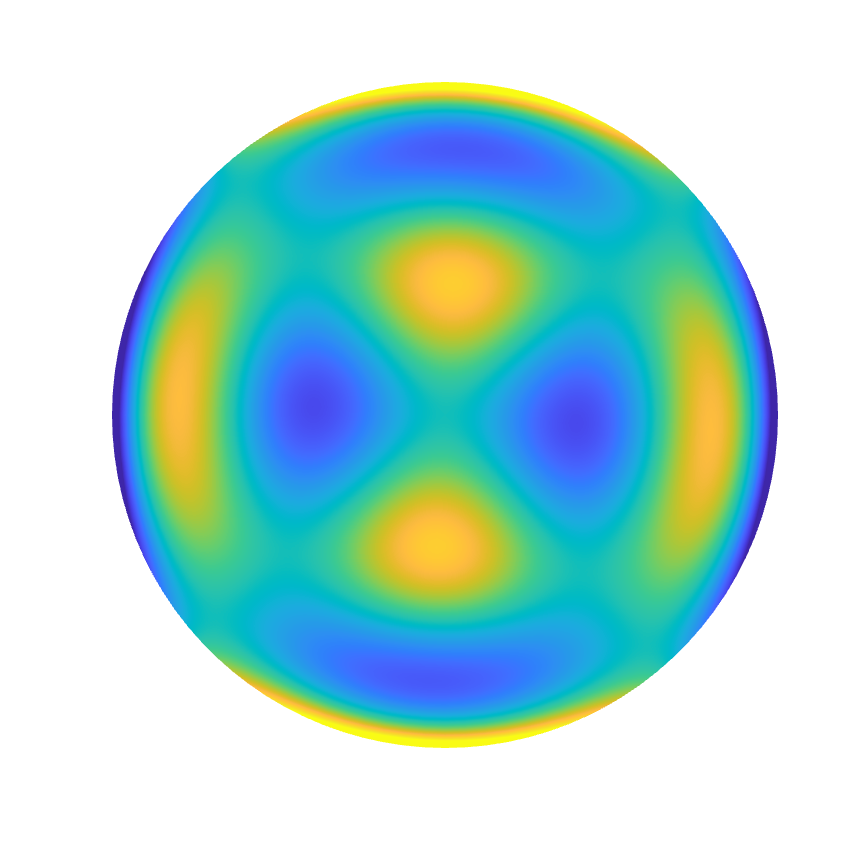}
    \end{subfigure}
    \begin{subfigure}[b]{.2\textwidth}
      \centering
      \includegraphics[width=\textwidth]{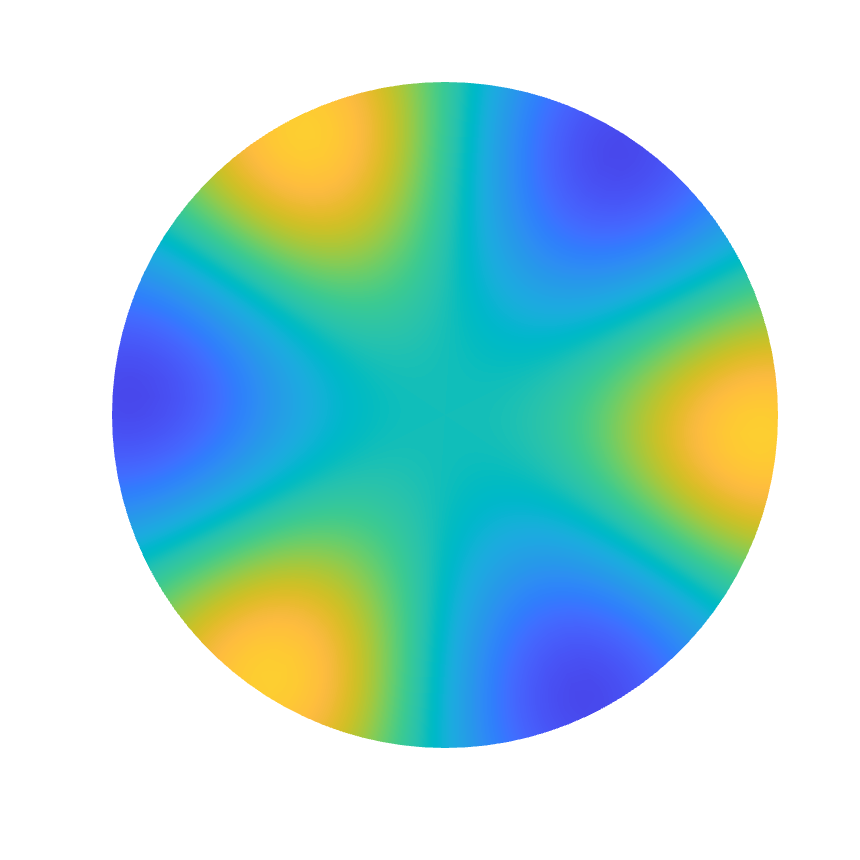}
    \end{subfigure}
    \begin{subfigure}[b]{.2\textwidth}
      \centering
      \includegraphics[width=\textwidth]{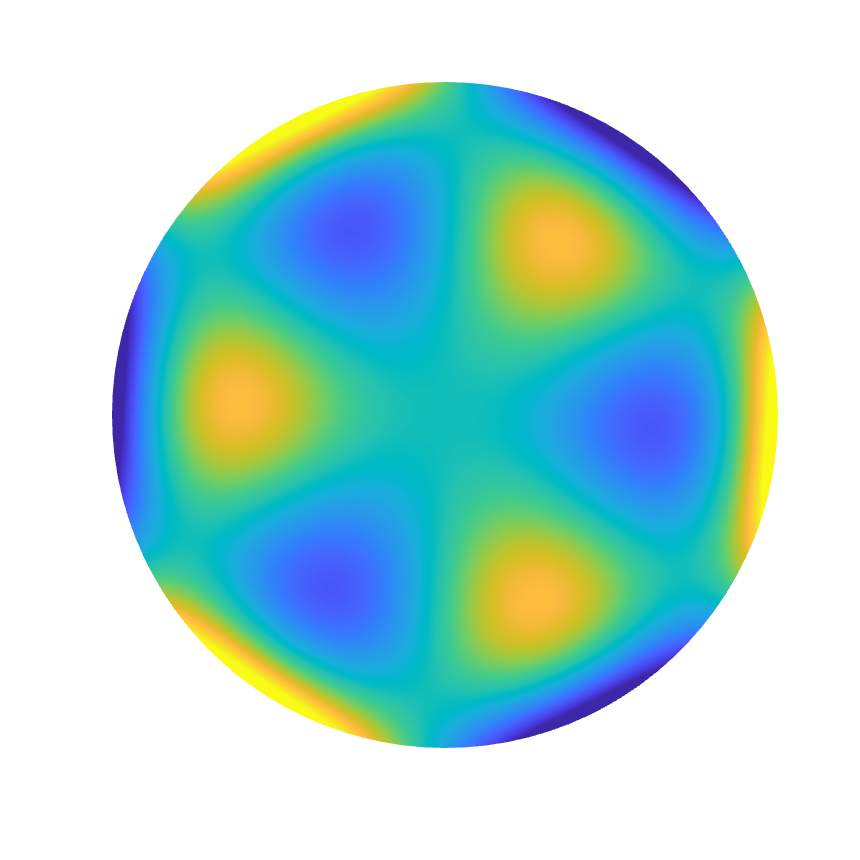}
    \end{subfigure}
    \begin{subfigure}[b]{.2\textwidth}
      \centering
      \includegraphics[width=\textwidth]{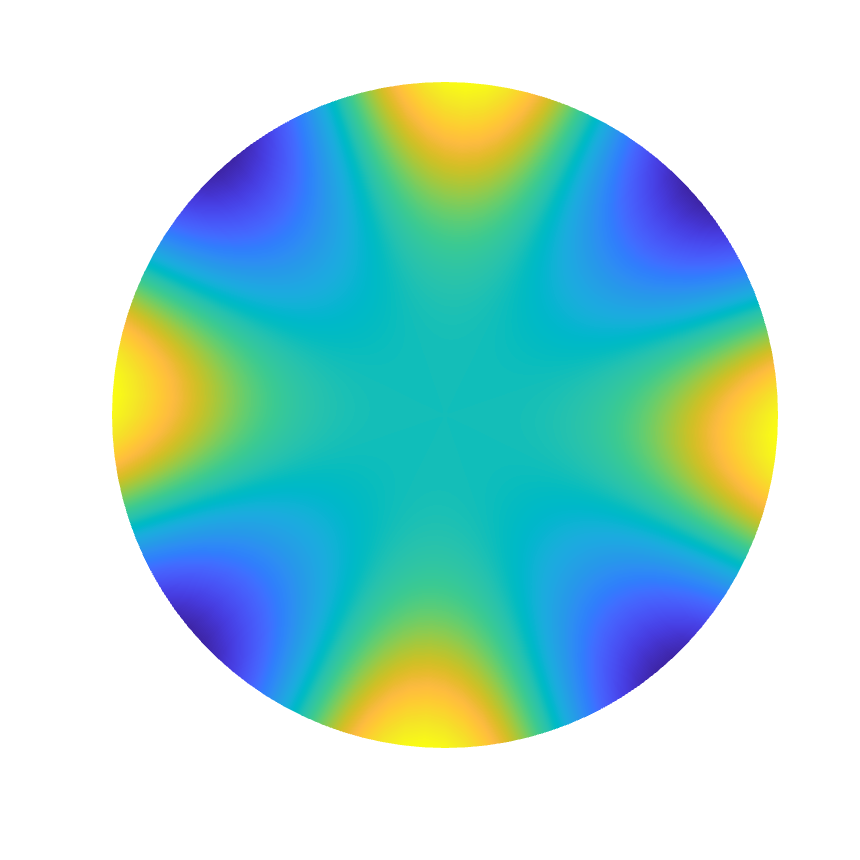}
    \end{subfigure}
    \begin{subfigure}[b]{.2\textwidth}
      \centering
      \includegraphics[width=\textwidth]{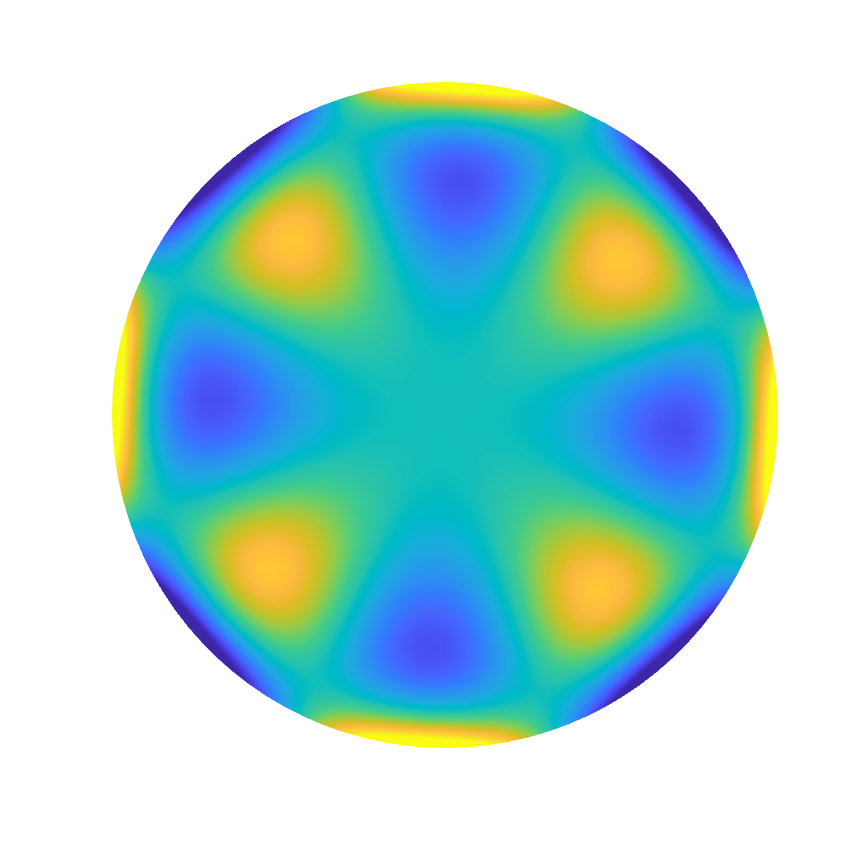}
    \end{subfigure}
    \begin{subfigure}[b]{.2\textwidth}
      \centering
      \includegraphics[width=\textwidth]{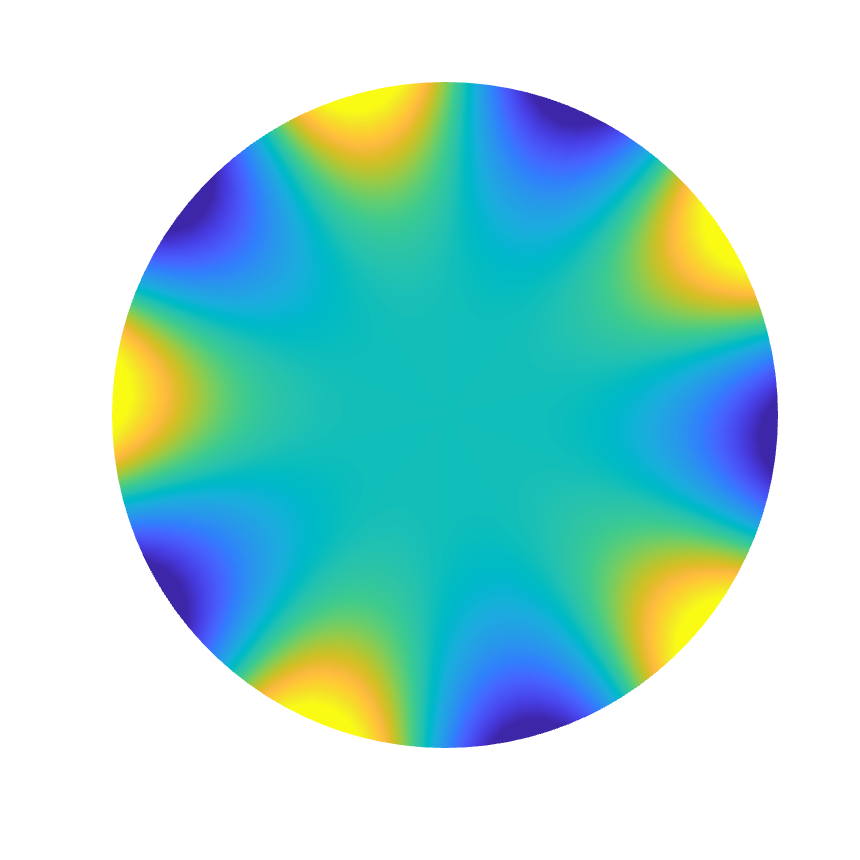}
    \end{subfigure}
    \begin{subfigure}[b]{.2\textwidth}
      \centering
      \includegraphics[width=\textwidth]{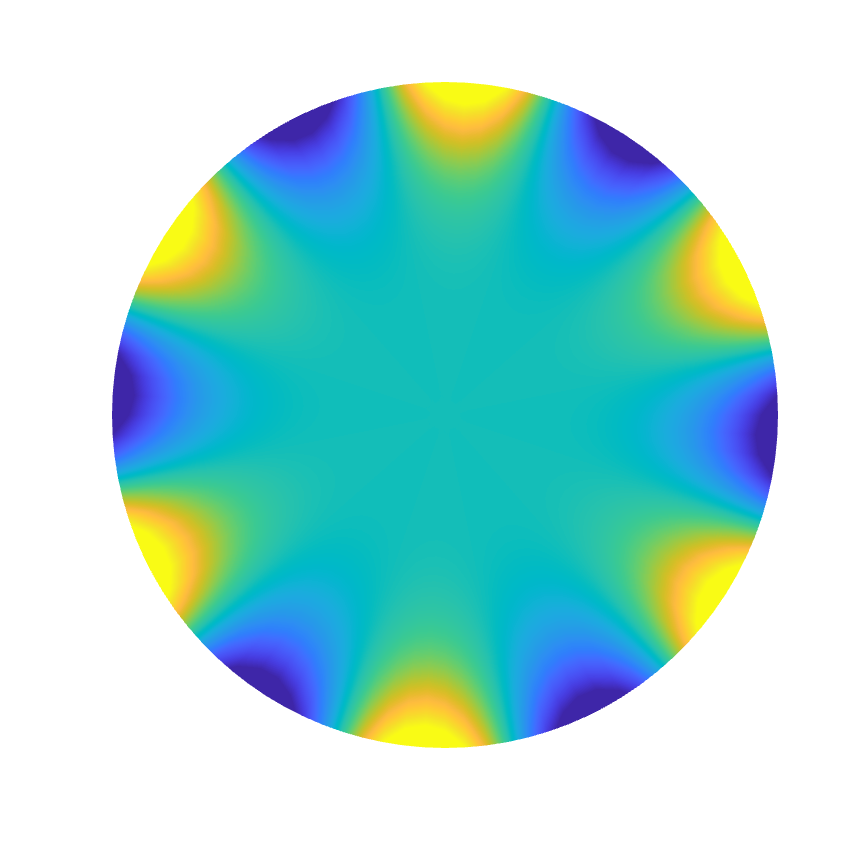}
    \end{subfigure}
    \caption{The real parts of the basis functions $(\Psi_{j,k})_{j\in\{0,\ldots,2N\}, k\in\{0,\ldots,N-\lceil j/2\rceil\}}$ for $\kappa R=2.5$ and $N=3$.
    First row:~$\Psi_{0,0}$, $\Psi_{0,1}$, $\Psi_{0,2}$ , $\Psi_{0,3}$.
    Second row:~$\Psi_{1,0}$, $\Psi_{1,1}$, $\Psi_{1,2}$ , $\Psi_{2,0}$.
    Third row:~$\Psi_{2,1}$, $\Psi_{2,2}$, $\Psi_{3,0}$ , $\Psi_{3,1}$.
    Fourth row:~$\Psi_{4,0}$, $\Psi_{4,1}$, $\Psi_{5,0}$ , $\Psi_{6,0}$. The color scale is the same in all subfigures.
    } 
    \label{fig:Psi}
  \end{figure}
\begin{figure}[t]
    \centering
    \begin{subfigure}[b]{.2\textwidth}
      \centering
      \includegraphics[width=\textwidth]{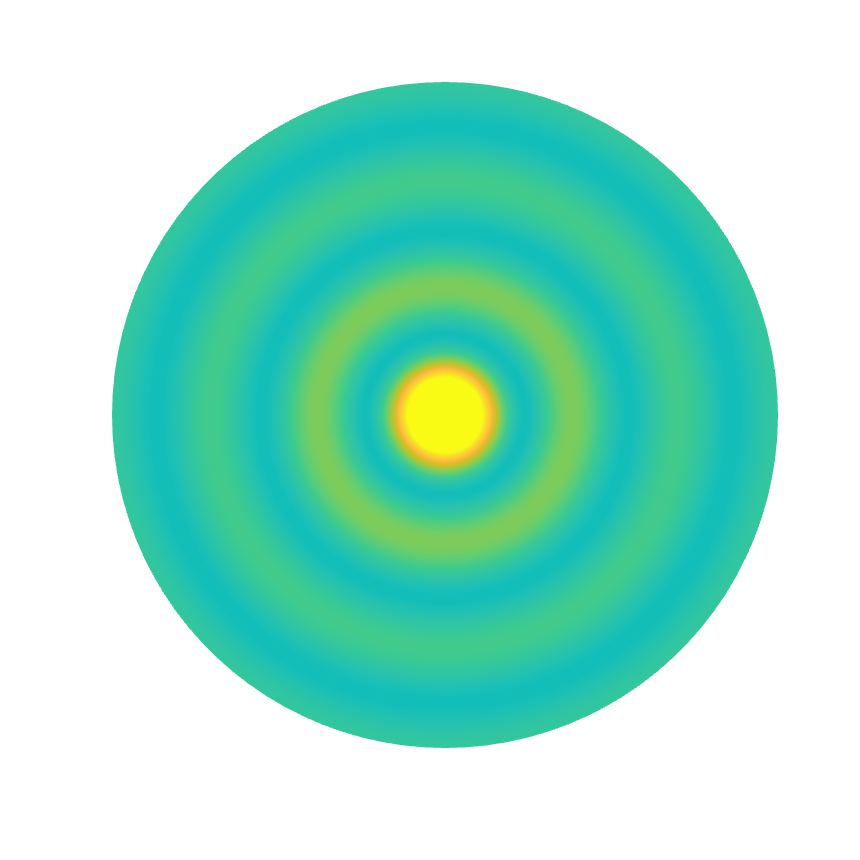}
    \end{subfigure}
    \begin{subfigure}[b]{.2\textwidth}
      \centering
      \includegraphics[width=\textwidth]{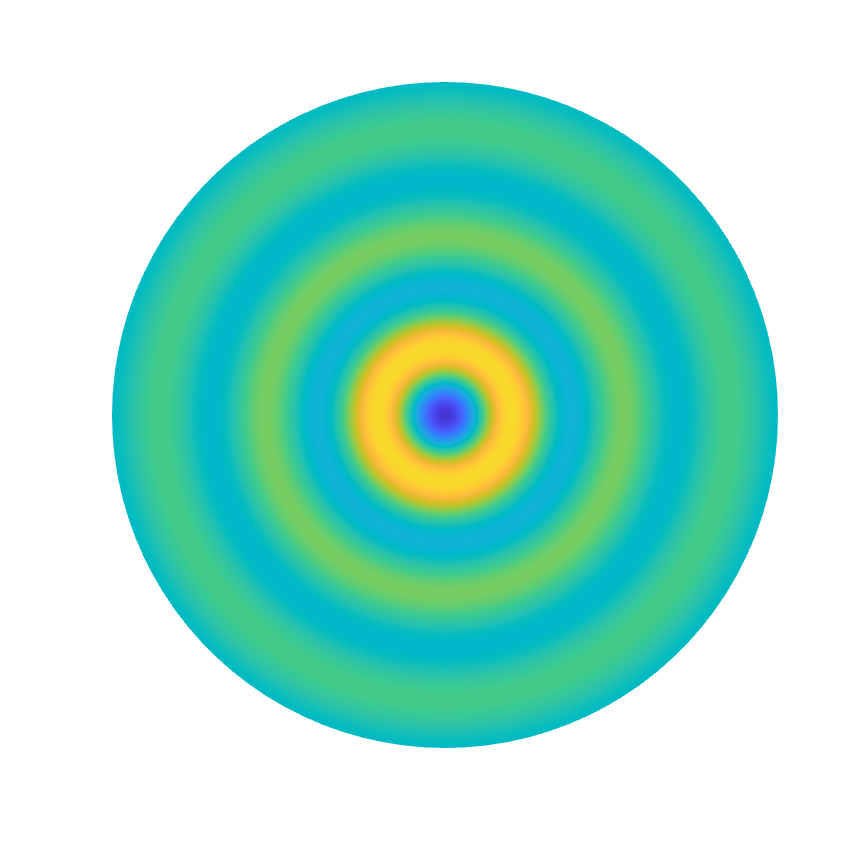}
    \end{subfigure}
    \begin{subfigure}[b]{.2\textwidth}
      \centering
      \includegraphics[width=\textwidth]{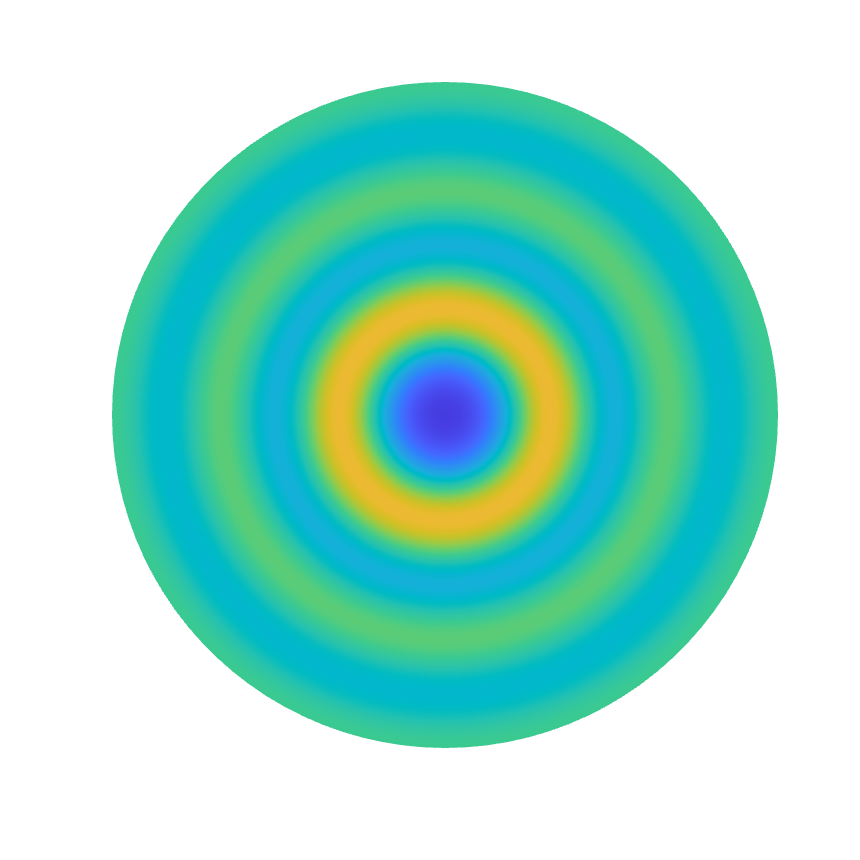}
    \end{subfigure}
    \begin{subfigure}[b]{.2\textwidth}
      \centering
      \includegraphics[width=\textwidth]{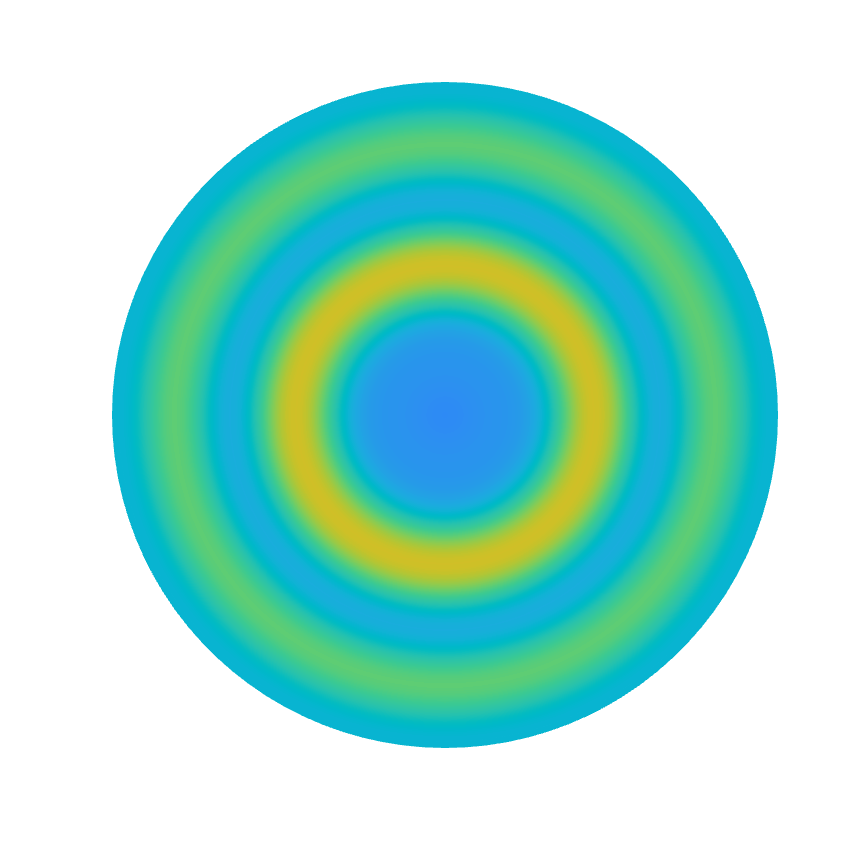}
    \end{subfigure}
    \begin{subfigure}[b]{.2\textwidth}
      \centering
      \includegraphics[width=\textwidth]{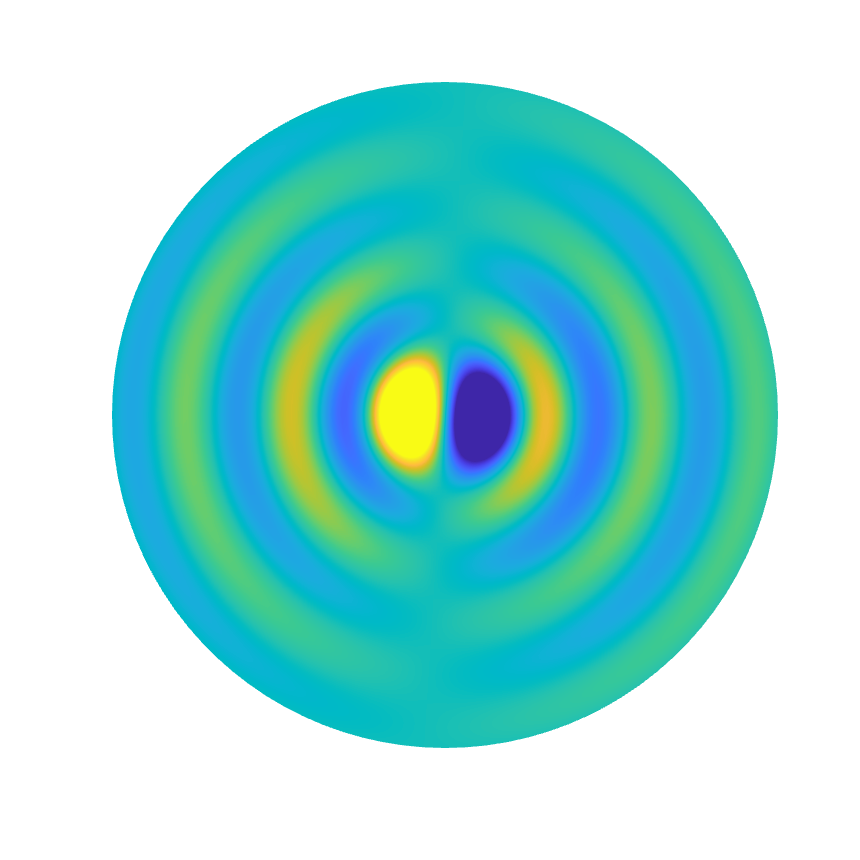}
    \end{subfigure}
    \begin{subfigure}[b]{.2\textwidth}
      \centering
      \includegraphics[width=\textwidth]{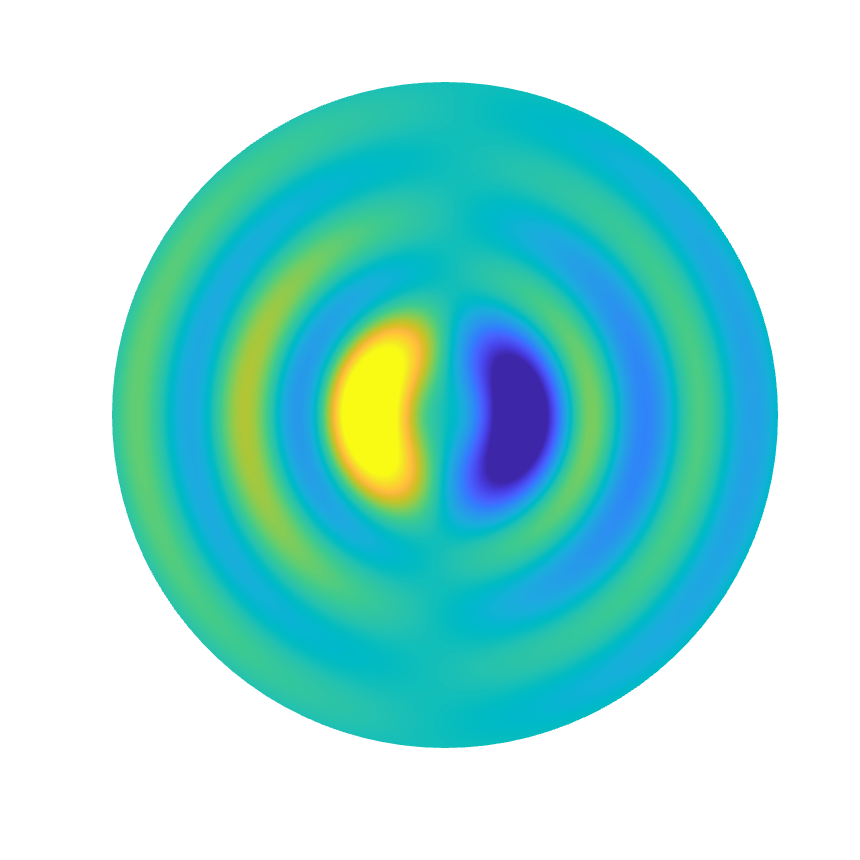}
    \end{subfigure}
      \begin{subfigure}[b]{.2\textwidth}
      \centering
      \includegraphics[width=\textwidth]{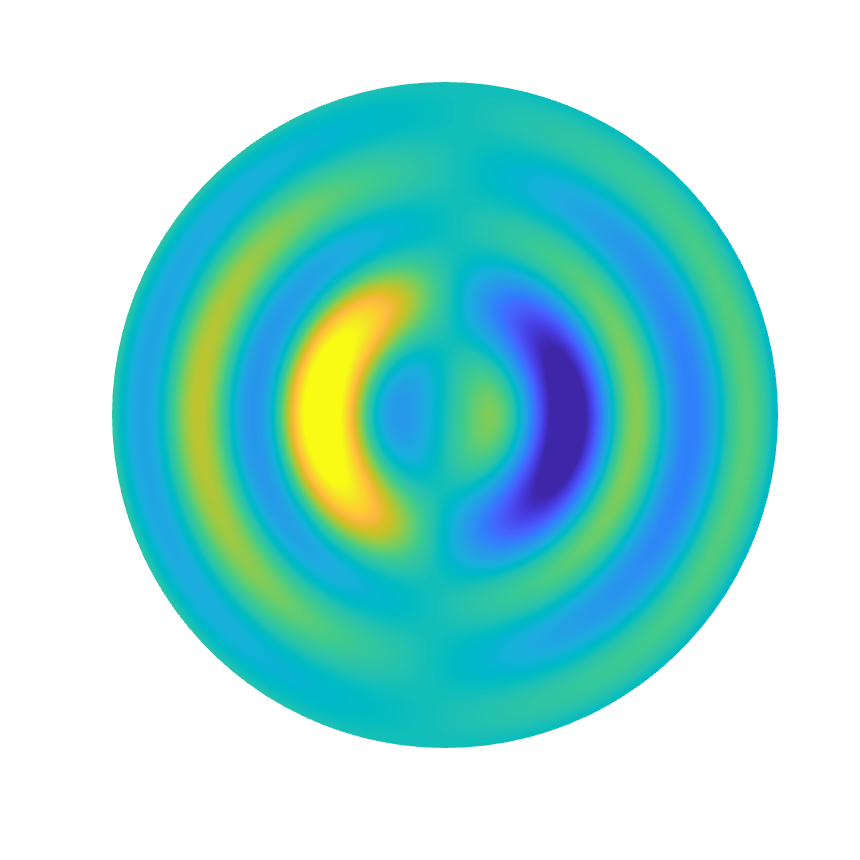}
    \end{subfigure}
    \begin{subfigure}[b]{.2\textwidth}
      \centering
      \includegraphics[width=\textwidth]{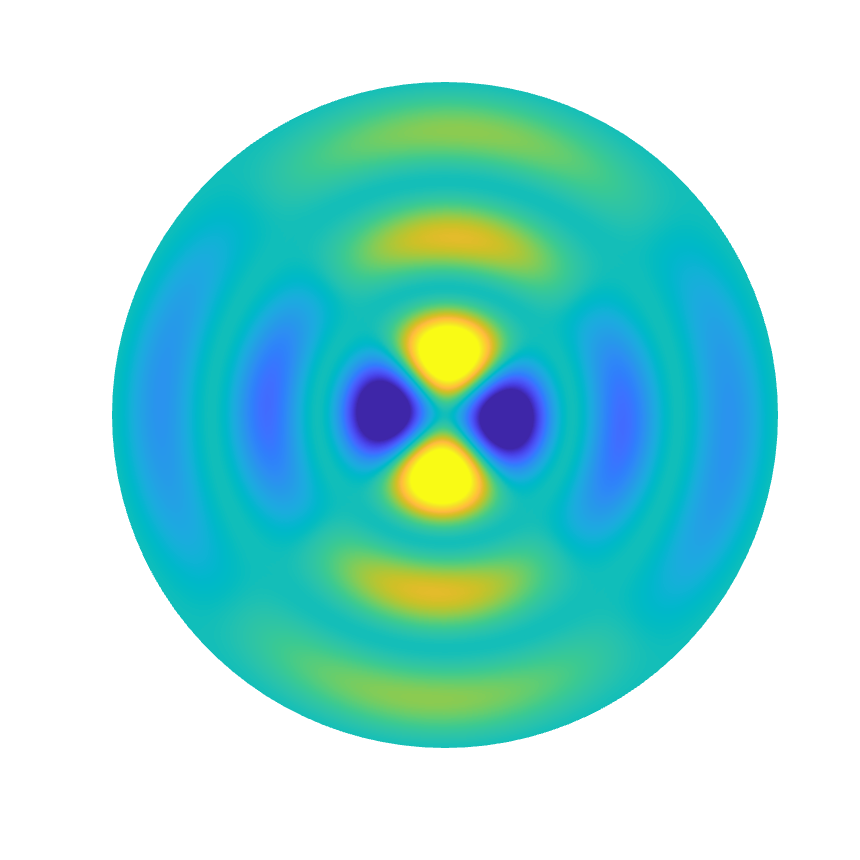}
    \end{subfigure}
    \begin{subfigure}[b]{.2\textwidth}
      \centering
      \includegraphics[width=\textwidth]{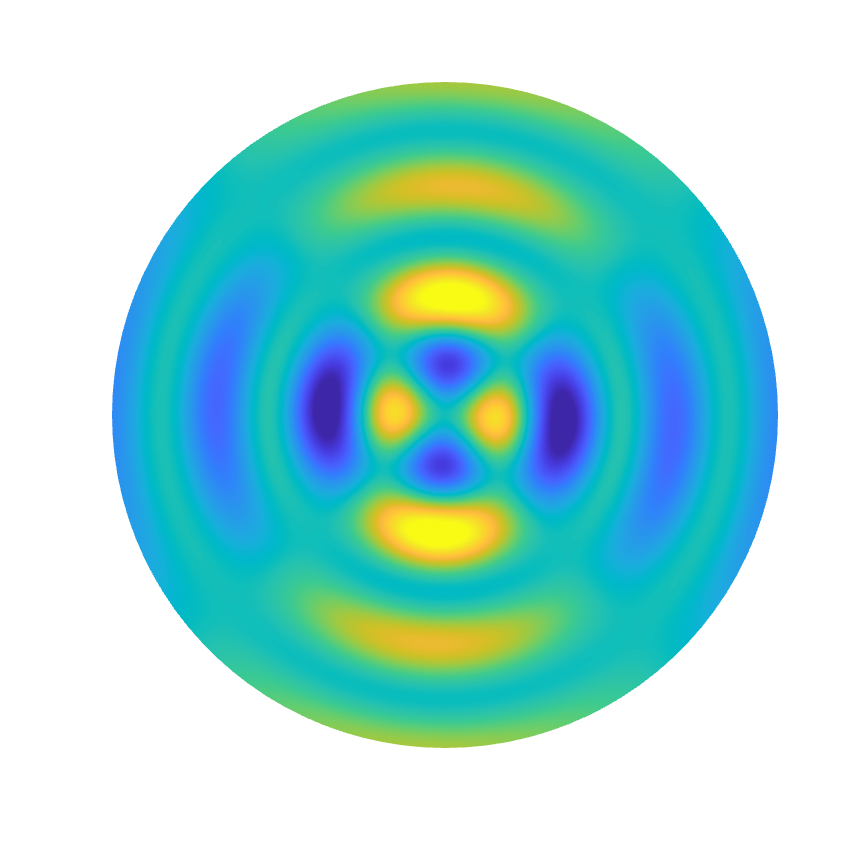}
    \end{subfigure}
    \begin{subfigure}[b]{.2\textwidth}
      \centering
      \includegraphics[width=\textwidth]{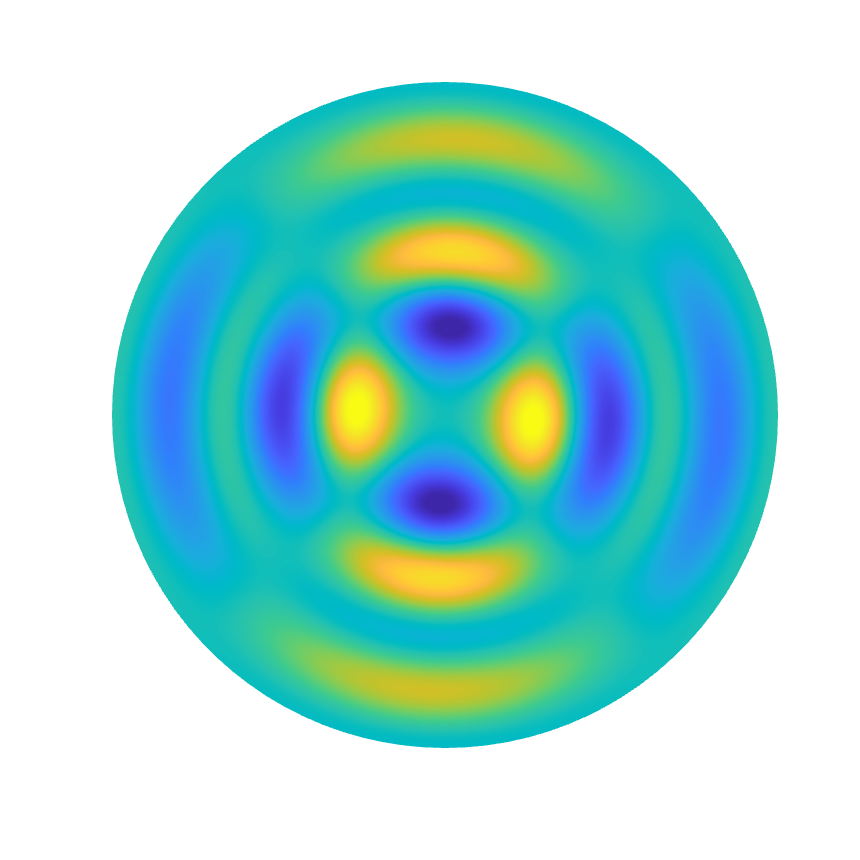}
    \end{subfigure}
    \begin{subfigure}[b]{.2\textwidth}
      \centering
      \includegraphics[width=\textwidth]{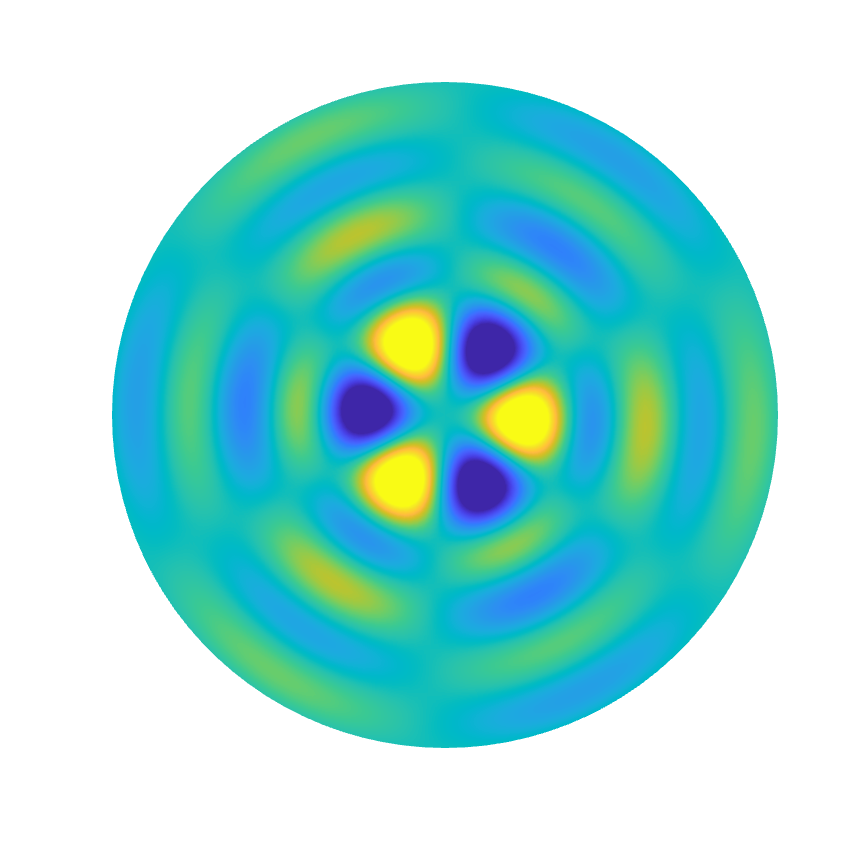}
    \end{subfigure}
    \begin{subfigure}[b]{.2\textwidth}
      \centering
      \includegraphics[width=\textwidth]{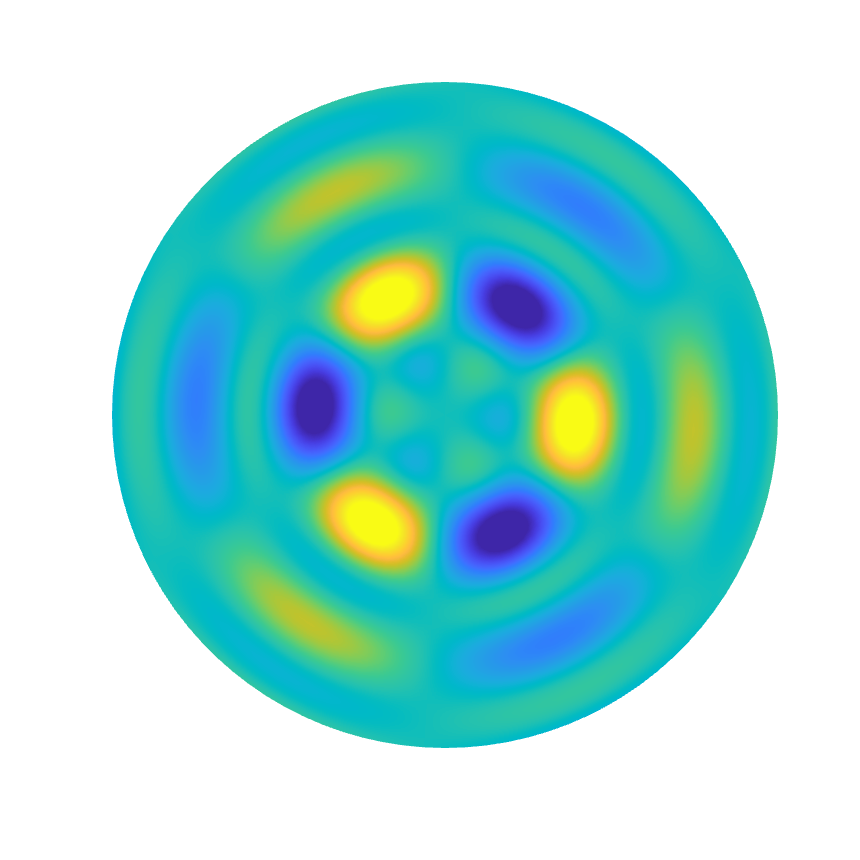}
    \end{subfigure}
    \begin{subfigure}[b]{.2\textwidth}
      \centering
      \includegraphics[width=\textwidth]{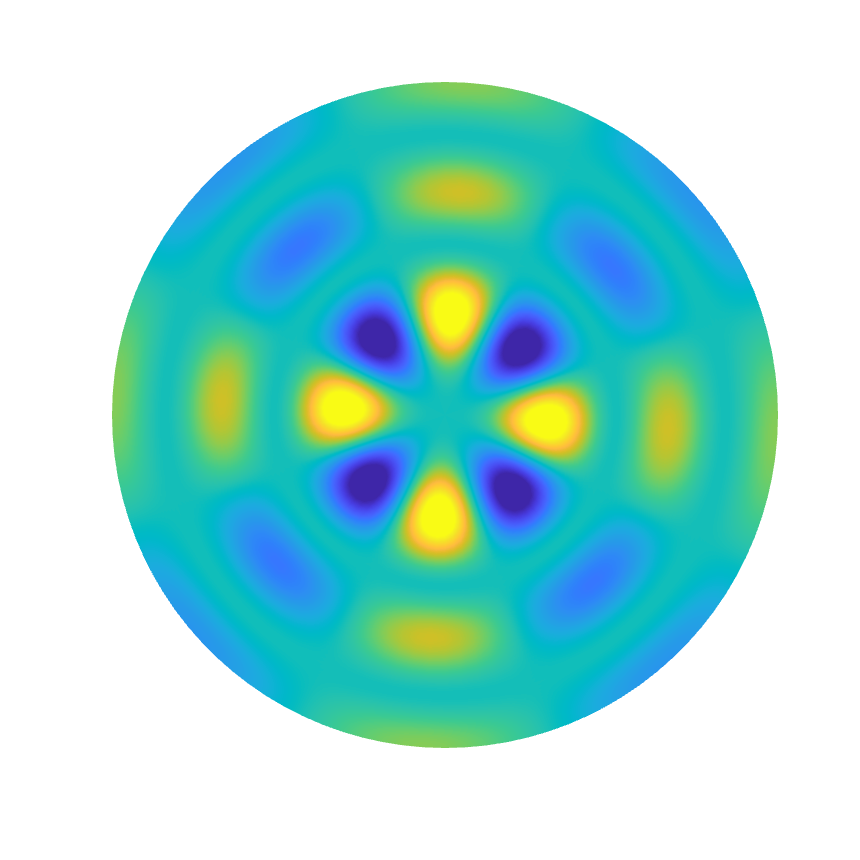}
    \end{subfigure}
    \begin{subfigure}[b]{.2\textwidth}
      \centering
      \includegraphics[width=\textwidth]{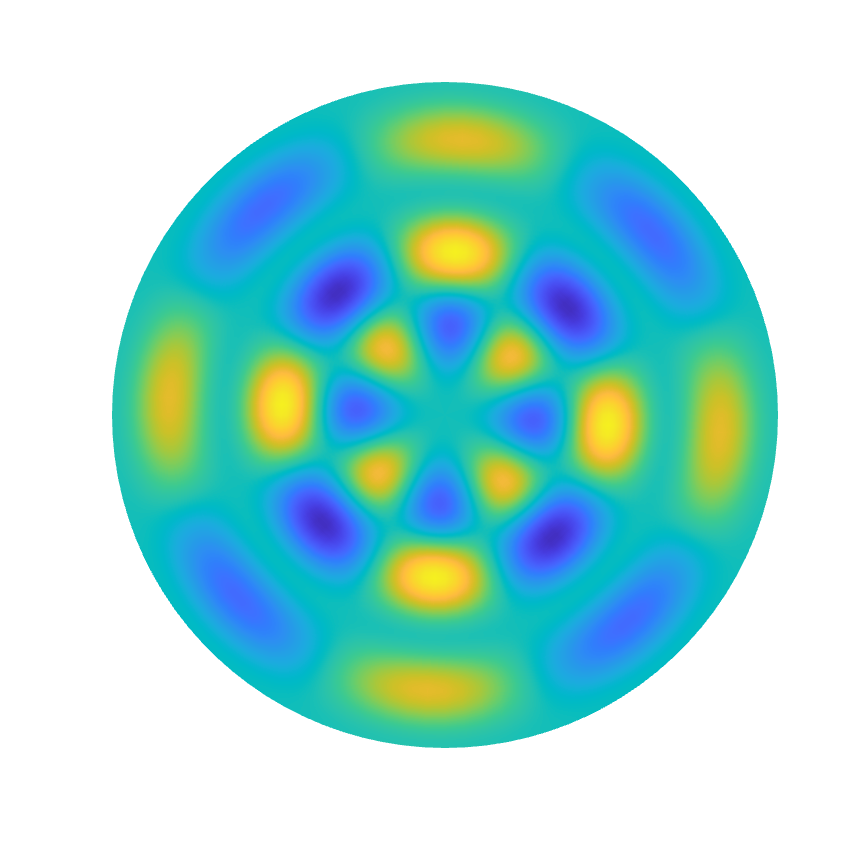}
    \end{subfigure}
    \begin{subfigure}[b]{.2\textwidth}
      \centering
      \includegraphics[width=\textwidth]{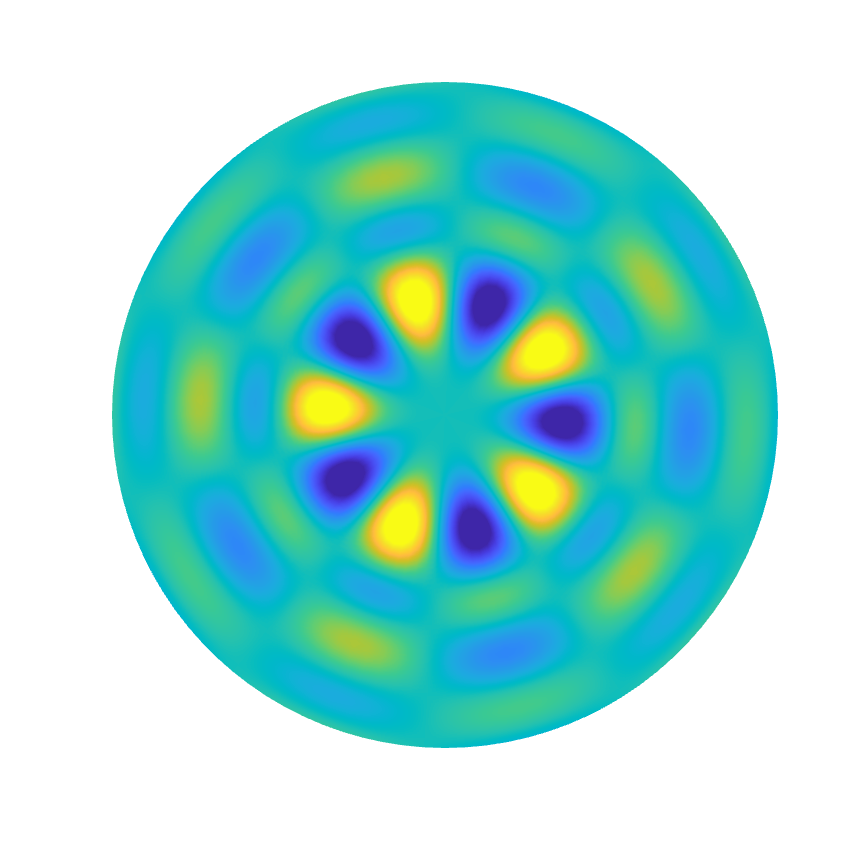}
    \end{subfigure}
    \begin{subfigure}[b]{.2\textwidth}
      \centering
      \includegraphics[width=\textwidth]{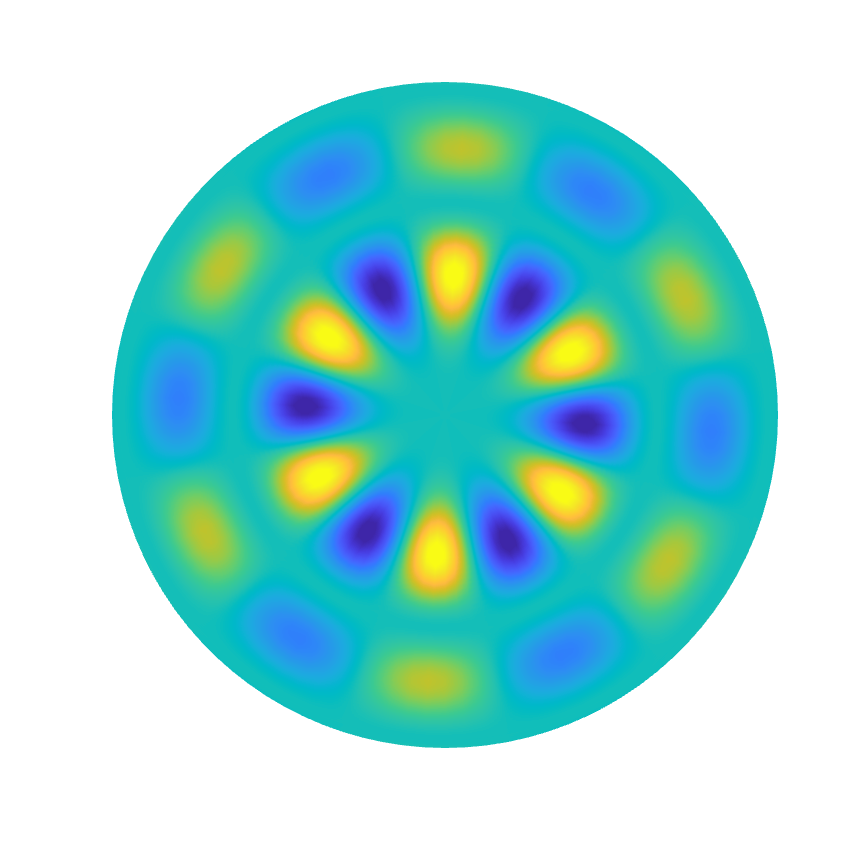}
    \end{subfigure}
    \caption{The real parts of the basis functions $(\Psi_{j,k})_{j\in\{0,\ldots,2N\}, k\in\{0,\ldots,N-\lceil j/2\rceil\}}$ for $\kappa R=10$ and $N=3$.
    First row:~$\Psi_{0,0}$, $\Psi_{0,1}$, $\Psi_{0,2}$ , $\Psi_{0,3}$.
    Second row:~$\Psi_{1,0}$, $\Psi_{1,1}$, $\Psi_{1,2}$ , $\Psi_{2,0}$.
    Third row:~$\Psi_{2,1}$, $\Psi_{2,2}$, $\Psi_{3,0}$ , $\Psi_{3,1}$.
    Fourth row:~$\Psi_{4,0}$, $\Psi_{4,1}$, $\Psi_{5,0}$ , $\Psi_{6,0}$. The color scale is the same in all subfigures.
    } 
    \label{fig:Psi_kR5}
\end{figure}
\hfill$\lozenge$
\end{example}

\section{Regularization of the infinite-dimensional systems}
\label{sec:Regularization}

As elaborated in \cite[Lmm.~2.4]{GriSch24}, almost all expansion coefficients $(a_{m,n})_{m,n \in \Z}$ of the given far field data are close to zero and thus negligible. More precisely, it is reasonable to estimate that
\begin{equation}
\label{eq:truncation}
    a_{m,n} \,=\, 0
    \qquad \text{for } |m|,|n|>N
\end{equation}
with $N$ chosen slightly larger than $\kappa R$.
The aforementioned work also provides explicit bounds for the approximation error introduced at this stage.
The nonzero structure of the expansion coefficients for the observed far field data under such a truncation is visualized in Figure~\ref{fig:DataMatrix_GSO_Error}~(left), with only the elements needed for \eqref{eq:forward_substitution} included (i.e.,~without employing the noise filtering step in \eqref{eq:noise_filt}).
\begin{figure}[t]
  \centering
  \includegraphics[width=.33\textwidth]{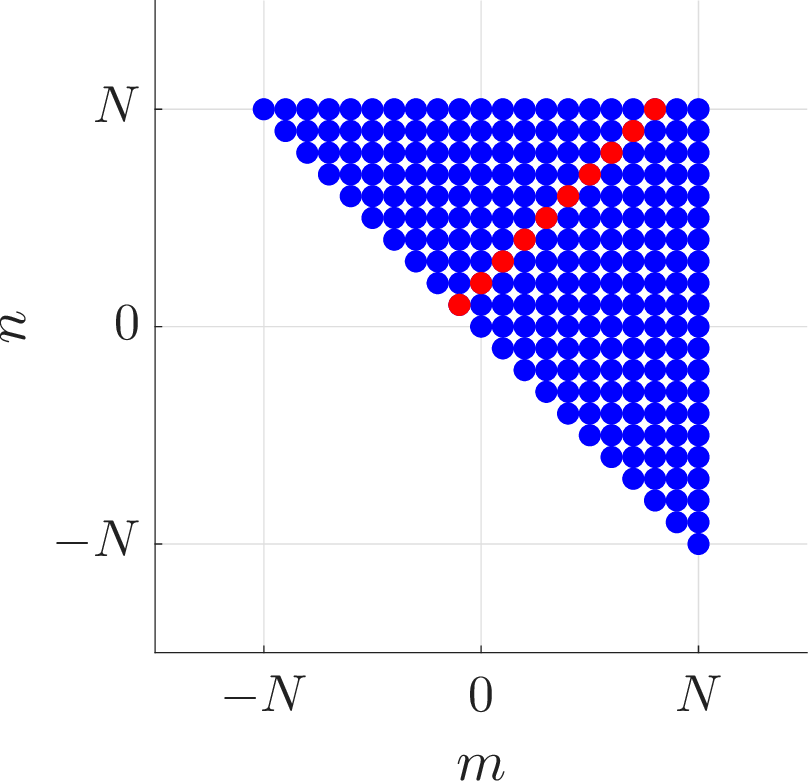}
  \hfill
  \includegraphics[width=.60\textwidth]{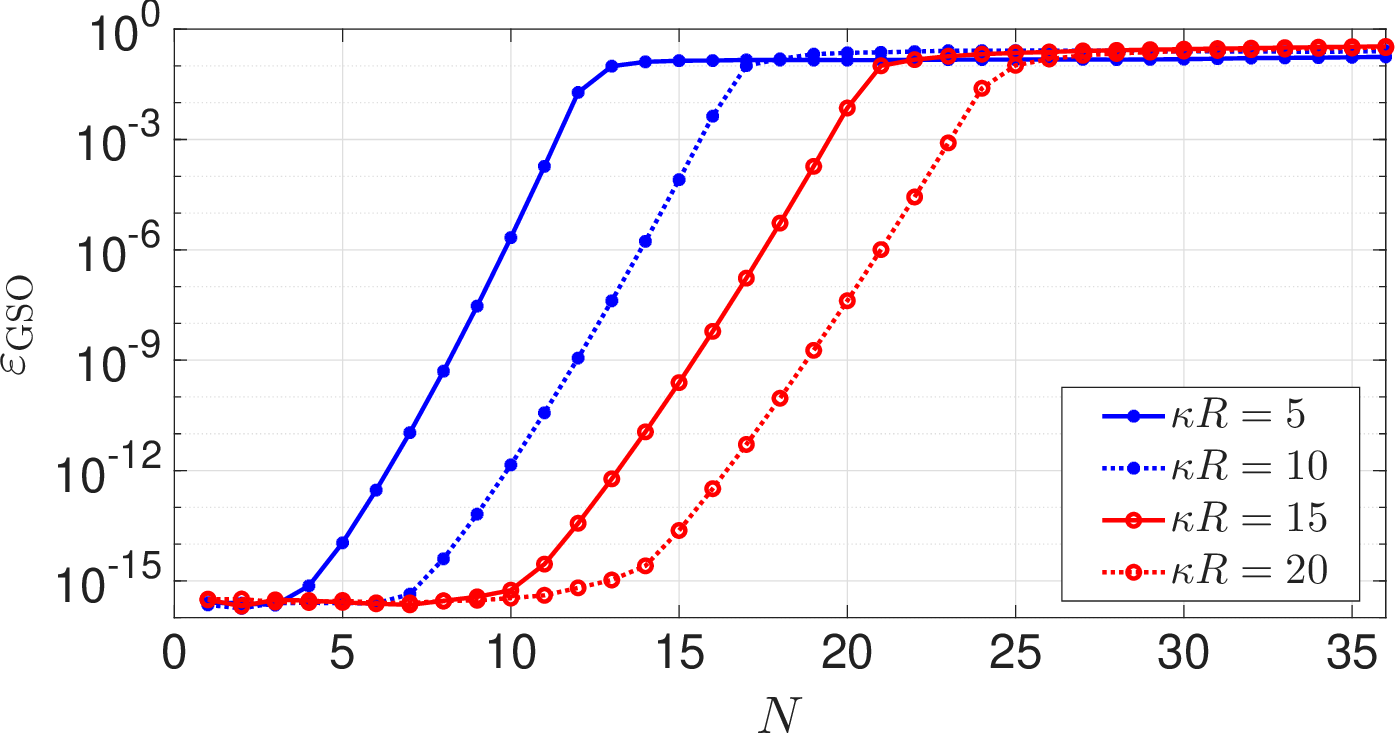}
  \caption{Left: Expansion coefficients $(a_{m,n})_{|m|,|n|\leq N}$ that are taken into account in the inversion (cf.~\eqref{eq:forward_substitution}) after introducing the truncation index $N$ (blue), with a representative diagonal $(a_{m,m-j})_{|m|,|m-j|\leq N}$ corresponding to the angular index $j=-2$ highlighted (red).
  Right: Orthonormalization error \eqref{eq:GSO_error} in the Gram--Schmidt process for different values of $\kappa R$ as functions of the truncation index $N$.
  } 
  \label{fig:DataMatrix_GSO_Error}
\end{figure}
To make the effect on the individual diagonals explicit, we rewrite this assumption as
\begin{equation*}
    a_{m+\lceil j/2\rceil, m-\lfloor j/2\rfloor} \,=\, 0 
    \qquad
    \begin{cases}
        \text{for } |j|>2N \,, \\
        \text{for } -2N \leq j < 0 \,,\, m>N+\lfloor j/2\rfloor=N-\lceil|j|/2\rceil\,, \\
        \text{for } 0 \leq j \leq 2N \,,\, m>N-\lceil j/2\rceil=N-\lceil|j|/2\rceil\,
    \end{cases}
\end{equation*}
for $m\in\N_0$.
Consequently, after introducing a truncation index $N$, only $4N+1$ systems for $j\in\{-2N,\ldots,2N\}$ are to be solved instead of the infinite number of systems in \eqref{eq:infinite_system}, and for a particular angular index $j$, the corresponding system becomes $(N-\lceil|j|/2\rceil+1)$-dimensional, cf.~Example~\ref{ex:Zernike}.

For $j\in\{-2N,\ldots,2N\}$, we define the truncated vectors (cf.~\eqref{eq:full_vectors})
\begin{equation*}
	\bfc^{j,N} \,:=\, \big[c_{j,k-1}\big]_{k=1}^{N+1-\lceil|j|/2\rceil}\,,\;
	\bfa^{j,N} \,:=\, \big[a_{(m-1)+\lceil j/2\rceil,(m-1)-\lfloor j/2\rfloor} \big]_{m=1}^{N+1-\lceil|j|/2\rceil} \,\in\, \C^{N+1-\lceil|j|/2\rceil}
\end{equation*}
and the lower triangular finite-dimensional system matrix
\begin{equation*}
    \bfF^{j,N} \, = \, \big[F^j_{m,k} \big]_{m,k=1}^{N+1-\lceil|j|/2\rceil} \,\in\, \C^{(N+1-\lceil|j|/2\rceil)\times(N+1-\lceil|j|/2\rceil)} \, .
\end{equation*}
Solving the resulting finite-dimensional systems
\begin{equation}
\label{eq:truncated_system}
    \bfF^{j,N}\bfc^{j,N} \,=\, \bfa^{j,N}, \qquad j\in\{-2N,\ldots,2N\},
\end{equation}
in place of the infinitely many infinite-dimensional systems of type \eqref{eq:infinite_system} constitutes the first step of our regularization scheme. If no further regularization is introduced, the forward substitution formula \eqref{eq:forward_substitution} still applies to solving the individual systems in \eqref{eq:truncated_system} for $k \in \{0, \dots, N-\lceil|j|/2\rceil \}$.

\begin{remark}
\label{rem:Choice_N}
    The choice of the truncation index $N$ involves a trade-off between two competing criteria.
    It should be large enough to retain the relevant information in the observed far field data, yet not so large that the Gram--Schmidt orthonormalization becomes numerically unstable.
    Indeed, due to the super-exponential decay of the Bessel function $J_n(t)$ for fixed~$t>0$ as~${n\rightarrow\infty}$, choosing $N$ too large leads to divisions by values that are numerically indistinguishable from zero during the normalization step.
    To illustrate this phenomenon, we construct for fixed $\kappa R>0$ and $N\in\N$ the corresponding $(N+1)^2$ Gram--Schmidt basis functions~${(R_k^{j})_{j=0,\ldots,2N, k=0,\ldots,N-\lceil j/2\rceil}}$ as defined by \eqref{eq:def_radial_basis}.
    For the numerical integration, we use $N_r=100$ Gauss–Legendre quadrature nodes and weights $(r_i,\omega_i)_{i=0,\ldots,N_r-1}$.
    By arranging the involved function evaluations for each $j$ as elements of a matrix, we obtain
    \begin{equation*}
        \bfQ^{j,N}
        \,:=\, 
        \begin{pmatrix}
            R_0^{j}(r_0) & R_1^{j}(r_0) & \cdots & R_{N-\lceil j/2\rceil}^{j}(r_0) \\
            R_0^{j}(r_1) & R_1^{j}(r_1) & \cdots & R_{N-\lceil j/2\rceil}^{j}(r_1) \\
            \vdots & \vdots & \ddots & \vdots \\
            R_0^{j}(r_{N_r-1}) & R_1^{j}(r_{N_r-1}) & \ldots & R_{N-\lceil j/2\rceil}^{j}(r_{N_r-1})
        \end{pmatrix}
        \in \R^{N_r\times (N+1-\lceil j/2\rceil)}\,,
    \end{equation*}
    whose columns are ideally orthonormal with respect to the inner product defined by the diagonal weight matrix $\bfW=\diag(\omega_{0}r_0,\ldots,\omega_{N_r-1}r_{N_r-1})\in\R^{N_r\times N_r}$.
    Consequently, we may quantify how much the constructed basis functions deviate on average from forming orthonormal systems via evaluating the mean error
    \begin{equation}
    \label{eq:GSO_error}
        \varepsilon_{\text{GSO}} \,:=\, \frac 1 {N+1} \sqrt{\sum_{j=0}^{2N}\big\|(\bfQ^{j,N})^\top \bfW\bfQ^{j,N} -\bfI_{N+1-\lceil j/2\rceil} \big\|^2_{\mathrm F}}
    \end{equation}
    where $\bfI_{N+1-\lceil j/2\rceil}\in\R^{(N+1-\lceil j/2\rceil)\times (N+1-\lceil j/2\rceil)}$ is the identity matrix and $\|\ph\|_{\mathrm F}$ denotes the Frobenius norm.
    This error is shown in Figure~\ref{fig:DataMatrix_GSO_Error}~(right) for $N\in\{1,\ldots,40\}$ and $\kappa R\in\{5,10,15,20\}$ on a semi-logarithmic scale.
    All error curves exhibit a regime of very low error for small $N$, followed by a sharp increase near $N \approx \kappa R$, after which they plateau.
    This behavior suggests that choosing $N$ significantly larger than $\kappa R$ is not reasonable.
    The rule $N:=\lceil \rme\kappa R/2\rceil$ used in \cite{GriSch24} turns out to be too large for our purposes, but $N:=\lceil\kappa R\rceil$ proves to be a reliable choice in our numerical experiments.
    \hfill$\lozenge$
\end{remark}

The truncated systems~\eqref{eq:truncated_system} can be rearranged into a single block-diagonal system
\begin{equation}
\label{eq:block_system}
	\bfF^N \bfc^N \,=\, \bfa^N
\end{equation}
by introducing
\begin{align*}
\bfF^N \,&:=\, \diag \!\big(\bfF^{-2N,N},\ldots,\bfF^{2N,N} \big) \,\in\, \C^{M\times M}\,,  \\[1mm]
	\bfa^N \,&:=\, \big[\bfa^{-2N,N}; \ldots;\bfa^{2N,N}\big] \,\in\, \C^M\,, \\[1mm]
	\bfc^N \,&:=\, \big[\bfc^{-2N,N};\ldots ;\bfc^{2N,N} \big] \,\in\, \C^M\,,
\end{align*}
where semicolon denotes vertical concatenation and 
\begin{equation*}
	M \,:=\, \sum_{j=-2N}^{2N} \left(N+1-\left\lceil\tfrac{|j|} 2\right\rceil\right)
	\,=\, (N+1)+4 \sum_{j=1}^N j
	\,=\, (N+1)(2N+1)\,.
\end{equation*}
In our numerical experiments documented in Section \ref{sec:NumericalExamples} below, we further employ a truncated singular value decomposition (SVD) for the system \eqref{eq:block_system} to handle full nonlinear far field data and/or additive noise.
The block-diagonal structure of the system matrix $\bfF^N$ allows its SVD to be completely characterized by SVDs of the individual small blocks $\bfF^{j,N}$, $j=\{-2N,\ldots,2N\}$, thereby avoiding the need to compute a high dimensional SVD. Indeed, assuming the standard computational complexity for the SVD, the cost of forming an SVD for a general system matrix of dimension $M$ behaves as $\mathcal{O}(M^3) = \mathcal{O}(N^6)$, whereas computing SVDs for all the subsystems in \eqref{eq:truncated_system} costs only $\mathcal{O}(N^4)$ operations.
See~\cite{AutGarHirHvy24} for more detailed analysis on applying a truncated SVD to a similar block-diagonal system in the framework of EIT. 

\section{Numerical examples}
\label{sec:NumericalExamples}

In this section, we demonstrate the functionality of our proposed method by numerical examples for both Born far field data and full far field data, with and without additive noise. Moreover, we compare the performance of our method to the recently introduced low-rank method for solving the inverse Born scattering problem~\cite{ZhoAudMenZha26} and with the MATLAB's built-in NUFFT, as it is described in~\cite{DutRoh93}.

We assume the ability to sample the the Born far field data at $2L\in\N$ equiangular illumination and measurement directions. That is, we assume the availability of the (noisy) matrix
\begin{equation}
\label{eq:observation}
    \frac {\pi} L \left[ u_B^\infty(\xhat_m,\idir_n) \right]_{1\leq m,n\leq 2L} \in\C^{2L\times 2L}
\end{equation}
with
\begin{equation*}
    \xhat_l \,=\, \idir_l \,=\, (\cos(\varphi_l),\sin(\varphi_l))^\top
    \,,\qquad
    \varphi_l \,=\, \frac {\pi (l-1)} L
    \,,\qquad
    l=1,\ldots,2L\,.
\end{equation*}
Here, $2L\in\N$ is chosen large enough to resolve all relevant information in the far field data, i.e., $L$ has to be larger than $\kappa$ times the radius of the smallest origin-centered ball containing the whole scatterer $\Omega$ (cf.~\eqref{eq:truncation} and \cite{GriSch24}).
The two-dimensional fast Fourier transform of the matrix 
\begin{equation}
\label{eq:meas_a}
    \frac \pi L\left[ \rme^{-\rmi\kappa\bfc\cdot(\idir_n-\xhat_m)} u_B^\infty(\xhat_m,\idir_n)\right]_{1\leq m,n\leq 2L}\in\C^{2L\times 2L}
\end{equation}
then yields an approximation for the expansion coefficients $(a_{m,n})_{-L\leq m,n\leq L-1}$ as defined in~\eqref{eq:exp_coeff_FF}. If we consider reconstruction from full far field data, then $u_B^\infty(\xhat_m,\idir_n)$ in \eqref{eq:meas_a} is replaced by $u^\infty(\xhat_m,\idir_n)$ from \eqref{eq:full_far_field}.

\begin{example}
\label{exa:three_discs}
As the first example, we consider the piecewise constant contrast
\begin{equation*}
    q=\chi_{B_{r_1}(\bfc_1)}-0.25\chi_{B_{r_2}(\bfc_2)}+0.5\chi_{B_{r_3}(\bfc_3)}\,,
\end{equation*}
where $\bfc_1=(-0.35,0.4)^\top$, $\bfc_2=(-0.1,-0.45)^\top$, $\bfc_3=(0.45,0.1)^\top$, $r_1=r_2=0.3$ and $r_3=0.2$.
We assume the prior knowledge that the ROI $B_1(\bfzero)$ encloses the support of the contrast and consider the wave number $\kappa = 30$. The exact contrast restricted to the ROI is shown in Figure~\ref{fig:exact_contrasts}~(left).
\begin{figure}[t]
    \centering
    \begin{subfigure}[b]{.36\textwidth}
      \centering
      \includegraphics[width=\textwidth]{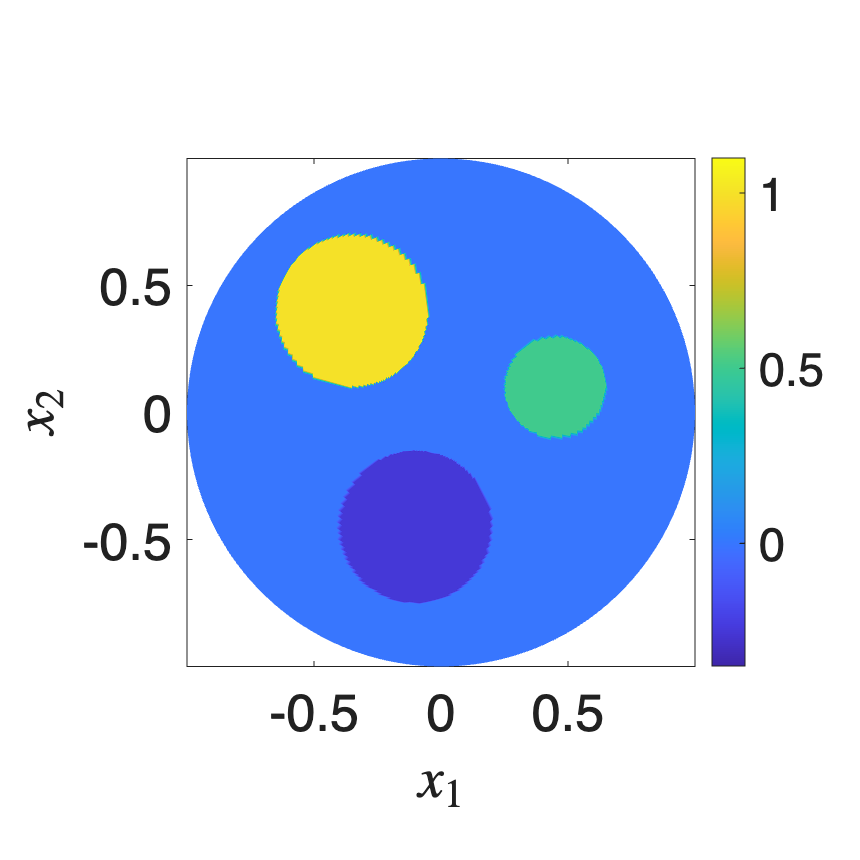}
    \end{subfigure}
    \hspace*{4em}
    \begin{subfigure}[b]{.37\textwidth}
      \centering
      \includegraphics[width=\textwidth]{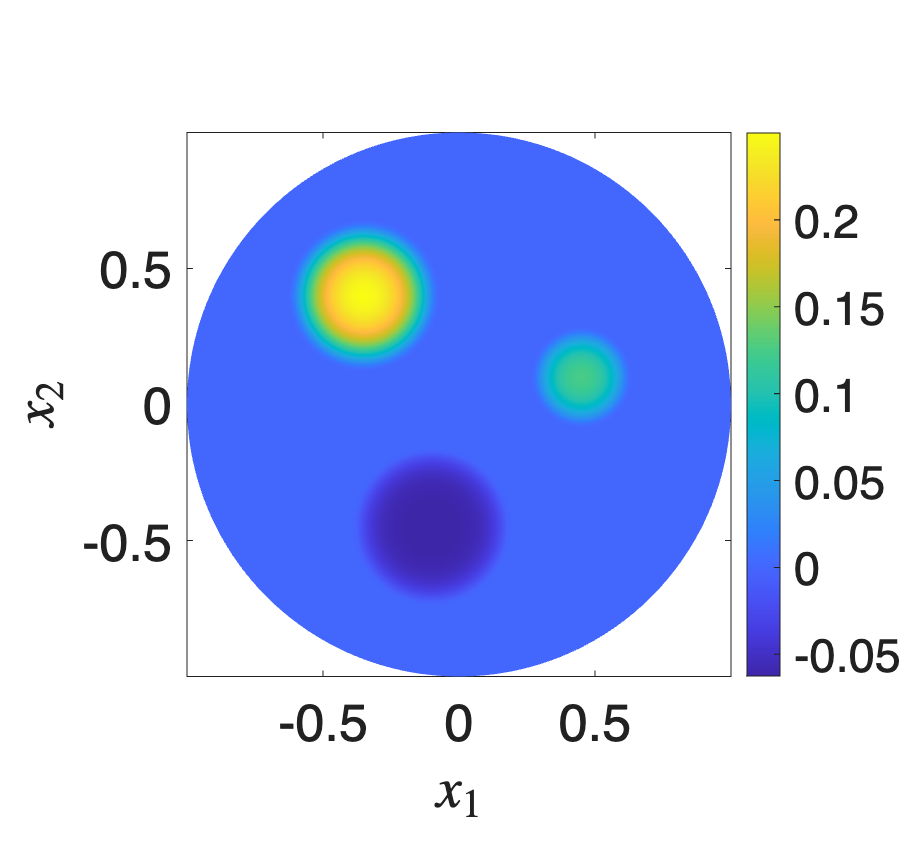}
    \end{subfigure}
    \caption{Exact contrast in Example \ref{exa:three_discs} (left) and in Example \ref{exa:three_discs_smooth} (right).}
    \label{fig:exact_contrasts}
\end{figure}
In this setting, the Born far field data can be expressed analytically: by separately considering the expansion coefficients of the Born far field operator deduced in  Example~\ref{exa:analytic_example} for the three discoidal inclusions, it follows that 
\begin{equation*}
    u_B^\infty(\xhat_m,\idir_n) \,=\, \sum_{i=1}^3 \sum_{l=1}^\infty 2\pi^2(\kappa r_i)^2
    \left( J_{l-1}^2(\kappa r_i)-J_{l-2}(\kappa r_i)J_l(\kappa r_i)\right)\bfe^{\bfc_i}_l(\xhat_m-\idir_n)
\end{equation*}
for $m,n=1,\ldots,2L$.
We choose $L=125$ and truncate the series at $250$ terms.
For expanding the contrast, we use $N_r=250$ Gauss--Legendre nodes in the radial direction and $N_\varphi=250$ equiangular nodes in the angular direction.

\begin{figure}[th!]
    \centering
    \begin{subfigure}[b]{.36\textwidth}
      \centering
      \includegraphics[width=\textwidth]{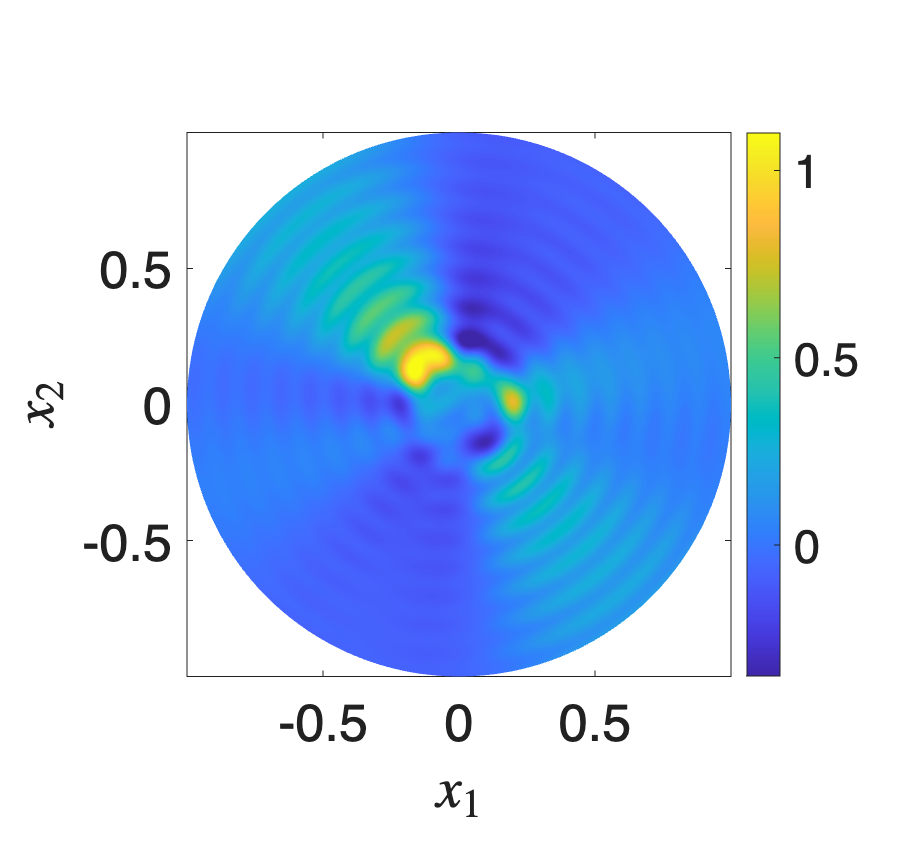}
    \end{subfigure}
    \hspace*{4em}
    \begin{subfigure}[b]{.36\textwidth}
      \centering
      \includegraphics[width=\textwidth]{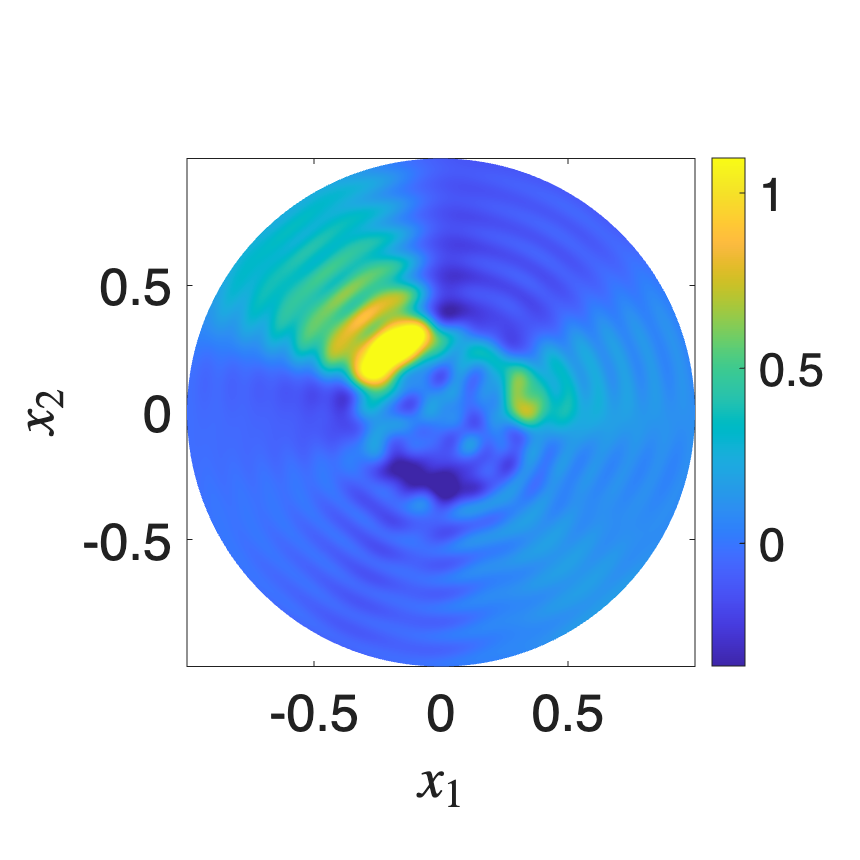}
    \end{subfigure}
    \begin{subfigure}[b]{.36\textwidth}
      \centering
      \includegraphics[width=\textwidth]{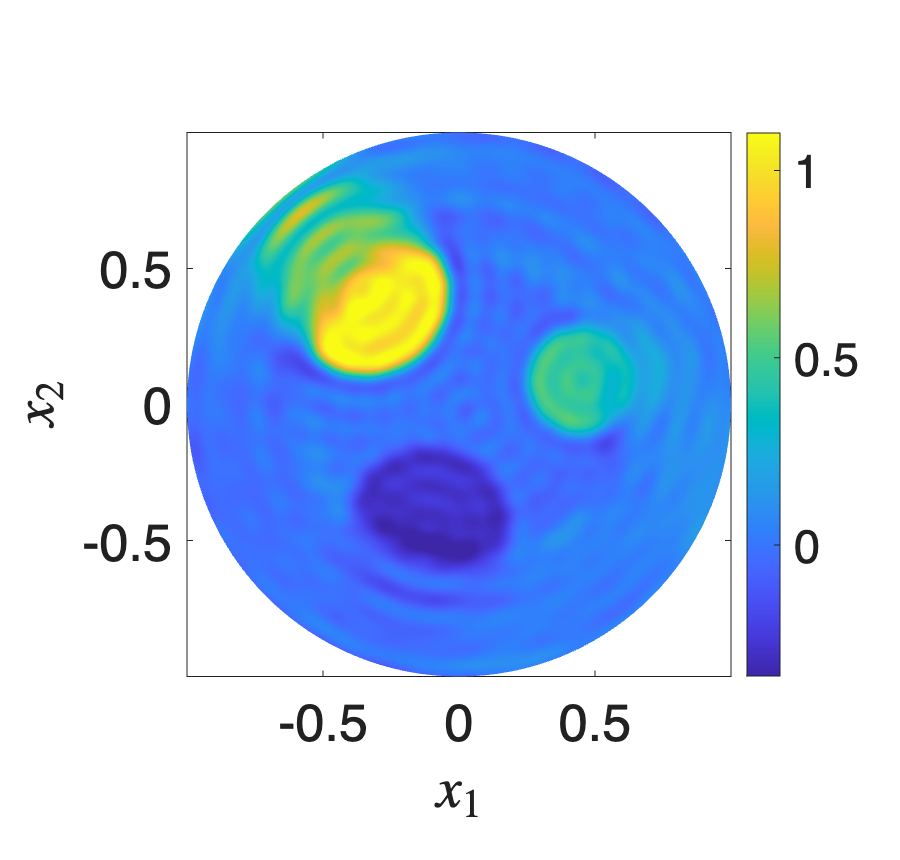}
    \end{subfigure}
    \hspace*{4em}
    \begin{subfigure}[b]{.36\textwidth}
      \centering
      \includegraphics[width=\textwidth]{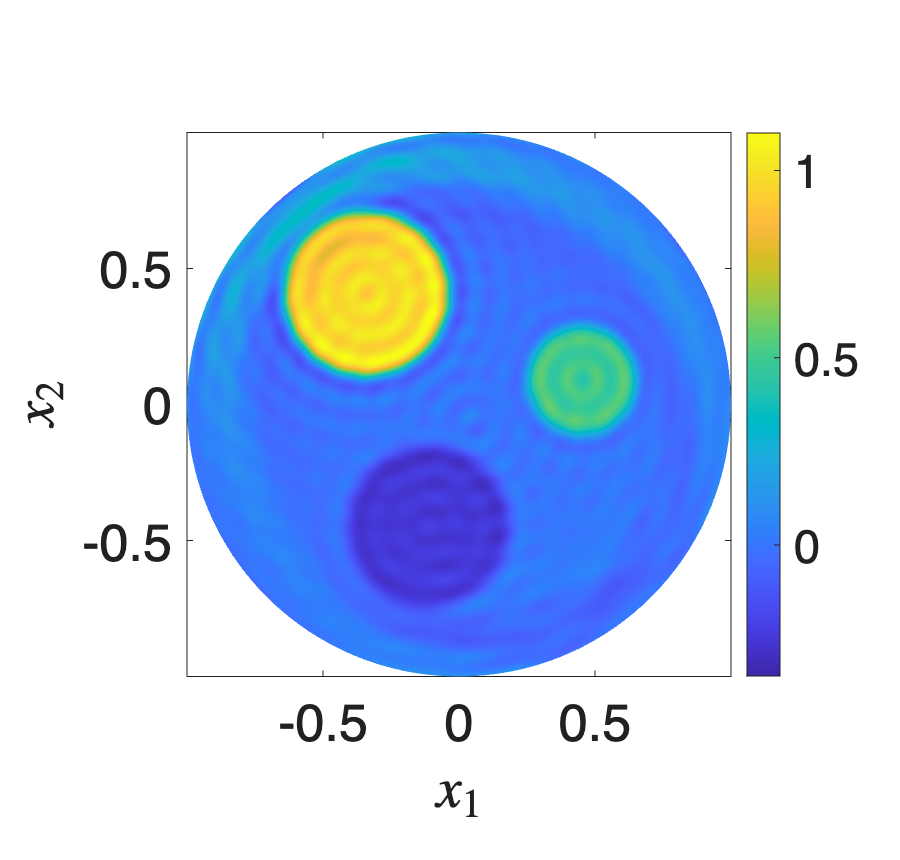}
    \end{subfigure}
    \begin{subfigure}[b]{.36\textwidth}
      \centering
      \includegraphics[width=\textwidth]{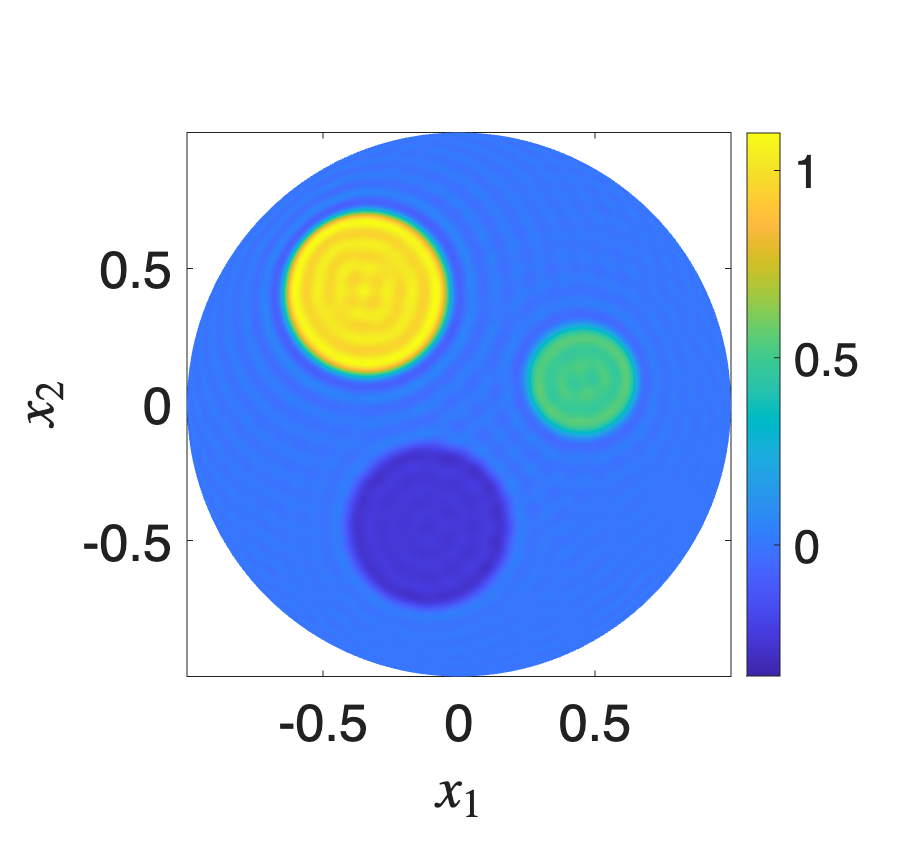}
    \end{subfigure}
    \hspace*{4em}
    \begin{subfigure}[b]{.36\textwidth}
      \centering
      \includegraphics[width=\textwidth]{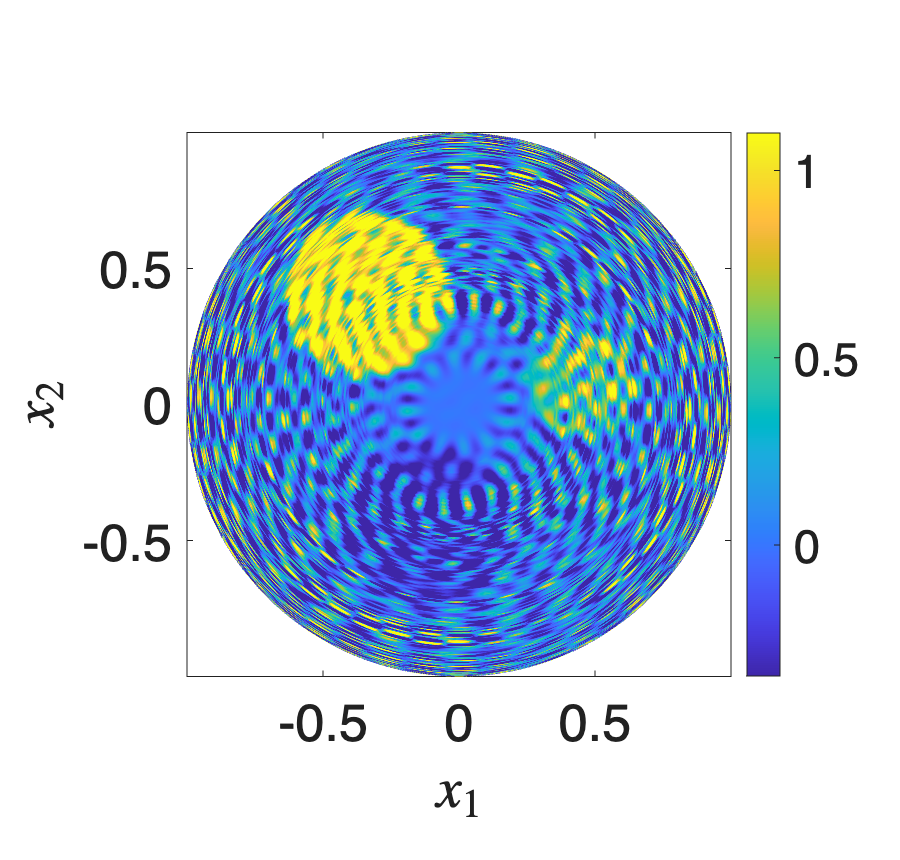}
    \end{subfigure}
    \caption{Example~\ref{exa:three_discs}. Reconstructed contrast from the exact Born far field data for different truncation indices.
    Top left: $N=5$ (too small).
    Top right: $N=10$ (too small).
    Middle left: $N=15$ (too small).
    Middle right: $N=20$ (slightly too small).
    Bottom left: $N=29$ (optimal).
    Bottom right: $N=31$ (too large).}
    \label{fig:reconstructions_disc_N}
\end{figure}
We first study the dependence of the reconstruction quality in the noise-free case on the truncation index $N \in \{1,\dots,35\}$ without employing truncated SVD for further regularization.
Selected reconstructions are shown in Figure~\ref{fig:reconstructions_disc_N}, and the corresponding relative $L^2$ reconstruction error over the ROI
is presented in Figure~\ref{fig:reconstructions_disc_error}~(left).
\begin{figure}[t]
    \centering
    \begin{subfigure}[b]{.45\textwidth}
      \centering
      \includegraphics[width=\textwidth]{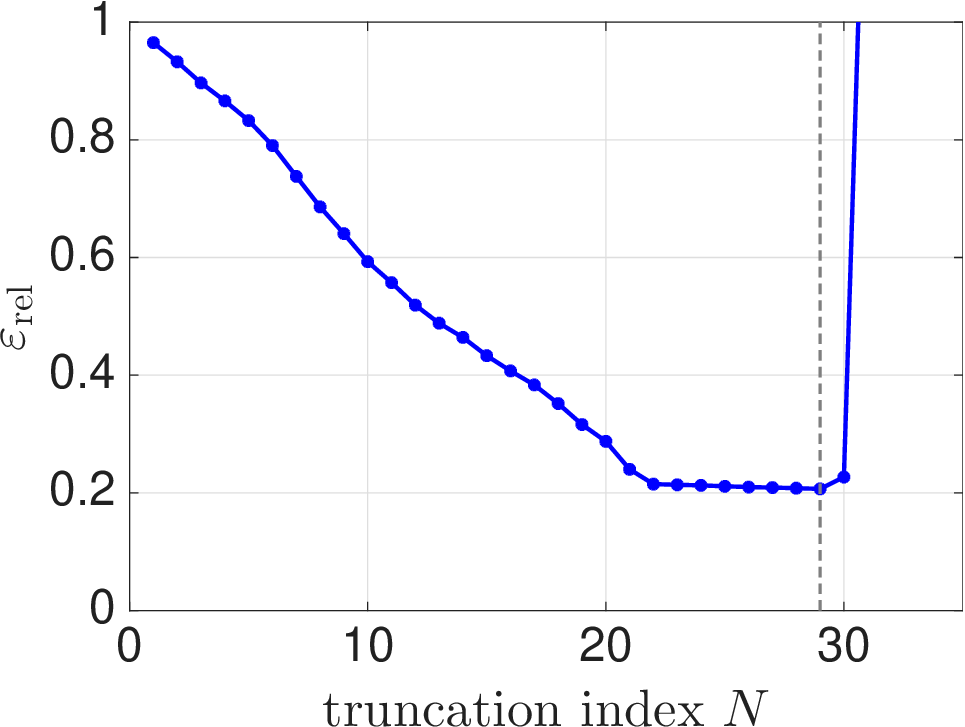}
    \end{subfigure}
    \hfill
    \begin{subfigure}[b]{.45\textwidth}
      \centering
      \includegraphics[width=\textwidth]{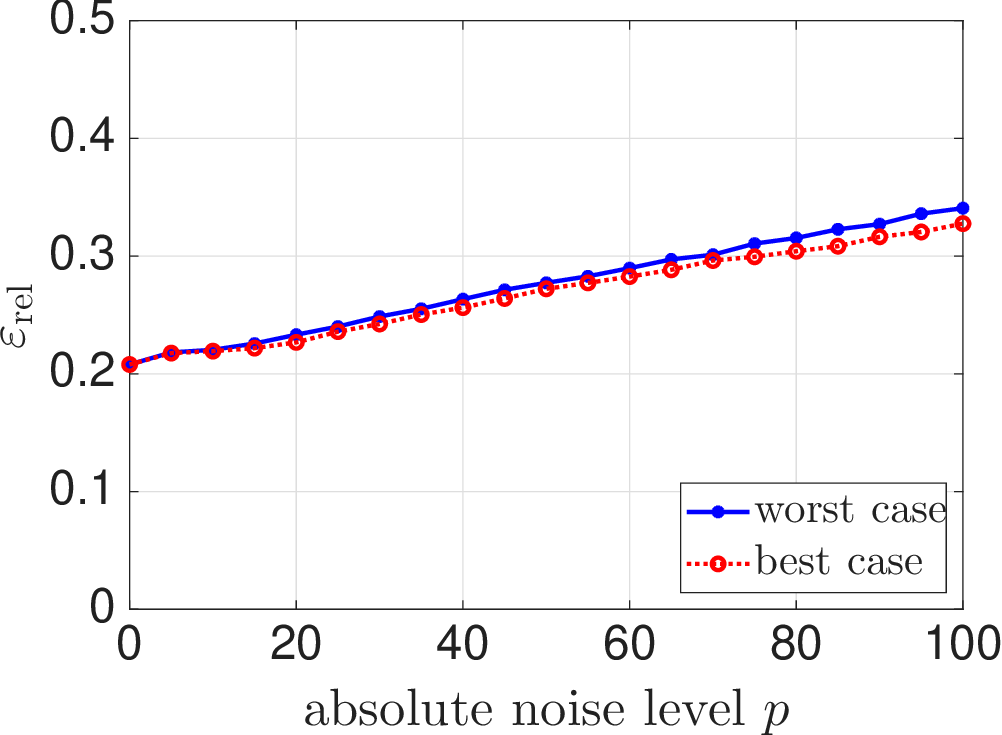}
    \end{subfigure}
    \caption{Example~\ref{exa:three_discs}. 
    Left: Relative $L^2$ reconstruction error $\varepsilon_{\mathrm{rel}}$ 
    as a function of the truncation index $N$ with exact Born far field data. The optimal choice $N=29$ is marked by a vertical line. 
    Right: The best and worst relative $L^2$ reconstruction errors 
    $\varepsilon_{\mathrm{rel}}$ as functions of the absolute noise level $p$ over $20$ runs with different realizations of noise.
    }
    \label{fig:reconstructions_disc_error}
\end{figure}
The error curve illustrates the trade-off as described in Remark~\ref{rem:Choice_N}:
choosing the truncation index $N$ too small leads to loss of relevant information, whereas choosing it too large results in dominance of the orthonormalization error introduced by the Gram--Schmidt process.
Interestingly, the region in which the contrast is reconstructed accurately expands gradually outward from the center when increasing $N$ until the minimal relative reconstruction error is achieved for $N=29\approx\kappa R$.
Figures~\ref{fig:reconstructions_disc_N} and \ref{fig:reconstructions_disc_error} reveal that our method cannot achieve arbitrarily accurate reconstructions since higher angular and radial modes are excluded from the computation by construction; the relative $L^2$ error remains above $20\%$ for all truncation indices.
In particular, discontinuities in the contrast, i.e., the jumps at the boundaries of the discs, cannot be reconstructed exactly.
However, the reconstructions  for $N = 20$ and $N = 29$ in Figure~\ref{fig:reconstructions_disc_N} can be considered visually satisfactory.
According to the reconstruction for $N=31$ in Figure~\ref{fig:reconstructions_disc_N} and the error plot in Figure~\ref{fig:reconstructions_disc_error}~(left), even a small increase in the truncation index $N$ beyond its optimal value leads to poor reconstructions.
Hence, without further regularization, the quality of the reconstruction is strongly influenced by the quality of the {\em a priori} known ROI.

\begin{figure}[t]
    \centering
    \begin{subfigure}[b]{.36\textwidth}
      \centering
      \includegraphics[width=\textwidth]{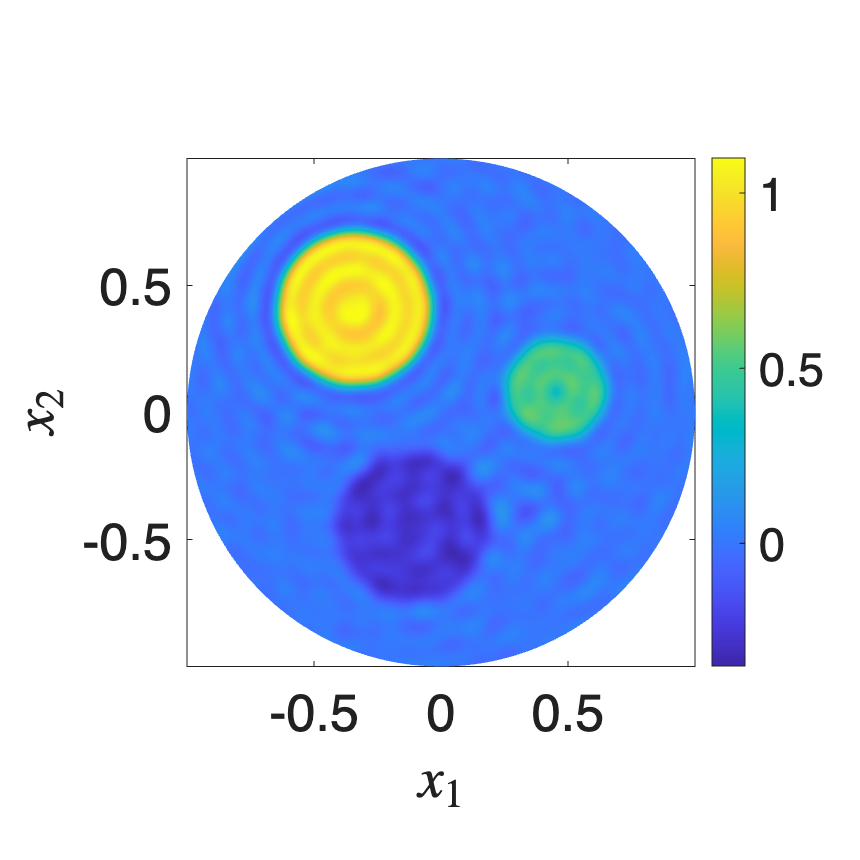}
    \end{subfigure}
    \hspace*{4em}
    \begin{subfigure}[b]{.36\textwidth}
      \centering
      \includegraphics[width=\textwidth]{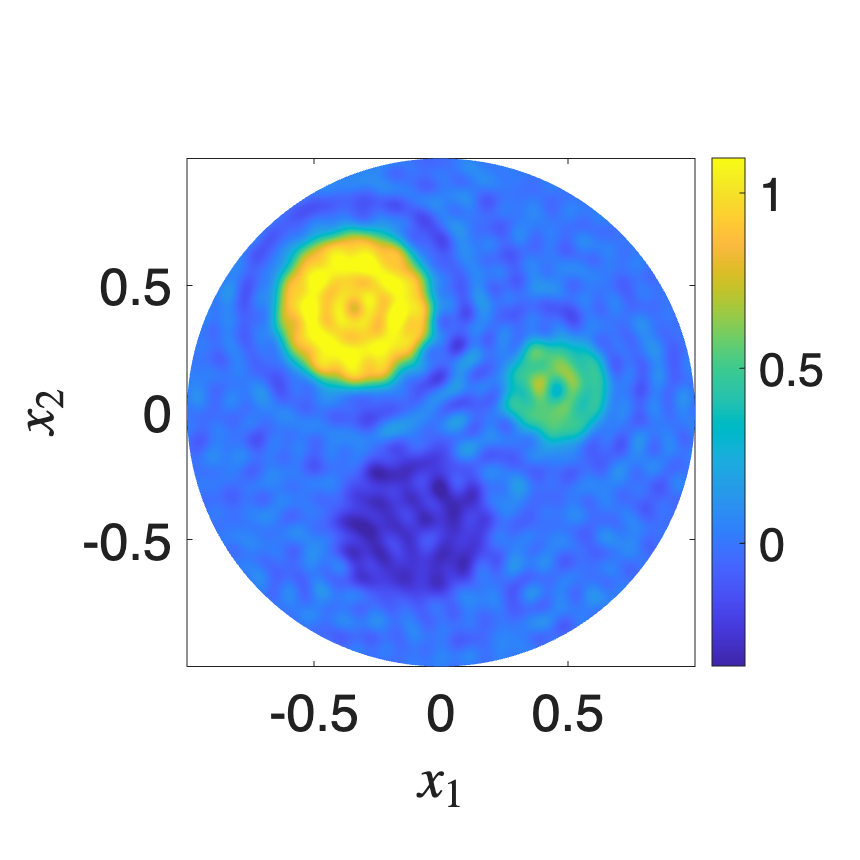}
    \end{subfigure}
    \caption{Example~\ref{exa:three_discs}. Worst reconstructed contrasts in the sense of relative $L^2$ error from noisy Born far field data over $20$ runs at the noise levels $p=20$ (left) and $p=80$ (right).}
    \label{fig:reconstructions_disc_noise}
\end{figure}
Next, we fix the truncation index $N=30$ and study the quality of our reconstructions when the observed far field data are contaminated by additional noise.
To this end, we add to each element of the exact data matrix in \eqref{eq:observation} complex noise whose real and imaginary parts are independently drawn from a zero-mean uniform distribution, with the standard deviation scaled {\em a posteriori} so that the Frobenius norm of the added noise matrix is $p\in\{0,5,\ldots,95,100\}$ per cent of the Frobenius norm of the exact data matrix. In what follows, we refer to this noise model by simply saying that the data contain $p$\% of noise. We use the truncated SVD together with the Morozov discrepancy principle with respect to the Euclidean norm as an additional regularization strategy, with the target vector carrying the noisy truncated coefficients of the Born far field operator (cf.~\eqref{eq:block_system} and \eqref{eq:meas_a}) and the employed noise level scaled appropriately to account for the amount of noise no longer present in the data after the truncation. We generate $20$ independent noise realizations for each $p$ and plot the relative $L^2$ errors of the resulting best and worst reconstructions in Figure~\ref{fig:reconstructions_disc_error}~(right).
The related worst-case reconstructions for $p=20$ and $p=80$ are presented in Figure~\ref{fig:reconstructions_disc_noise}. In the shown worst cases, $45\%$ and $24\%$ of the singular components are taken into account for $p = 20$ and $p=80$, respectively. Interestingly, our method turns out to be very robust to this form of additive noise since it is distributed across both high and low modes in the data, and thus a large fraction of the noise is filtered out by the introduction of the truncation index $N$. In the next example, we will see that this does not apply when the model discrepancy originates from multiple scattering effects included in full far field data.
\hfill$\lozenge$
\end{example}

\begin{example}
\label{exa:three_discs_smooth}
Next we study a smooth contrast of a similar geometric structure as in the previous example, with the aim to compare reconstructions obtained using Born far field data and full far field data as the input for our reconstruction algorithm. 
We examine how reconstructions from full far field data deteriorate with increasing multiple scattering effects, i.e., as the Born far field data become an increasingly inaccurate approximation of the full far field data observed in practice.
This is achieved by keeping the contrast $q$ fixed as shown in Figure \ref{fig:exact_contrasts}~(right) while gradually increasing the wavenumber from $\kappa=1$ to $\kappa=56$, which also leads to an increase in the truncation index from $N=1$ to $N=56$. When considering full far field data, we accompany the data truncation with truncated SVD with the noise level for the Morozov principle chosen (unrealistically) to be the Euclidean norm of the discrepancy in the data vector in \eqref{eq:block_system} between Born and full far field data; see~\eqref{eq:meas_a} and the comment succeeding it.
For evaluating both Born and full far field data in $2L=250$ equiangular observation and illumination directions, we use the fast Lippmann--Schwinger solver proposed by Vainikko~\cite{Vai97}.
For expanding the contrast, we choose $N_r=N_\varphi=310$. 

The corresponding relative $L^2$ reconstruction errors along with the relative Frobenius-based data error (cf.~\eqref{eq:observation}) induced by approximating the Born far field data with the full far field data are shown in Figure \ref{fig:reconstructions_discs_error_kappa}.
The reconstructed contrasts for $\kappa=11$ and for $\kappa=46$ are illustrated in Figure~\ref{fig:reconstructions_discs_kappa}. According to Figures~\ref{fig:reconstructions_discs_error_kappa} and \ref{fig:reconstructions_discs_kappa}, the best reconstruction using full far field data is obtained at $\kappa = 11$, with practically no visual difference in quality to the corresponding reconstruction based on Born far field data.
Although multiple scattering effects dominate at $\kappa = 46$, amounting to about 80\% of the magnitude of the linearized component, the reconstruction from full far field data still provides a clear picture of the positions and sizes of the three components of the scatterer, with the exterior shapes of two of the components still recognizable. However, the reconstruction does not capture the dynamic range of the target contrast, and it includes holes inside the scattering components.

\begin{figure}[t]
    \centering
    \begin{subfigure}[b]{.45\textwidth}
      \centering
      \includegraphics[width=\textwidth]{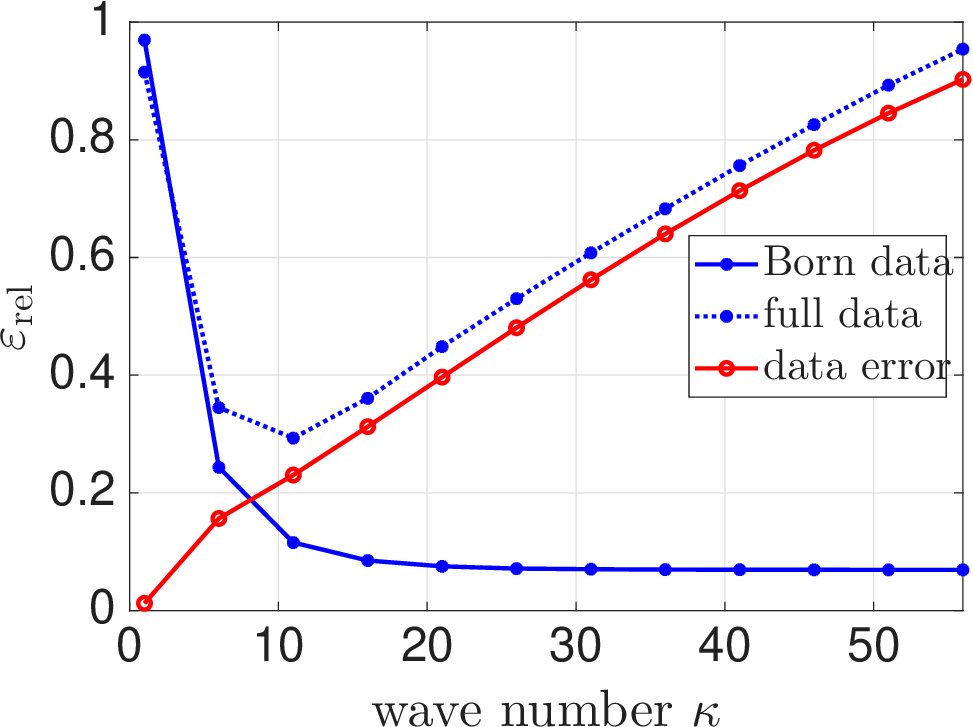}
    \end{subfigure}
    \caption{Example~\ref{exa:three_discs_smooth}. Relative $L^2$ reconstruction error $\varepsilon_{\mathrm{rel}}$ 
    as a function of the wave number $\kappa$ with the exact Born far field data and the exact full far field data as the inputs for the proposed method. The relative Frobenius-based approximation error  induced by replacing the Born far field data with the full far field data is also shown for reference (cf.~\eqref{eq:meas_a}).}
    \label{fig:reconstructions_discs_error_kappa}
\end{figure}
\begin{figure}[t]
    \centering
    \begin{subfigure}[b]{.36\textwidth}
      \centering
      \includegraphics[width=\textwidth]{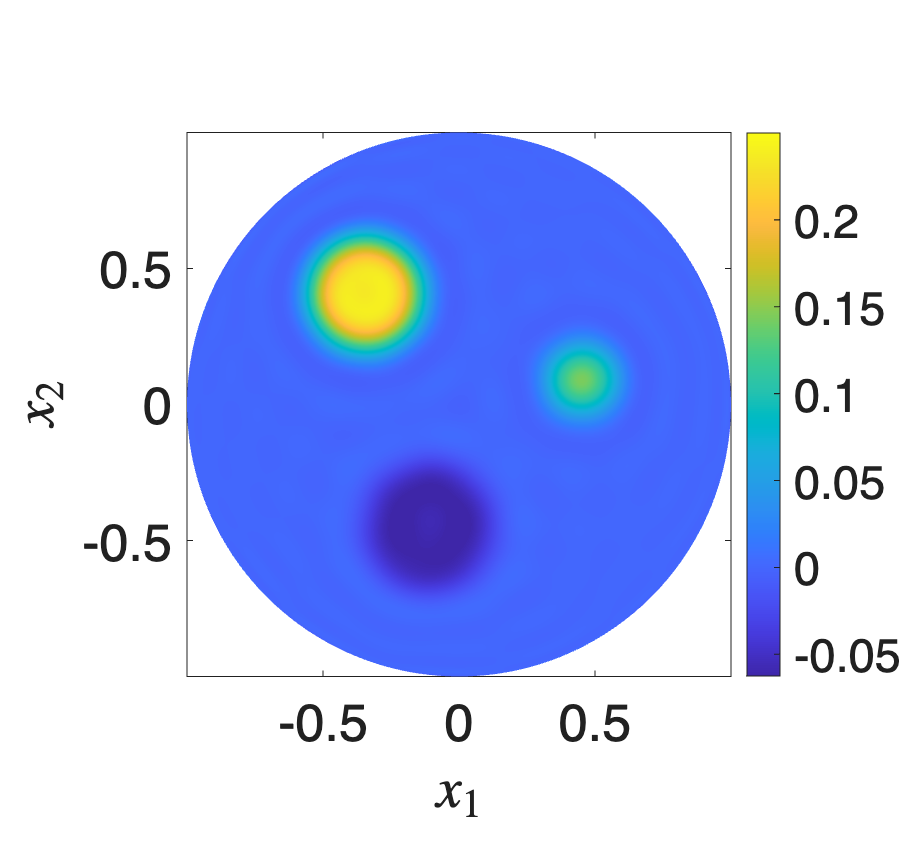}
    \end{subfigure}
    \hspace*{4em}
    \begin{subfigure}[b]{.36\textwidth}
      \centering
      \includegraphics[width=\textwidth]{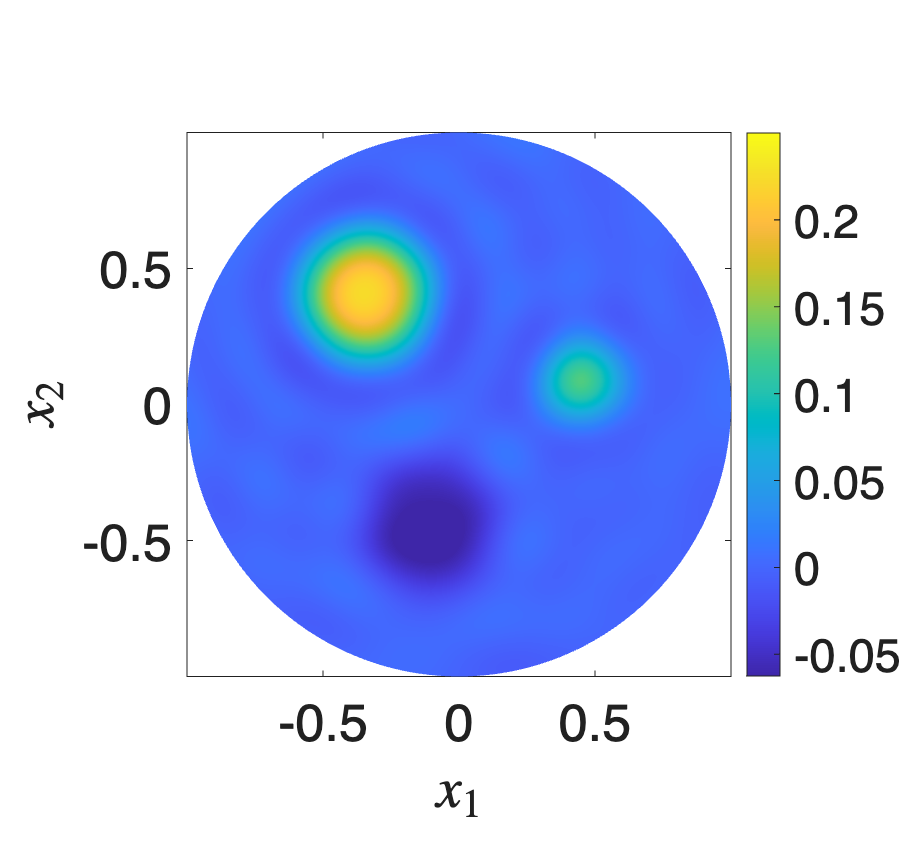}
    \end{subfigure}
    \begin{subfigure}[b]{.36\textwidth}
      \centering
      \includegraphics[width=\textwidth]{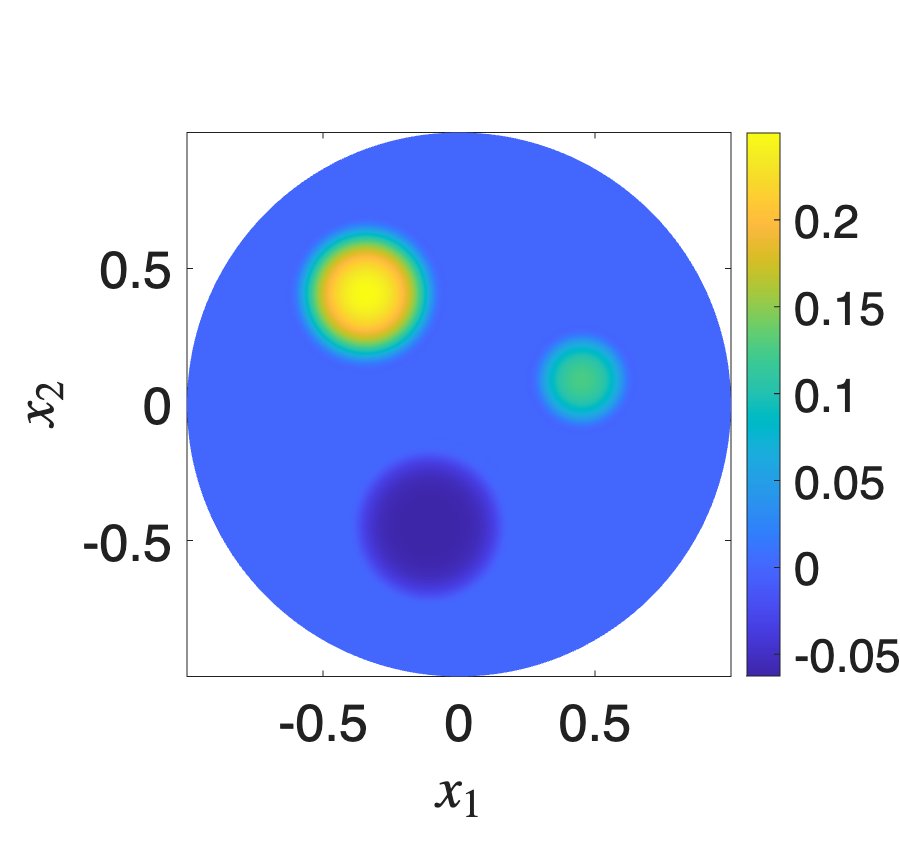}
    \end{subfigure}
    \hspace*{4em}
    \begin{subfigure}[b]{.36\textwidth}
      \centering
      \includegraphics[width=\textwidth]{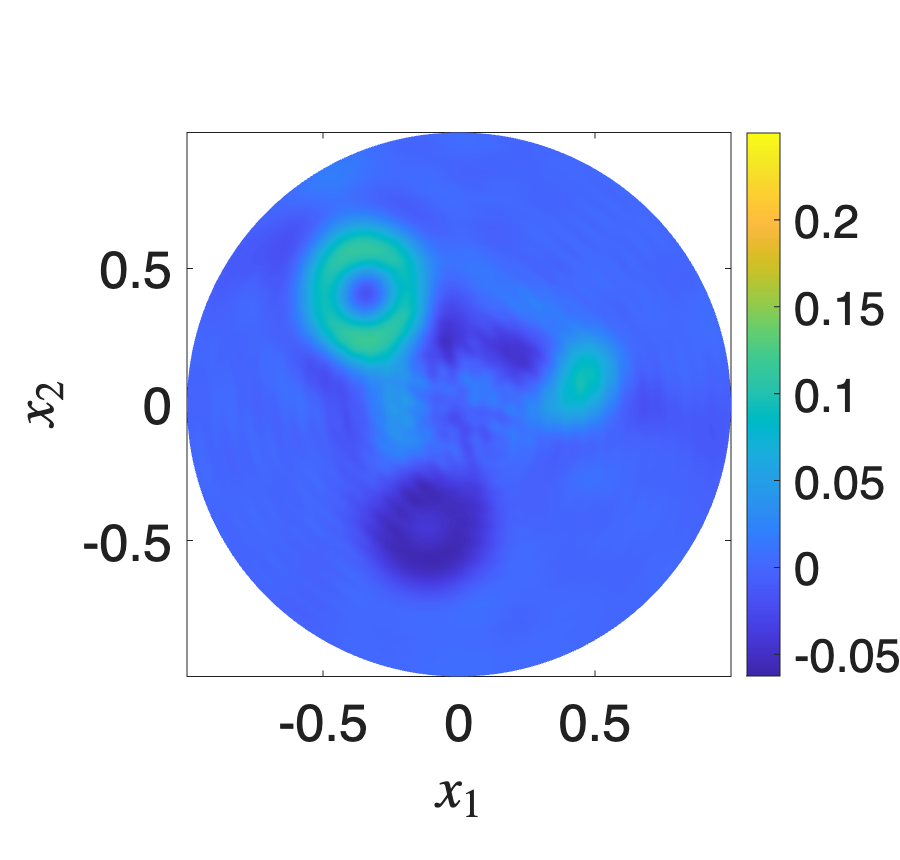}
    \end{subfigure}
    \caption{Example~\ref{exa:three_discs_smooth}. Reconstructed contrast from the exact Born far field data (left) and the exact full far field data (right). 
    Top: $\kappa=11$ (the best reconstruction from the full far field data, cf.~Figure~\ref{fig:reconstructions_discs_error_kappa}).
    Bottom: $\kappa=46$ (multiple scattering effects dominate linearized data).}
    \label{fig:reconstructions_discs_kappa}
\end{figure}

For $\kappa < 11$, accurate reconstructions cannot be obtained even from Born far field data, as $N$ is chosen so small that the amount of data is insufficient for reconstructing high spatial frequencies in the contrast $q$. For higher values of $\kappa$, enough information is available, with the relative error curve for Born far field data plateauing at around $7\%$ in Figure~\ref{fig:reconstructions_discs_kappa}.

According to~\cite{GriSch24}, the expansion coefficients of full far field data have the same essential support as those of Born far field data (cf.~\eqref{eq:truncation}).
Hence, it can be argued that the truncation of the infinite-dimensional systems at index $N$ does not remove retrievable information, which partially explains why structural information about the contrast geometry can be reconstructed reasonably even when multiple scattering effects are dominant. On the other hand, unlike in the case of artificially generated noise in the previous example, no noise filtering of the model error can be expected via truncation of the data matrix. 
\hfill$\lozenge$
\end{example}

\begin{example}
\label{ex:ship}
In this final example we compare our method with the low-rank method proposed in \cite{ZhoAudMenZha26} and MATLAB's built-in NUFFT as described in~\cite{DutRoh93}.

The method presented in~\cite{ZhoAudMenZha26} follows a similar strategy to ours:
based on the Born approximation, a simple invertible system matrix is derived by expanding both the contrast and the observed far field data in terms of suitably orthonormal systems for $L^2(B_1(\bfzero))$, namely the prolate spheroidal wave functions. Obtaining a diagonal system matrix, instead of a block-wise triangular one as in our case, comes at the cost of additional modeling error due to an unnatural sampling structure for the far field data.
On the positive side, owing to its duality, this method provides an explicit stability estimate because the (generalized) prolate spheroidal wave functions form an orthonormal basis for both a Fourier-type integral operator and a Sturm–Liouville differential operator.

When viewing \eqref{eq:BornFF} as an inverse Fourier transform of $q$, an approximation for $q$ can be obtained by completing these inverse Fourier data by zeros and subsequently applying the Fourier transform.
Taking into account that a transformation of the integration over $S^1×\times S^1$ to an integration over $B_{2\kappa}(\bfzero)$ is required, we can use NUFFT for approximating
\begin{equation*}
    \frac 1 {4L^2} \sum_{m,n=1}^{2L} \chi_{\Sigma}(\xhat_m,\idir_n)\uinfty_B(\xhat_m,\idir_n)\sin(\arg\idir_n-\arg\xhat_m)\,\rme^{2\pi\rmi\bfy\cdot\kappa(\idir_n-\xhat_m)/(2\pi)}\,,
\end{equation*}
where
\begin{equation*}
    \Sigma \,:=\, \big \{ (\xhat,\idir)\in\Sone\times\Sone
    \,:\,
    (\arg\idir-\arg\xhat)\mod{2\pi}\in(0,\pi) \big\}
\end{equation*}
ensures that no point is integrated multiple times.
We also refer to the state of the art implementation fiNUFFT by the Flatiron Institute (cf.~\cite{Bar21,BarMagKli19}), which clearly outperformed previous MATLAB's built-in implementations in terms of computational time for nonuniform sampling points and additionally allows for explicit error control.
For our purposes, however, the built-in implementation in MATLAB R2026a achieves comparable performance while offering easier operation. We therefore employed it in the presented examples.

To generate the data, we use the MATLAB file ``ship2D.m'' from the toolbox IPscatt (see \cite{BruKazLec19}), which was also used for the numerical tests in \cite[Fig.~12--14]{ZhoAudMenZha26}.
The real and imaginary parts of the exact contrast are shown in Figure~\ref{fig:ship_exact}, where we assume the prior information that the support of the contrast is contained in $B_1(\bfzero)$.
We simulate full far field data using the fast Lippmann--Schwinger solver~\cite{Vai97} and $2L=200$,
and subsequently add $20\%$ of noise to the data (as described in Example~\ref{exa:three_discs}).
We run the three methods for the wave numbers $\kappa=30$ and $\kappa=60$. This leads to the truncation indices $N=30$ and $N=60$ for our method, and in addition, we employ truncated SVD with the Morozov discrepancy principle as in Example~\ref{exa:three_discs_smooth}, assuming unrealistically the knowledge of the truncated Born far field data for choosing the spectral cut-off. 
For our method and the low-rank method, we set $N_r=N_\varphi=400$ for expanding the contrast.
For NUFFT, we use a cartesian grid enclosing the ROI and consisting of $100$ equidistant nodes in both dimensions.
As the spectral cut-off in the low rank method, we use $90\%$ of the prolate eigenvalue of largest magnitude as suggested in \cite{ZhoAudMenZha26}.

\begin{figure}[t]
    \centering
    \begin{subfigure}[b]{.36\textwidth}
      \centering
      \includegraphics[width=\textwidth]{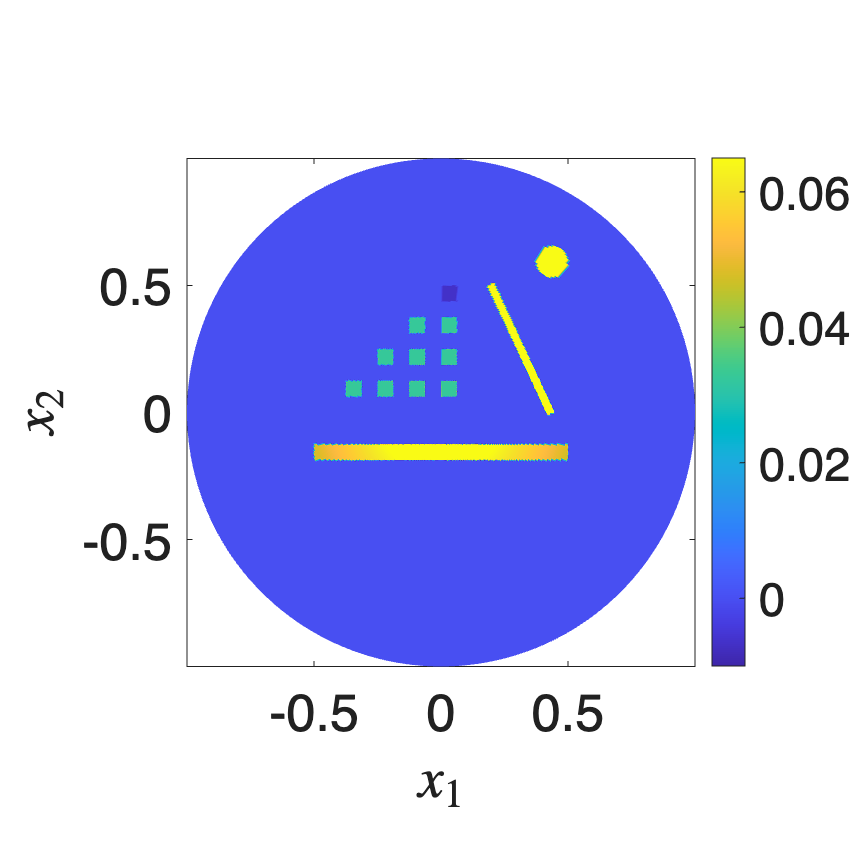}
    \end{subfigure}
    \hspace*{4em}
    \begin{subfigure}[b]{.36\textwidth}
      \centering
      \includegraphics[width=\textwidth]{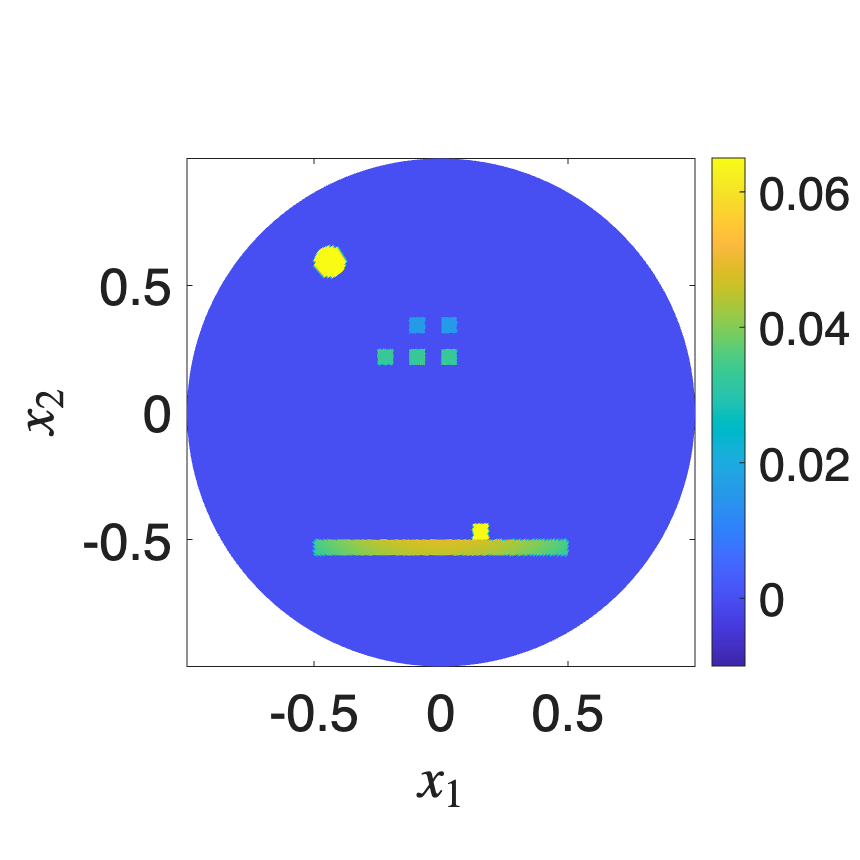}
    \end{subfigure}
    \caption{Example~\ref{ex:ship}. Real part (left) and imaginary part (right) of the exact contrast function.} 
    \label{fig:ship_exact}
  \end{figure}

  \begin{figure}[t]
    \centering
    \begin{subfigure}[b]{.36\textwidth}
      \centering
      \includegraphics[width=\textwidth]{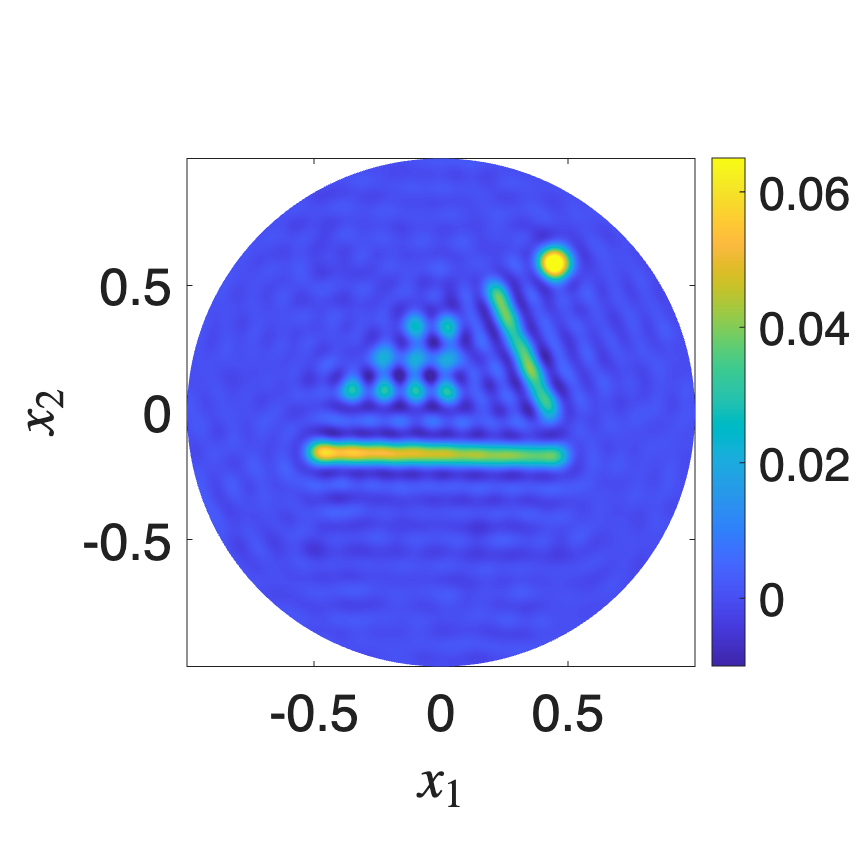}
    \end{subfigure}
    \hspace*{4em}
    \begin{subfigure}[b]{.36\textwidth}
      \centering
      \includegraphics[width=\textwidth]{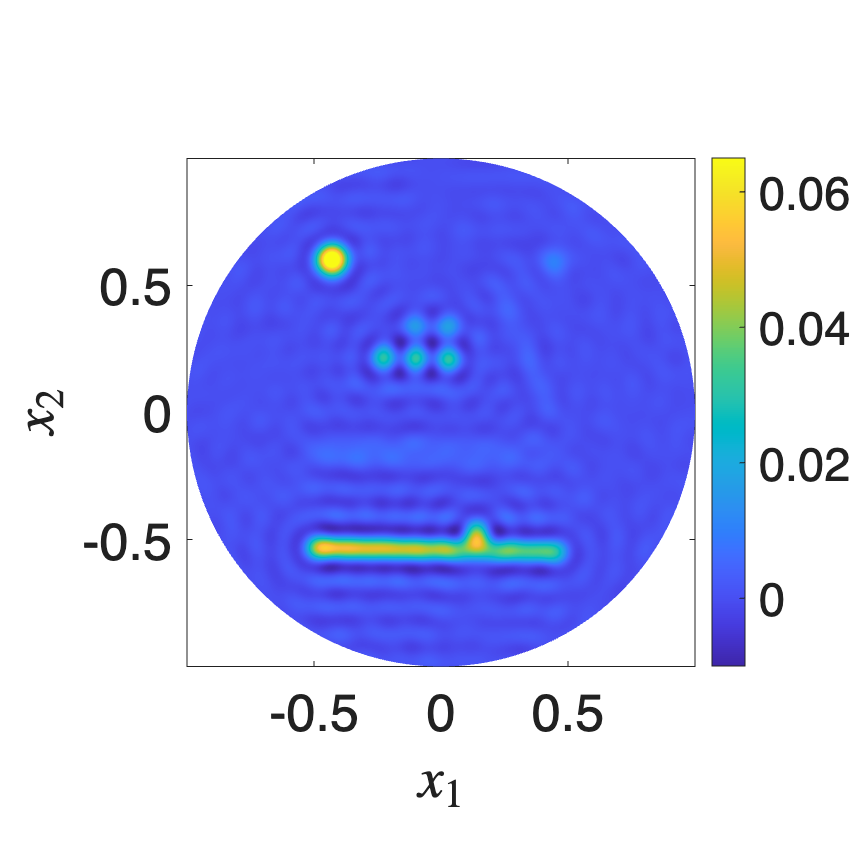}
    \end{subfigure}
    \begin{subfigure}[b]{.36\textwidth}
      \centering
      \includegraphics[width=\textwidth]{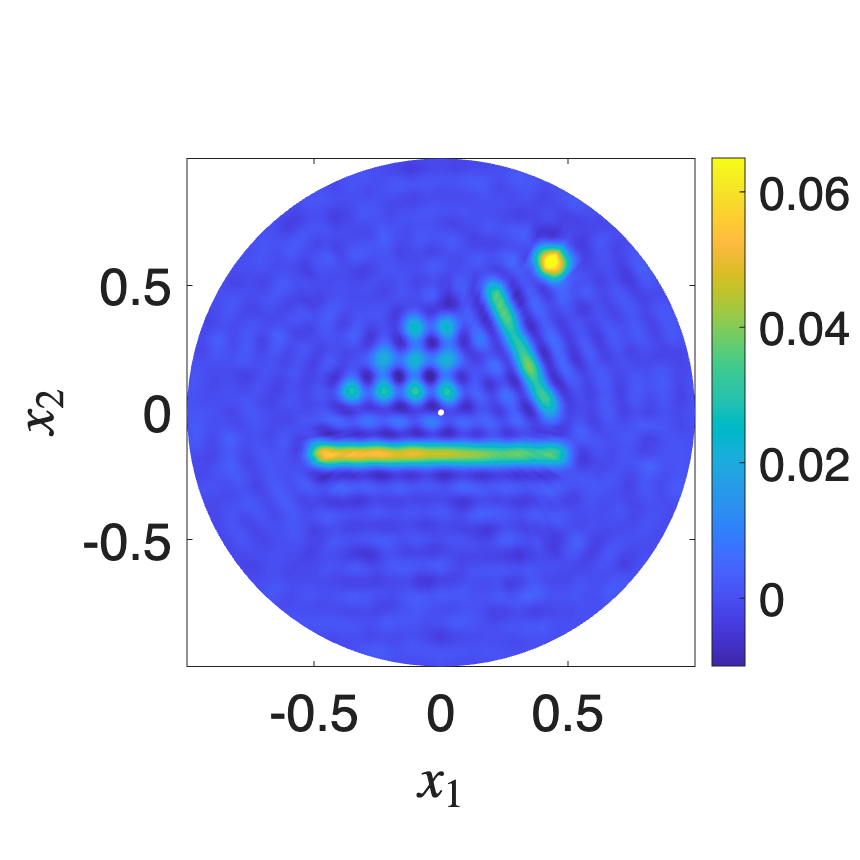}
    \end{subfigure}
    \hspace*{4em}
    \begin{subfigure}[b]{.36\textwidth}
      \centering
      \includegraphics[width=\textwidth]{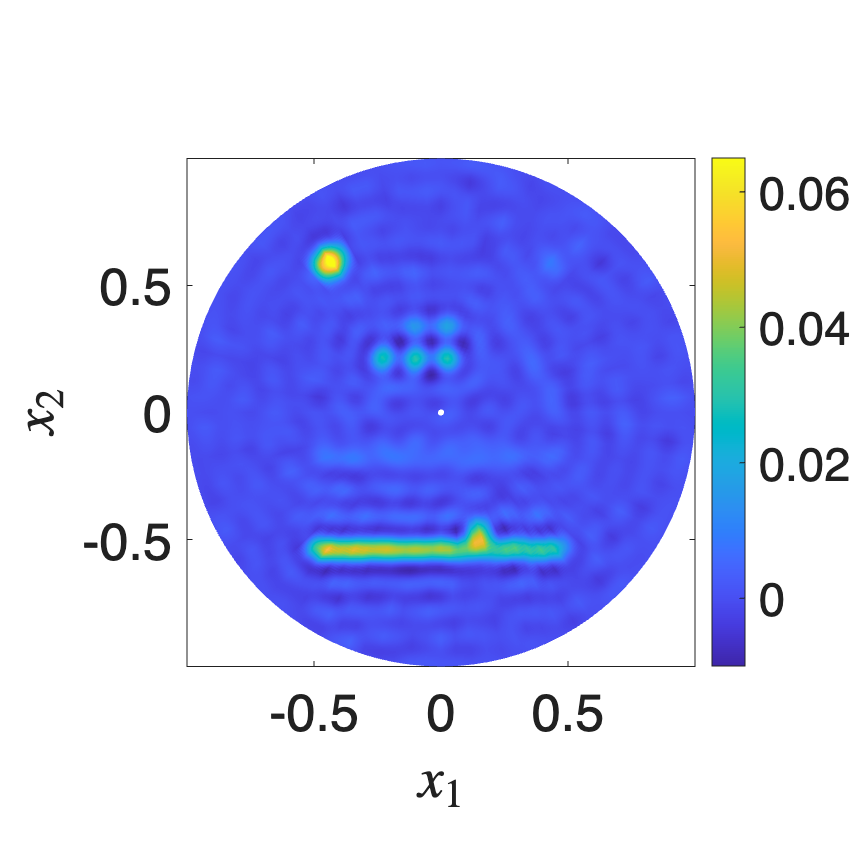}
    \end{subfigure}
    \begin{subfigure}[b]{.36\textwidth}
      \centering
      \includegraphics[width=\textwidth]{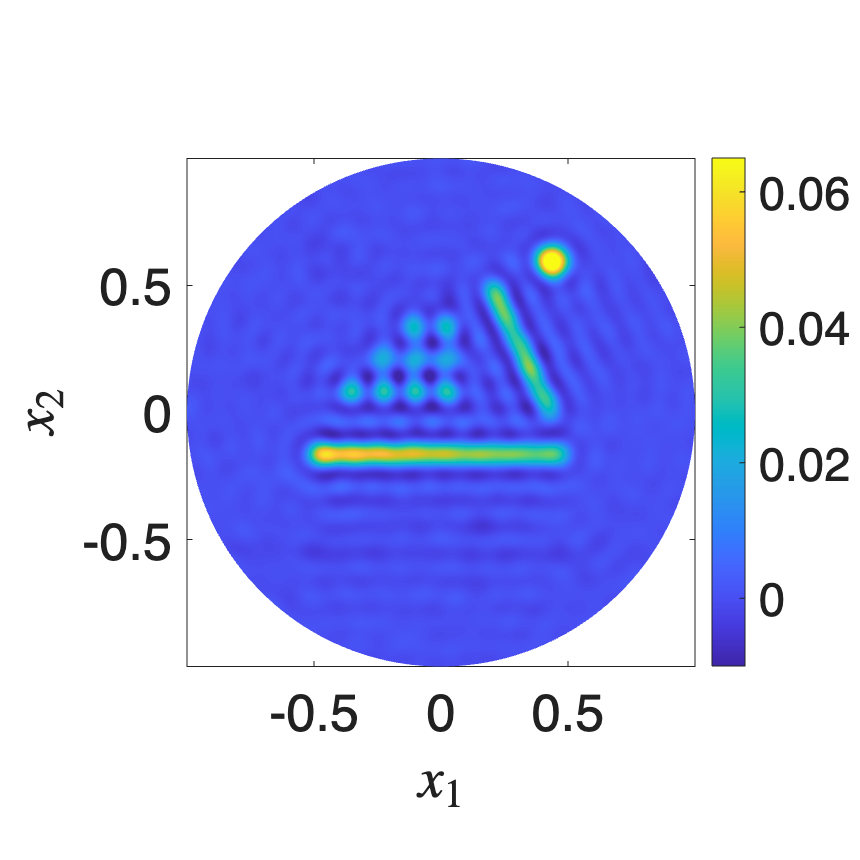}
    \end{subfigure}
    \hspace*{4em}
    \begin{subfigure}[b]{.36\textwidth}
      \centering
      \includegraphics[width=\textwidth]{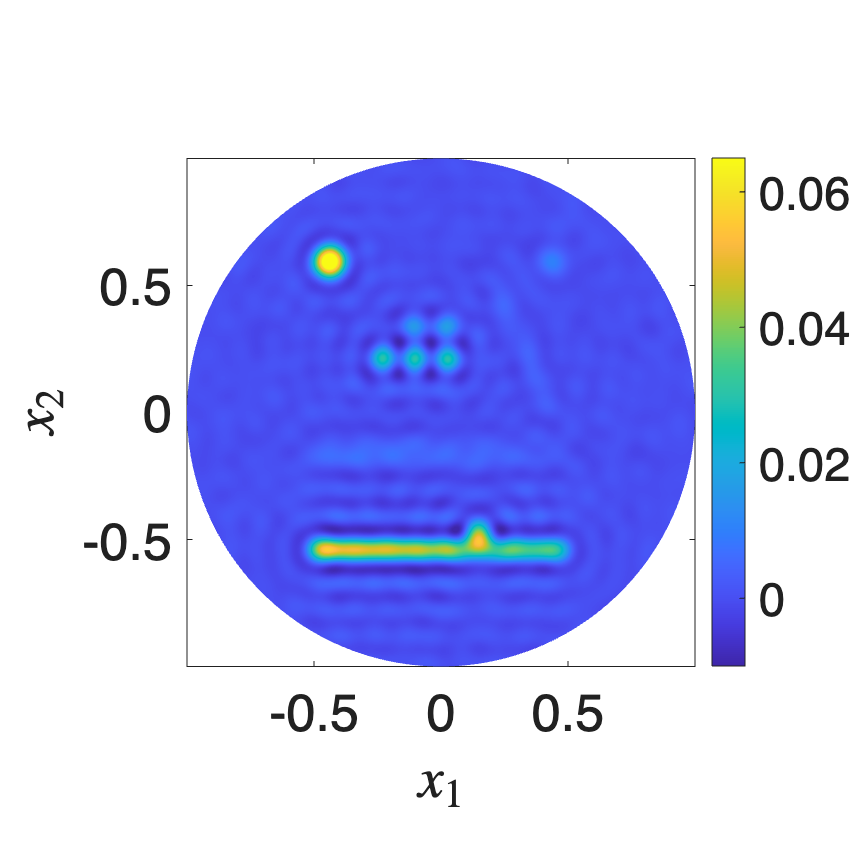}
    \end{subfigure}
    \caption{Example~\ref{ex:ship}. Real part (left) and imaginary part (right) of reconstructed contrast function for wave number $\kappa=30$ from full far field data with $20\%$ of additive noise. Top: our proposed method. Middle: the low-rank method~\cite{ZhoAudMenZha26}. Bottom: MATLAB's built-in NUFFT.} 
    \label{fig:ship_reconstructed_k30}
  \end{figure}
 \begin{figure}[t]
    \centering
    \begin{subfigure}[b]{.36\textwidth}
      \centering
      \includegraphics[width=\textwidth]{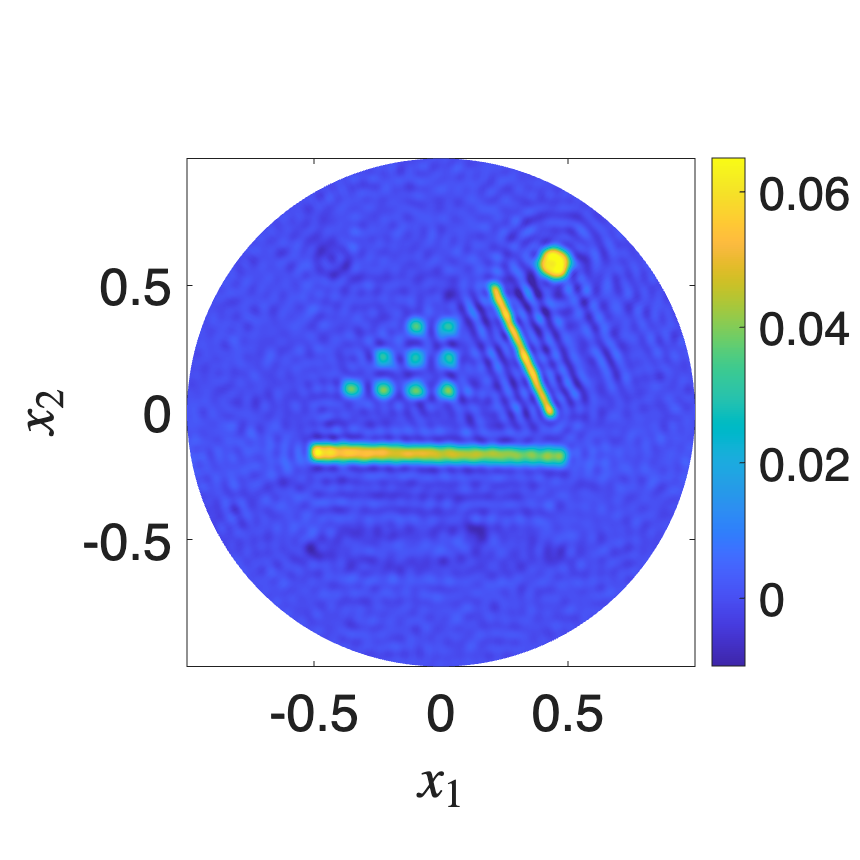}
    \end{subfigure}
    \hspace*{4em}
    \begin{subfigure}[b]{.36\textwidth}
      \centering
      \includegraphics[width=\textwidth]{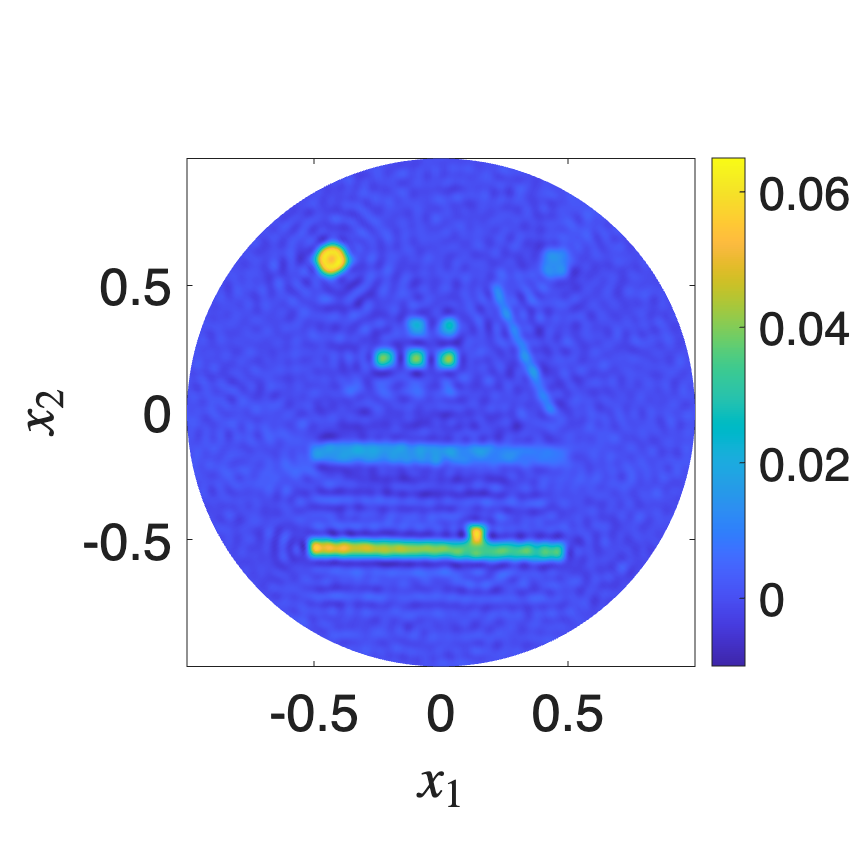}
    \end{subfigure}
    \begin{subfigure}[b]{.36\textwidth}
      \centering
      \includegraphics[width=\textwidth]{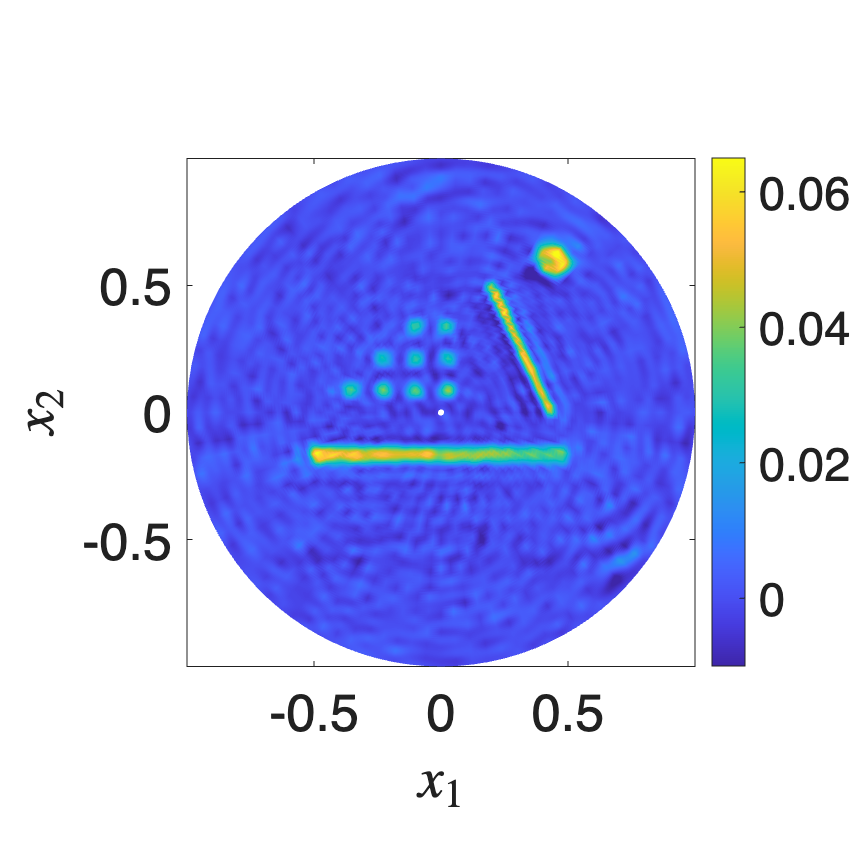}
    \end{subfigure}
    \hspace*{4em}
    \begin{subfigure}[b]{.36\textwidth}
      \centering
      \includegraphics[width=\textwidth]{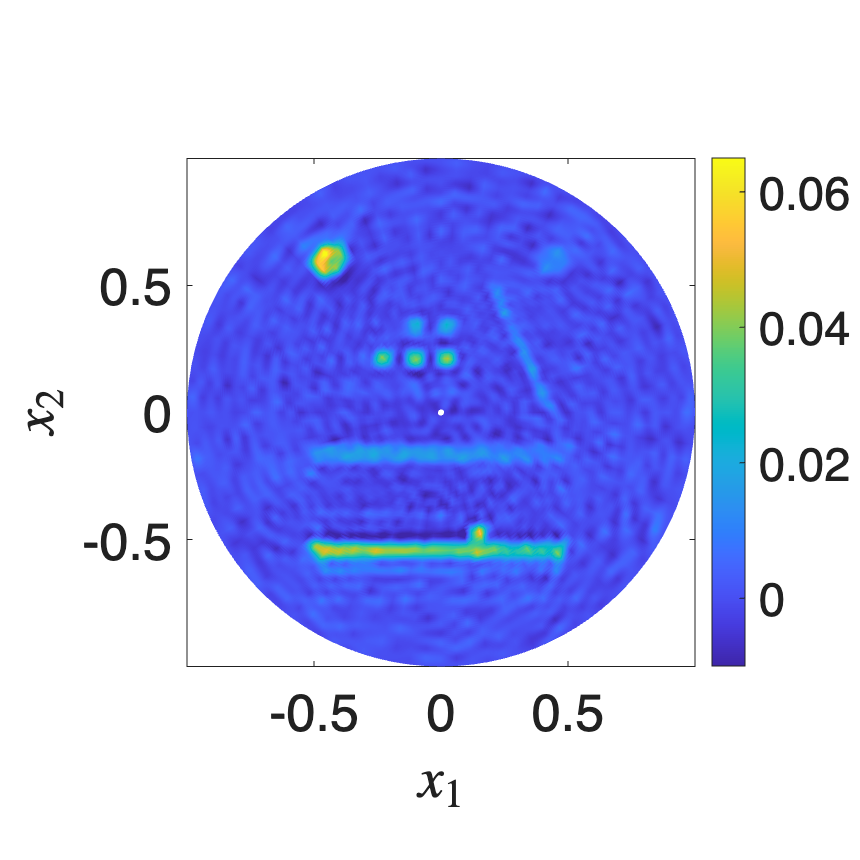}
    \end{subfigure}
    \begin{subfigure}[b]{.36\textwidth}
      \centering
      \includegraphics[width=\textwidth]{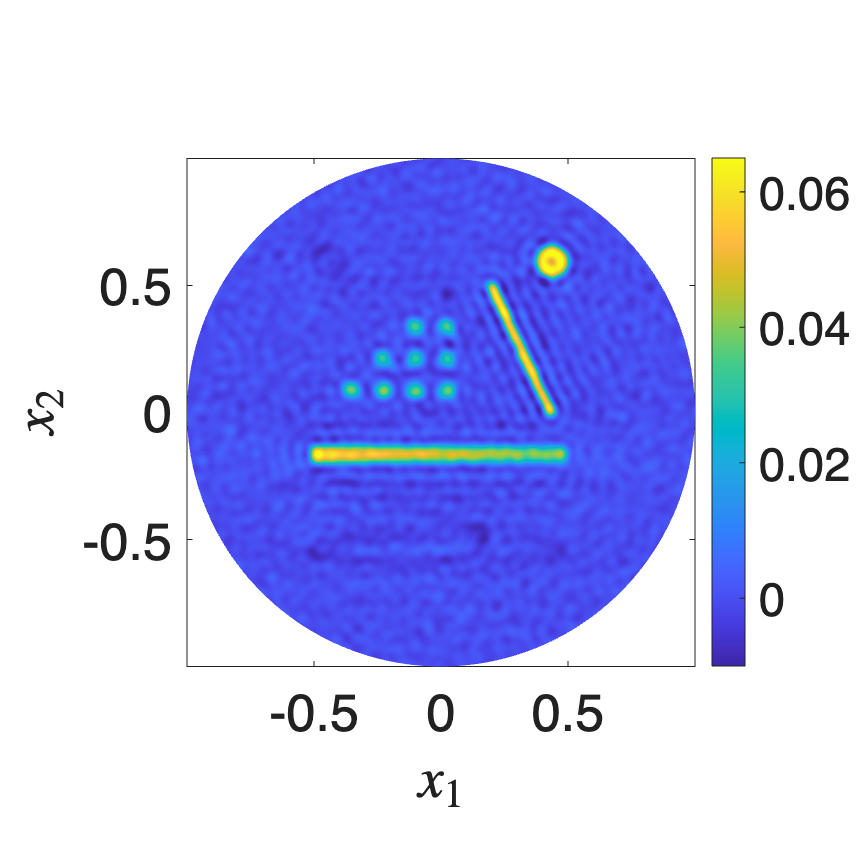}
    \end{subfigure}
    \hspace*{4em}
    \begin{subfigure}[b]{.36\textwidth}
      \centering
      \includegraphics[width=\textwidth]{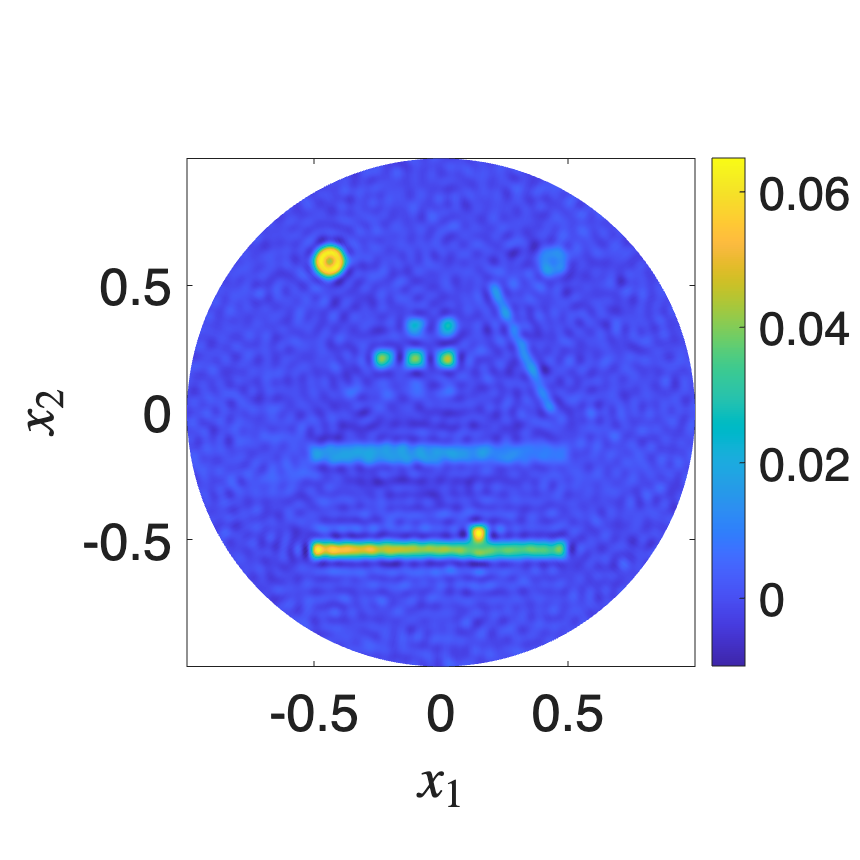}
    \end{subfigure}
    \caption{Example~\ref{ex:ship}. Real part (left) and imaginary part (right) of reconstructed contrast function for wave number $\kappa=60$ from full far field data with $20\%$ additive noise. Top: our proposed method. Middle: the low-rank method~\cite{ZhoAudMenZha26}. Bottom: MATLAB's built-in NUFFT.
    } 
    \label{fig:ship_reconstructed_k60}
  \end{figure}
  The reconstructions for $\kappa=30$ are shown in Figure~\ref{fig:ship_reconstructed_k30} and those for $\kappa=60$ in Figure~\ref{fig:ship_reconstructed_k60}.
All three methods produce reconstructions of comparable quality.
Based on this single example, the performance of our method seems to be on par with other direct reconstruction algorithms that are based on the Born approximation, but drawing more precise conclusions is not possible without more detailed testing.
\hfill$\lozenge$
\end{example}

\section*{Conclusions}
Following the ideas in \cite{AutGarHirHvy24,GarHyv24} for EIT, we introduced a direct reconstruction method for inverse medium Born scattering for the Helmholtz equation in two spatial dimensions. Choosing appropriate basis for representing the far field operator and the contrast function, the proposed method reduces the inverse problem to solving decoupled triangular systems that correspond to different angular frequencies in the contrast. The bases for representing the far field operator and the angular behavior of the contrast function are of Fourier type and can be given explicitly, but introducing the needed radial bases for the contrast requires numerically orthogonalizing certain products of Bessel functions, which adds an additional unstable step to the algorithm. On the positive side, this orthogonalization can be performed offline before the data is available and can also be stabilized by truncating the to-be-inverted system based on ideas in \cite{GriSch24}. Due to the achieved angular decoupling and the triangular structure of the subsystems, the proposed method allows an efficient numerical implementation, as well as an explicit recursive reconstruction formula (Theorem~\ref{Theorem:main}) if instability issues are not considered.

The presented numerical experiments demonstrate that our method produces reconstruction of good quality from (noisy) Born far field data, and it can also be applied to full nonlinear far field data to obtain reconstructions ranging in quality from good to reasonable depending on the extent of multiple scattering effects. According to our numerical experiments, the proposed method compares favorably to other algorithms designed for solving the inverse medium Born scattering problem, in particular, producing reconstructions comparable to those by the low-rank method introduced recently
in~\cite{ZhoAudMenZha26}. 

A natural direction for future work is to derive explicit representations for the radial basis functions; this would also be a first step toward extending the stability results of~\cite{GarHyv24} from EIT to inverse medium scattering. Another promising avenue is to generalize our reconstruction method to three dimensions, building on the ideas in~\cite{Garde25}.


\section*{Acknowledgments}
We would like to thank one of the anonymous referees for spotting an error in the original version of the proof of Proposition 3.2.

This work was supported by the Research Council of Finland (Flagship of Advanced Mathematics for Sensing Imaging and Modelling grant 359181).

{\small
  \bibliographystyle{abbrvurl}
  \bibliography{inverse_Born_scattering}
}

\end{document}